\documentclass[1p,sort&compress,number,preprint]{elsarticle}
\usepackage[english]{babel} 
\usepackage[final]{hyperref}
\usepackage{amsmath}
\usepackage{amsthm}
\usepackage{amssymb,stmaryrd}
\usepackage{subcaption}
\usepackage{etoolbox}
\usepackage{xcolor}
\usepackage{cancel}
\usepackage{graphicx,overpic,rotating}
\usepackage{comment}
\usepackage{bm}
\definecolor{verde}{rgb}{0,0.5,0}

\newcommand{\bn}{\mathbf{n}}
\newcommand{\bx}{\mathbf{x}}
\newcommand{\by}{\mathbf{y}}

\newcommand{\jump}[1]{\llbracket #1 \rrbracket}
\newcommand{\Sh}{\mathcal{E}_h}
\newcommand{\dxi}{d\xi}
\newcommand{\dx}{d\mathbf{x}}

\newcommand{\EMPTY}{\,\cdot\,}

\newcommand{\EOD}{\end{document}}

\newcommand{\abs}[1]{\left\lvert #1\right\rvert}
\newcommand{\seminorm}[2][]{\left\lvert{#2}\right\rvert_{#1}}
\newcommand{\norm}[2][]{\left\lVert{#2}\right\rVert_{{#1}}}

\newcommand{\ennorm}[2][]{\left\lvert\!\left\lvert\!\left\lvert{#2}\right\rvert\!\right\rvert\!\right\rvert_{#1}}

\newcommand{\scal}[3][]{\left(#2,#3\right)_{\! #1}}
\newcommand{\jmp}[1]{\left\llbracket {#1} \right\rrbracket}
\newcommand{\dtot}[3][1]{\frac{\mathrm{d}\ifstrequal{#1}{1}{}{^{#1}}#2}{\mathrm{d}#3\ifstrequal{#1}{1}{}{^{#1}}}}
\newcommand{\pder}[3][1]{\frac{ \partial \ifstrequal{#1}{1}{}{^{#1}} #2 }
  {\partial #3 \ifstrequal{#1}{1}{}{^{#1}} }}
\newcommand{\Dtot}[3]{\frac{\mathrm{D}\ifstrequal{#1}{1}{}{^{#1}}#2}%
  {\mathrm{D} #3\ifstrequal{#1}{1}{}{^{#1}}}}

\newcommand{\defeq}{:=}
\renewcommand{\a}[3][]{\ifstrempty{#2#3}{a_{#1}}{a_{#1}\left(#2,#3\right)}} % forma bilineare per scrivere problemi ellittici
\renewcommand{\b}[3][]{\ifstrempty{#2#3}{b_{#1}}{b_{#1}\left(#2,#3\right)}}
\renewcommand{\c}[3][]{\ifstrempty{#2#3}{c_{#1}}{c_{#1}\left(#2,#3\right)}}
\renewcommand{\d}[3][]{\ifstrempty{#2#3}{d_{#1}}{d_{#1}\left(#2,#3\right)}}
\newcommand{\aE}[3][E]{\ifstrempty{#2#3}{a^{#1}}{a^{#1}\left(#2,#3\right)}}
\newcommand{\ahE}[3][E]{\ifstrempty{#2#3}{a_h^{#1}}{a_h^{#1}\left(#2,#3\right)}}

\newcommand{\vemstab}[3][]{S^{#1}\!\ifstrempty{#2#3}{}{\left(#2,#3\right)}}
\newcommand{\vemstabAST}[3][]{S^{#1}_\ast\!\ifstrempty{#2#3}{}{\left(#2,#3\right)}}
\newcommand{\shE}[2]{\vemstab[E]{#1}{#2}}
\renewcommand{\div}{\nabla\!\cdot}
\newcommand{\proj}[3][0]{\Pi^{#1}_{#2} \ifstrempty{#3}{}{\!\left(#3\right)}}
\newcommand{\sobh}[3][]{\mathrm{H}^{#2}_{#1}\ifstrempty{#3}{}{\left(#3\right)}}
\newcommand{\sobho}[3][1]{\sobh[0\ifstrempty{#3}{}{,#3}]{#1}{#2}} % spazi di Sobolev con lo zero
\newcommand{\lebl}[2][2]{\mathrm{L}^{#1}\ifstrempty{#2}{}{\left(#2\right)}} % spazi di Lebesgue
\newcommand{\sob}[3]{\mathrm{W}^{#1}_{#2}\!\left(#3\right)}
\newcommand{\sobnc}[2]{\mathrm{H}^{1,nc}_{#1}\left(#2\right)}
\newcommand{\monom}[2]{\mathcal{M}_{#1}\left(#2\right)}
\newcommand{\n}[1]{\hat{n}_{#1}} % normale a un segmenti
\newcommand{\Th}[1][]{\mathcal{T}_{h\ifstrempty{#1}{}{,#1}}}
\newcommand{\Eh}{\mathcal{E}_h}
\newcommand{\Poly}[2]{\mathbb{P}_{#1}\!\left(#2\right)} % spazi di polinomi
\newcommand{\K}[1][]{\mathsf{K}_{#1}}
\newcommand{\Kmin}[1][]{\mathsf{K}^{\vee}_{#1}}

\newcommand{\eval}[2]{\left.{#1}\right|_{#2}}

\newcommand{\B}[3][]{\ifstrempty{#2#3}{B^{#1}}{B^{#1}\left(#2,#3\right)}}

\newcommand{\F}[2][]{\ifstrempty{#2}{F^{#1}}{F^{#1}\left(#2\right)}}

\newcommand{\Bsupg}[3][]{B_{\mathrm{supg}\ifstrempty{#1}{}{,#1}}\ifstrempty{#2#3}{}{\left(#2,#3\right)}}
\newcommand{\Bhsupg}[3][]{\ifstrempty{#2#3}{B_{\mathrm{supg},h}^{#1}}{B_{\mathrm{supg},h}^{#1}\left(#2,#3\right)}}
\newcommand{\Fsupg}[2][]{F_{\mathrm{supg}\ifstrempty{#1}{}{,#1}}\ifstrempty{#2}{}{\left(#2\right)}}
\newcommand{\Nh}[2]{\mathcal{N}_h\ifstrempty{#1#2}{}{\left(#1,#2\right)}}

\newcommand{\Pe}[1][]{\mathrm{Pe}_{#1}} % numero di Peclet
\newtheorem{theorem}{Theorem}

\newtheorem{lemma}{Lemma}

\theoremstyle{remark}

\newtheorem{remark}{Remark}
\newtheorem{assumption}{Assumption}
\theoremstyle{definition}

\title{SUPG stabilization for the nonconforming virtual element method
  for advection-diffusion-reaction equations\tnoteref{miur}}

\tnotetext[miur]{
  G. Manzini has been partially supported by the Laboratory Directed
  Research and Development program (LDRD), U.S. Department of Energy
  Office of Science, Office of Fusion Energy Sciences, under the
  auspices of the National Nuclear Security Administration of the
  U.S. Department of Energy by Los Alamos National Laboratory,
  operated by Los Alamos National Security LLC under contract
  DE-AC52-06NA25396. 
  This work is assigned the LA-UR number LA-UR-17-20866.
  The three authors are also members of the Gruppo Nazionale Calcolo
  Scientifico (GNCS) at Istituto Nazionale di Alta Matematica
  (INdAM).
  S. Berrone and A. Borio kindly acknowledge partial financial 
  support by INdAM-GNCS Projects 2018, HPC@Polito and by the MIUR project
  ``Dipartimenti di Eccellenza 2018-2022''.
  }

\author[poli]{S. Berrone}
\ead{stefano.berrone@polito.it}

\author[poli]{A. Borio}
\ead{andrea.borio@polito.it}

\author[lanl,imati]{G. Manzini}
\ead{gmanzini@lanl.com}

\address[poli]{Dipartimento di Scienze Matematiche, Politecnico di
  Torino\\Corso Duca degli Abruzzi 24, Torino, 10129, Italy}

\address[lanl]{T-5 Group, Theoretical Division, Los Alamos National
  Laboratory,\\ Los Alamos, New Mexico, USA} 

\address[imati]{Istituto di
  Matematica Applicata e Tecnologie Informatiche - CNR, \\via Ferrata 1,
  Pavia 27100, Italy} %\cortext[cor1]{Corresponding author}
 
\begin{document}
\begin{abstract}
  We present the design, convergence analysis and numerical
  investigations of the nonconforming virtual element
  method with Streamline Upwind/Petrov-Galerkin (VEM-SUPG)
  stabilization for the numerical resolution of
  convection-diffusion-reaction problems in the convective-dominated
  regime.

  According to the virtual discretization approach, the bilinear form
  is split as the sum of a consistency and a stability term.
  The consistency term is given by substituting the functions of the
  virtual space and their gradients with their polynomial projection
  in each term of the bilinear form (including the SUPG stabilization
  term).
  Polynomial projections can be computed exactly from the degrees of
  freedom.
  The stability term is also built from the degrees of freedom by
  ensuring the correct scalability properties with respect to the mesh
  size and the equation coefficients.

  The nonconforming formulation relaxes the continuity conditions at
  cell interfaces and a weaker regularity condition is considered
  involving polynomial moments of the solution jumps at cell
  interface.
  Optimal convergence properties of the method are proved in a
  suitable norm, which includes contribution from the
  advective stabilization terms.
  Experimental results confirm the theoretical convergence rates.
\end{abstract}
\maketitle

% mainfile: vem-nc_supg

\newcommand{\restrict}[2]{{#1}_{|#2}}
\newcommand{\hE}{h_E}
\newcommand{\TERM}[2]{\mathsf{#1}_{#2}}
\newcommand{\FNOTE}[1]{\footnote{\textbf{#1}}}
\newcommand{\CPF}{\mathcal{C}_{P\!F}}
\newcommand{\CNST}[1]{\mathcal{C}_{\mbox{\begin{footnotesize}#1\end{footnotesize}}}\,}

\section{Introduction}
\label{sec:intro}

The virtual element method (VEM) was proposed
in~\cite{BeiraodaVeiga-Brezzi-Cangiani-Manzini-Marini-Russo:2013} as a
variational reformulation of the \emph{nodal} mimetic finite
difference (MFD) method~\cite{%
  BeiraodaVeiga-Manzini-Putti:2015,%
  BeiraodaVeiga-Lipnikov-Manzini:2011,%
  Brezzi-Buffa-Lipnikov:2009,%
  Manzini-Lipnikov-Moulton-Shashkov:2017%
} for solving diffusion problems on unstructured polygonal meshes.
A survey on the MFD method can be found in the review
paper~\cite{Lipnikov-Manzini-Shashkov:2014} and the research
book~\cite{BeiraodaVeiga-Lipnikov-Manzini:2014}.
The VEM inherits the great flexibility of the MFD method with respect
to the admissible meshes, and, despite its introduction dates back to
a few years ago a huge amount of development has taken place, see, for
example, ~\cite{%
  Antonietti-BeiraodaVeiga-Scacchi-Verani:2016,%
  Cangiani-Georgoulis-Pryer-Sutton:2016,%
  % BeiraodaVeiga-Lovadina-Vacca:2016,%
  BeiraodaVeiga-Chernov-Mascotto-Russo:2016,%
  BeiraodaVeiga-Brezzi-Marini-Russo:2016a,%
  BeiraodaVeiga-Brezzi-Marini-Russo:2016b,%
  BeiraodaVeiga-Brezzi-Marini-Russo:2016c,%
  BeiraodaVeiga-Brezzi-Marini-Russo:2016d,%
  Benedetto-Berrone-Borio-Pieraccini-Scialo:2016b,%
  Berrone-Borio-Scialo:2016,%
  Benedetto-Berrone-Scialo:2016,%
  Berrone-Pieraccini-Scialo:2016,%
  % Perugia-Pietra-Russo:2016,%
  Wriggers-Rust-Reddy:2016,%
  % BeiraodaVeiga-Lovadina-Mora:2015,%
  BeiraodaVeiga-Manzini:2015,%
  % Mora-Rivera-Rodriguez:2015,%
  Natarajan-Bordas-Ooi:2015,%
  Berrone-Pieraccini-Scialo-Vicini:2015,%
  Vacca-BeiraodaVeiga:2015,%
  % Paulino-Gain:2015,%
  % Antonietti-BeiraodaVeiga-Mora-Verani:2014,%
  BeiraodaVeiga-Brezzi-Marini-Russo:2014,%
  BeiraodaVeiga-Brezzi-Marini-Russo:2014b,%
  BeiraodaVeiga-Manzini:2014,%
  Benedetto-Berrone-Pieraccini-Scialo:2014,%
  % BeiraodaVeiga-Brezzi-Marini:2013,%
  Brezzi-Marini:2013,%
  Berrone-Benedetto-Borio:2016chapter,%
  Berrone-Borio:2017,%
  Benedetto-Berrone-Borio-Pieraccini-Scialo:2016eccomas%
}.
We emphasize that the VEM is not the only existing way to treat
partial differential equations numerically on unstructured meshes.
Other methods or families of methods that are available from the
literature include %%
the polygonal/polyhedral finite element method
(PFEM)~\cite{Wachspress:1975,Sukumar-Tabarraei:2004}, %%
the BEM-based FEM~\cite{Hofreither-Langer-Weisser:2016,Weisser:2011},
%% FEM~\cite{Hofreither-Langer-Weisser:2016,Rjasanow-Weisser:2012,Weisser:2011},
the finite volume
methods~\cite{Droniou:2014,Droniou-Eymard-Gallouet-Herbin:2010}, %%
hybrid high-order (HHO) method~\cite{DiPietro-Ern:2014}, %%
the discontinuous Galerkin (DG)
method~\cite{DiPietro-Ern:2011,Cangiani-Georgoulis-Houston:2014}, %%
and the hybridized discontinuous Galerkin (HDG)
method~\cite{Cockburn-Gopalakrishnan-Lazarov:2009}. %%, %%
Many of these methods are also part of the Gradient Scheme framework
recently proposed
by~\cite{Droniou-Eymard-Herbin:2016,Droniou-Eymard-Herbin:2013}.
Moreover, the connection between VEM and finite elements on
polygonal/polyhedral meshes is touroughly investigated
in~\cite{Manzini-Russo-Sukumar:2014,Cangiani-Manzini-Russo-Sukumar:2015,DiPietro-Droniou-Manzini:2018},
between VEM and BEM-based FEM method
in~\cite{Cangiani-Gyrya-Manzini-Sutton:2017:GBC:chbook}.
% , and between VEM and the wG method in~\cite{Chen:2015}.

The virtual element method is a finite element method, but is dubbed
\emph{virtual} because its formulation does not require the explicit
knowledge of a set of shape functions and gradients of shape functions
to compute the bilinear forms, e.g., mass and stiffness matrices.
The global approximation space is defined over the whole domain by
gluing together local elemental spaces under some regularity
constraint.
Each elemental space is formed by the solutions of a local Poisson
problem with a polynomial right-hand side and nonhomogeneous
polynomial Dirichlet (or Neumann) boundary conditions.
Clearly, a subspace of polynomials up to a given degree always belongs
by construction to each elemental space.
The remarkable fact is that we can compute \emph{exactly} the
projections of the virtual functions and their first derivatives onto
such polynomials by using only the degrees of freedom.
Therefore, a straightforward strategy to approximate the bilinear
forms is to substitute the shape functions and their derivatives in
their arguments with their polynomial projections.
This approach yields the so-called \emph{consistency term}, to which
we add a \emph{stability term} that ensures the nonsingularity of the
resulting discretization.
The stability term is designed to be easily computable from the
degrees of freedom.

The VEM was originally formulated
in~\cite{BeiraodaVeiga-Brezzi-Cangiani-Manzini-Marini-Russo:2013} as a
conforming FEM for the Poisson problem.
It was later extended to convection-reaction-diffusion problems with
variable coefficients
in~\cite{Ahmad-Alsaedi-Brezzi-Marini-Russo:2013,BeiraodaVeiga-Brezzi-Marini-Russo:2016b}.
Meanwhile, the nonconforming formulation for diffusion problems was
proposed in~\cite{AyusodeDios-Lipnikov-Manzini:2016} as the finite
element reformulation of~\cite{Lipnikov-Manzini:2014} and later
extended to general elliptic
problems~\cite{Cangiani-Manzini-Sutton:2016}, Stokes
problem~\cite{Cangiani-Gyrya-Manzini:2016}, and the biharmonic
equation~\cite{Antonietti-Manzini-Verani:2018,Zhao-Chen-Zhang:2016}.
The two major differences between the conforming and nonconforming
formulations are:
\begin{description}
\item $(i)$ at the elemental level the virtual space is formed by the
  solution of a Poisson problem with Neumann boundary conditions;
\item $(ii)$ at the global level we relax the interelement conformity
  requirement, and the definition of the global discrete space just
  relays on some form of weaker regularity according
  to~\cite{Crouzeix-Raviart:1973}.
\end{description}
Nonconforming finite element spaces were historically proposed to
approximate the velocity field of the Stokes equations on triangular
meshes~\cite{Crouzeix-Raviart:1973}.
The functions in these finite element spaces are piecewise polynomials
of degree $k=1$~\cite{Crouzeix-Raviart:1973},
$k=2$~\cite{Fortin-Soulie:1983}, $k=3$~\cite{Crouzeix-Falk:1989}, and
$k>3$~\cite{Stoyan-Baran:2006,Baran-Stoyan:2007,Ainsworth:2005}.
In such formulations, continuity is required only at a discrete set of
special points located at cell interfaces, which are the roots of the
one-dimensional $k^{th}$-order Legendre polynomials defined over each
edge, i.e., the nodes of the Gauss-Legendre quadrature rule of order
$k$.
This minimal continuity requirement ensures the optimal convergence
rate; see, for instance,~\cite{Crouzeix-Raviart:1973}.
Attempts to extend non-conforming finite elements to quadrilaterals,
tetrahedra and hexahedra are found
in~\cite{Rannacher-Turek:1992,Matthies-Tobiska:2005,Matthies:2007}.
%% in~\cite{Rannacher-Turek:1992,Cai-Douglas-Ye:1999,Matthies-Tobiska:2005,Matthies:2007}.
%%
A major issue of the nonconforming formulations is that they may
strongly depend on the parity of the underlying polynomial space, the
geometric shape of the element and its spatial dimensionality (2-D or
3-D).
For example, on triangles the nonconforming finite element space for
even $k\geq 2$
in~\cite{Fortin-Soulie:1983,Stoyan-Baran:2006,Baran-Stoyan:2007} must
be enriched by a one-dimensional subspace generated by a bubble
function.
Also, the definition of nonconforming spaces is substantially
different from 2D and 3D and requires a simple geometric shape for the
element (e.g., a simplex, a quadrilateral, or an hexahedral cell), and
also differs from 2D to 3D.
Instead, the nonconforming virtual element space proposed
in~\cite{AyusodeDios-Lipnikov-Manzini:2016} has the same construction
for every $k$ regardless of the parity, the space dimension, and the
elemental geometric shape.

In the case of the convection-dominated regime, a stabilization must
be included in the variational formulation to deal with high
P\'{e}clet number situations.
In finite element approximations, different strategies have been
designed to such purpose, as, for example, %%
local projections~\cite{Ganesan-Tobiska:2010}, %%
bubble functions~\cite{Brezzi-Franca-Russo:1998,Franca-Tobiska:2002} %
the SUPG method~\cite{%
  Brooks-Hughes:1982,%
  % Johnson-Navert-Pitkaranta:1984,%
  Franca-Frey-Hughes:1992,%
  Gelhard-Lube-Olshanskii-Starcke:2005,%
  Roos-Stynes-Tobiska:2008,%
  Burman-Smith:2011,%
  Burman:2010%
}.
The SUPG stabilization in the conforming virtual element formulation
was previously considered
in~\cite{Benedetto-Berrone-Borio-Pieraccini-Scialo:2016a}.
The main goal of this work is the development of the nonconforming
formulation with SUPG stabilization suitable to solve
convection-dominated transport problems with a moderate reaction term.
In such a situation the SUPG stabilization parameters becomes
dependent on a local P\'{e}clet number and a local Karlovitz-like
number.
In case of a vanishing reaction the SUPG stabilization parameter
converges to its expected classical definition. We prove the
robustness of the method with respect to high P\'eclet numbers when
the problem coefficients are constants and a conforming formulation is
considered, whereas a weak dependence on the P\'eclet number is
observed, due to the non-consistency of the VEM bilinear form
\cite{BeiraodaVeiga-Brezzi-Cangiani-Manzini-Marini-Russo:2013} when
the coefficients are variable or a non conforming formulation is
considered. The presented analysis is also valid for SUPG-stabilized
conforming virtual elements as presented in
\cite{Benedetto-Berrone-Borio-Pieraccini-Scialo:2016a}.

\medskip
The outline of the paper is as follows.
In Section~\ref{sec:problem} we introduce the mathematical model of
the convection-reaction-diffusion problem.
In Section~\ref{sec:vem} we present the nonconforming VEM with the
SUPG stabilization for the convection-dominated regime.
In Section~\ref{sec:estimate} we carry out the convergence analysis
and derive optimal a priori error estimates.
In Section~\ref{sec:numerics} we show the performance of the method on
a set of representative problems.
In Section~\ref{sec:conclusions} we offer our final remarks and
conclusions.

\subsection{Notation}
The notation throughout the paper is as follows:
$\scal{\cdot}{\cdot}$ and $\norm{\EMPTY}$ denote the $\lebl{\Omega}$
scalar product and norm, and
$\scal[\omega]{\cdot}{\cdot}$ and $\norm[\omega]{\EMPTY}$ denote the
$\lebl{\omega}$ scalar product and norm defined on the subdomain
$\omega\subseteq\Omega$;
$\norm[\alpha]{\EMPTY}$ and $\seminorm[\alpha]{\EMPTY}$ denote the
$\sobh{\alpha}{\Omega}$ norm and semi-norm;
$\norm[\alpha,\omega]{\EMPTY}$ and $\seminorm[\alpha,\omega]{\EMPTY}$
denote the $\sobh{\alpha}{\omega}$ norm and semi-norm;
$\norm[\sob{q}{p}{\omega}]{\EMPTY}$ and
$\seminorm[\sob{q}{p}{\omega}]{\EMPTY}$ denote the
$\sob{q}{p}{\omega}$ norm and semi-norm, where $p\geq 1$ is the
Lebesgue regularity index and $q$ is the order of the Sobolev space.
Moreover, $\Poly{k}{\omega}$ denotes the space of polynomial functions
of degree up to the integer number $k\geq 0$ that are defined on the
$d$-dimensional subset $\omega\subseteq\Omega$ with $d=1,2,3$.
If $\Th$ is a partitioning of $\Omega$ in a set of non-overlapping
polytopal elements $E$, i.e., the \emph{mesh}, (for the formal
definition see Section~\ref{sec:mesh}), by $\Poly{k}{\Th}$ we denote
the space of discontinuous functions defined on $\Omega$ whose
restriction to any element $E$ is a polynomial of degree less than or
equal to $k$; hence, $p\in\Poly{k}{\Th}$ iff $p_{|E}\in\Poly{k}{E}$.
Finally, $\sobh{t}{\Th}$ for any $t\geq 1$ is the \emph{broken Sobolev
  space} of globally $\lebl{}$-integrable functions on $\Omega$ whose
restriction to any mesh element $E$ of the mesh $\Th$ belongs to
$\sobh{t}{E}$; formally, we can write that
\begin{align}
  \sobh{t}{\Th} 
  := \Big\{ v\in\lebl{\Omega}\,:\, v_{|E}\in\sobh{t}{E},\,\forall E\in\Th \Big\}.
\end{align}
To ease the notation, since these spaces contain discontinuous
functions, in the following we will intend all norms and seminorms to
be ``broken'' on the mesh.
For example:
\begin{equation*}
  \norm{\nabla v} = 
  \left(
    \sum_{E\in\Th}\norm[E]{\nabla v}^2
  \right)^{\frac12}.
\end{equation*}
Furthermore, we will use the symbol $\mathcal{C}$ to denote a generic
constant independent of the mesh size and the problem data $\K$,
$\beta$ and $\gamma$.
In the estimates this constant may have a different value for each
occurrence.

\section{The variational formulation}
\label{sec:problem}
Let $\Omega\subset\mathbb{R}^d$, $d=2,3$ be a polytopal domain with
boundary $\partial\Omega$ and consider the
convection-diffusion-reaction problem:
\begin{equation}
  \label{eq:problem}
  \begin{array}{rl}
    -\div\left(\K\nabla u\right) +\beta\cdot\nabla u + \gamma u= f &
    \mbox{\text{in $\Omega$}},
    \\
    u=0 & \mbox{\text{on $\partial\Omega$}},
  \end{array}
\end{equation}
We assume that $\K\in\left[\lebl[\infty]{\Omega}\right]^{d\times d}$
is a strongly elliptic and symmetric tensor almost everywhere (a.e.)
on $\Omega$.
Hence, there exist two positive constant $\kappa_*$ and $\kappa^*$
such that
$\kappa_*{\bm\xi}\cdot{\bm\xi}\leq{\bm\xi}\cdot\K(x){\bm\xi}\leq\kappa^*{\bm\xi}\cdot{\bm\xi}$
for every ${\bm\xi}\in\mathbb{R}^d$ and almost every $x\in\Omega$.
We denote $\mathcal{C}_{\kappa}=\kappa^*\slash{\kappa_*}$.
Moreover, we assume that
$\beta\in\left[\lebl[\infty]{\Omega}\right]^d$ with $\div\beta=0$ and
$\gamma\in\lebl[\infty]{\Omega}$ such that
$\inf_{x\in\Omega}\gamma(x)=\gamma_0\geq 0$.
To ease the exposition, we present the virtual element formulation and
the convergence analysis assuming homogeneous Dirichlet boundary
conditions.
However, all the results presented in this paper can readily be
extended to more general situations.

Consider the bilinear form
$\B{}{}\colon \sobho{\Omega}{}\times\sobho{\Omega}{}\to \mathbb{R}$
defined by
\begin{equation}
  \label{eq:defB}
  \B{w}{v} \defeq \scal{\K\nabla w}{\nabla v}
  + \scal{\beta\cdot\nabla w}{v} + \scal{\gamma w}{v}
  \quad \forall w,v \in \sobho{\Omega}{},
\end{equation}
and the linear functional $\F{}\colon \sobho{\Omega}{}\to \mathbb{R}$
defined by
\begin{equation*}
  \label{eq:defF}
  \F{v} \defeq \scal{f}{v}
  \quad \forall v \in \sobho{\Omega}{}.
\end{equation*}
The variational formulation of \eqref{eq:problem} reads as: \emph{Find
  $u\in\sobho{\Omega}{}$ such that}
\begin{equation}
  \label{eq:exvarform}
  \B{u}{v} = \F{v}\quad \forall v\in\sobho{\Omega}{}.
\end{equation}
The bilinear form $\B{}{}$ is coercive and bounded, and the
variational problem~\eqref{eq:exvarform} has a unique solution in view
of the Lax-Milgram lemma.

\section{The virtual element formulation}
\label{sec:vem}

Hereafter, we consider only the case for $d=2$.
However, the nonconforming virtual element formulation is almost the
same for $d=2$ and $3$, the main substantial difference being
necessarily in the mesh assumptions that for $d=3$ must also consider
a star-shaped condition on the faces.
Therefore, most of the results presented in the next sections can
easily be generalized to the three-dimensional case with minor or no
changes at all.

\subsection{General assumptions}
\label{sec:mesh}
Let $\big\{\Th\big\}_{h}$ be a sequence of meshes of $\Omega$, i.e., a
sequence of non-overlapping polygonal partitions of the domain
$\Omega$.
Each $\Th$ is labeled by the subscript $h$, the maximum diameter of
its polygonal elements $E$.
The polygonal elements can have a different number of edges and
hanging node-like configurations are possible with nodes placed on an
edge and forming a flat angle.
We denote the set of all the mesh edges $e$ of the polygonal cells in
$\Th$ by $\Sh$.
We also distinguish between the subset of internal edges $\Sh^{int}$
and the subset of the boundary edges $\Sh^{bnd}$; clearly,
$\Sh=\Sh^{int}\cup\Sh^{bnd}$.

\medskip
We assume that the members of the sequence $\big\{\Th\big\}_{h}$
satisfy the following regularity assumptions:
\emph{There exists a global constant $\rho>0$ such that for each mesh
  $\Th$:}
\begin{enumerate}
\item[$(i)$] every polygon $E\in\Th$ is star-shaped with respect to a
  ball whose radius is greater than or equal to $\rho h_E$, where
  $h_E=\max_{\bx,\by\in E}\norm{\bx-\by}$ is the element diameter;
  \smallskip
\item[$(ii)$] $\forall E\in\Th$, each side $e$ of $E$ is such that
  $h_e\geq \rho h_E$, where $h_e$ is the length of $e$; \smallskip
\end{enumerate}

\begin{remark}
  Assumption $(i)$ implies that each element is simply connected.
  Assumption $(ii)$ implies that the number of sides of each polygon
  of the mesh is uniformly bounded over the mesh sequence.
  %% \cite{Beirao-Lovadina-Russo:2016}.
\end{remark}
The restriction of $\K$ to any element $E\in\Th$ is still a strongly
elliptic tensor and its spectrum can be locally bounded by using two
constants $\Kmin[E]$ and $\K[E]$, so that for any vector-valued field
${\bm\xi}(x)$ defined on $E$ it holds that
\begin{align}
  \Kmin[E]{\bm\xi}(x)\cdot{\bm\xi}(x)
  \leq{\bm\xi}(x)\cdot\K(x){\bm\xi}(x)
  \leq\K[E]{\bm\xi}(x)\cdot{\bm\xi}(x)
  \quad\,\forall x\in E.
  \label{eq:K:local:ellipticity}
\end{align}
We will find convenient for the next theoretical developments to
assume that the inequalities
$0<\kappa_* \leq \Kmin[E] \leq \K[E] \leq \kappa^*$ holds true for
every mesh element $E$.
Since $\K$ is represented by a symmetric and positive definite matrix
we consider the decomposition $\K={\left(\sqrt{\K}\right)}^\intercal \sqrt{\K}$ and we write
$\scal[E]{\K\nabla v}{\nabla v}=\scal[E]{\sqrt{\K}\nabla
  v}{\sqrt{\K}\nabla v}=\norm[E]{\sqrt{\K}\nabla v}^2$ for any
sufficiently regular function $v$.
Therefore, setting ${\bm\xi}=\nabla v$
in~\eqref{eq:K:local:ellipticity} yields
\begin{align*}
  \Kmin[E]\norm[E]{\nabla v}^2\leq\norm[E]{\sqrt{\K}\nabla
  v}^2 \leq\K[E]\norm[E]{\nabla v}^2.
\end{align*}
We will use this relation extensively in the analysis of the next
sections.
For each element $E\in\Th$ we also set
\begin{align*}
  \beta_E &\defeq \sup_{x\in E}\norm[\mathbb{R}^2]{\beta(x)}\,,
  & \gamma_E &\defeq  \norm[\infty,E]{\gamma}.
\end{align*}
Let $k\geq 0$ be an integer number and
$\boldsymbol{\alpha}=(\alpha_1,\alpha_2)$ a two-dimensional
multi-index of order
$\abs{\boldsymbol{\alpha}}=\alpha_1+\alpha_2\leq k$.
The polynomial space $\Poly{k}{E}$ is spanned by the monomials
$m_{\boldsymbol{\alpha}}\in\monom{k}{E}$ defined as
\begin{equation*}
  \label{eq:defmonomials}
  m_{\boldsymbol{\alpha}}(\mathbf{x}) \defeq
  \frac{(\mathbf{x}-\mathbf{x}_E)^{\boldsymbol{\alpha}}}
  {h_E^{\abs{\boldsymbol{\alpha}}}} \quad \forall\mathbf{x}\in E,
\end{equation*}
where $\mathbf{x}_E$ is the center of the ball with respect to which
$E$ is star-shaped.
Similarly, $\Poly{k}{e}$, the space of polynomials of degree $k$
defined on edge $e$, is spanned by the monomials
$m_{\alpha}(\xi):=(\xi-\xi_{e})^{\alpha}\slash{h_e^\alpha}\in\monom{k}{e}$
for $0\leq\alpha\leq k$, where $\xi$ is a local coordinate defined on
$e$, $\xi_{e}$ the midpoint of $e$, and $h_e$ the length of $e$.

In the formulation of the method we will make use of the
\emph{elliptic projection operator}
$\proj[\nabla]{k}{}\colon\sobh{1}{\Th}\rightarrow\Poly{k}{\Th}$, whose
restriction to each element $E$ is the solution of the local problem:
\begin{equation}
  \label{eq:defPinabla}
  \begin{cases}
    \scal[E]{\nabla\proj[\nabla]{k}{}v}{\nabla p} = \scal[E]{\nabla
      v}{\nabla p} & \quad\forall p\in\Poly{k}{E}
    \\[0.5em]
    \scal[\partial E]{\proj[\nabla]{k}{}v}{1} = \scal[\partial
    E]{v}{1} & \quad\text{if $k=1$},
    \\[0.5em]
    \scal[E]{\proj[\nabla]{k}{}v}{1} = \scal[E]{v}{1} & \quad\text{if
      $k>1$}.
  \end{cases}
\end{equation}
We will also consider the $\lebl{}$-\emph{projection operator}
$\proj[0]{l}{}:\sobh{1}{\Th}\rightarrow\Poly{l}{\Th}$ %with $l=k,k-1$,
whose restriction to each element $E$ is the $\lebl{}$-projection onto
$\Poly{l}{E}$.
A crucial property of these projection operators, which will be
discussed in the next section (see Remark~\ref{remark:projections}),
is that they are computable on the functions of the virtual element
space using only their degrees of freedom.

\subsection{The local nonconforming virtual element space}
\label{subsec:local:VEM:space}
The local nonconforming virtual element space of order $k\geq 1$ is
defined as follows:
\begin{align*}
  V_h^E \defeq \bigg\{ v_h\in \sobh{1}{E}\colon 
  & \Delta
  v_h\in\Poly{k}{E},\,\pder{v_h}{\n{e}} \in\Poly{k-1}{e} \,\forall
  e\subset\partial E,
  \\
  & \scal[E]{v_h}{p}=\scal[E]{\proj[\nabla]{k}{}v_h}{p}\,\forall
  p\in\Poly{k}{E}/ \Poly{k-2}{E} \bigg\},
\end{align*}
where $\Poly{k}{E}/ \Poly{k-2}{E}$ is the subspace of $\Poly{k}{E}$ of
the polynomials that are $\lebl{}$-orthogonal to $\Poly{k-2}{E}$ (or,
alternatively, the polynomials whose degree is exactly $k-1$ and $k$),
and for $k=1$ we conventionally take $\Poly{-1}{E}=\{0\}$.
The definition of $V_h^E$ is based on the \emph{enhancement strategy}
that was introduced in \cite{Ahmad-Alsaedi-Brezzi-Marini-Russo:2013}
for the conforming case and extended to the nonconforming case in
\cite{Cangiani-Manzini-Sutton:2016}.
From the definition above it follows immediately that $\Poly{k}{E}$ is
a linear subspace of $V_h^E$.

\begin{figure}[!t]
  \centering
  \begin{tabular}{cccc}
    \includegraphics[width=0.2\textwidth]{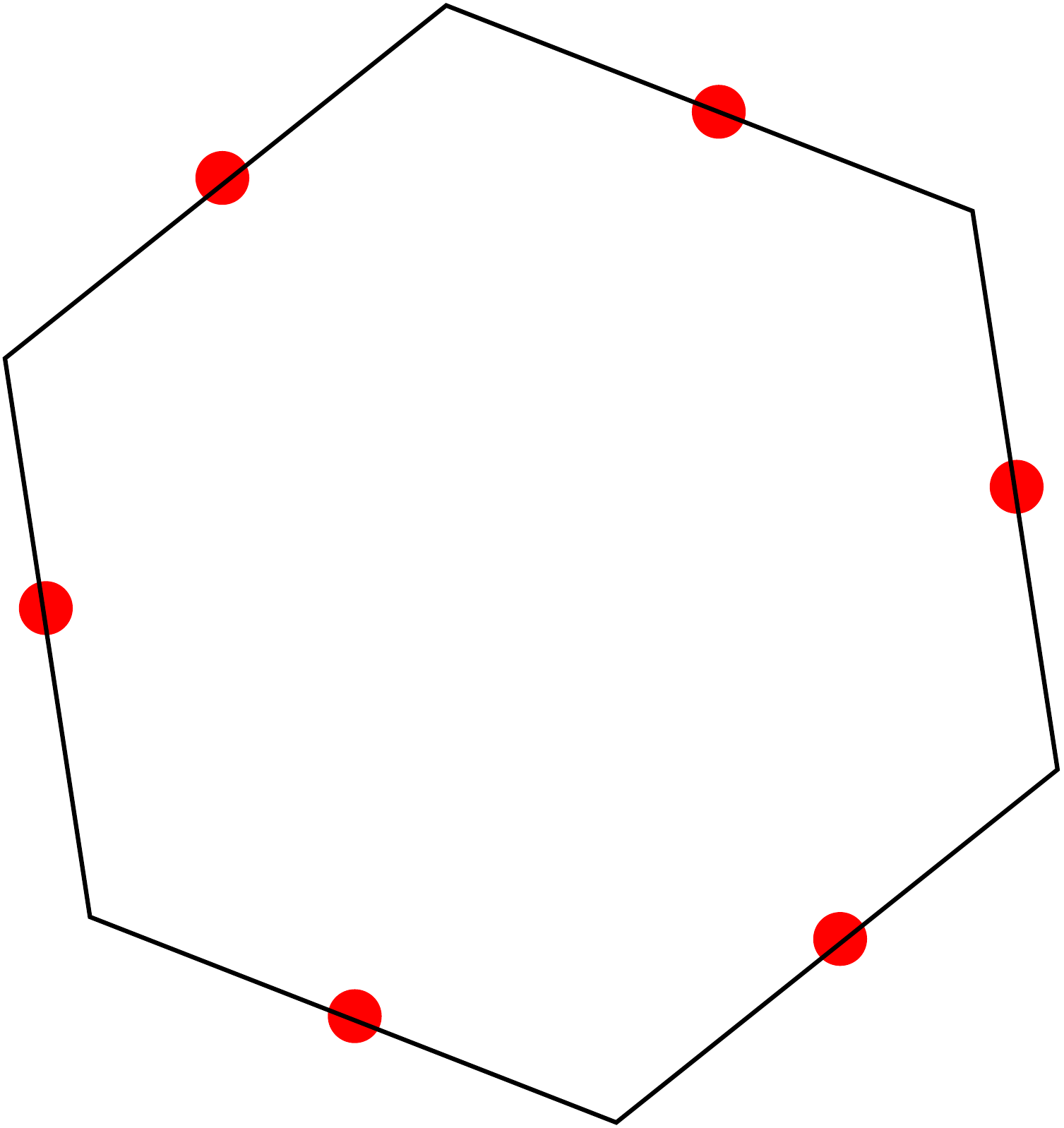} &\quad
    \includegraphics[width=0.2\textwidth]{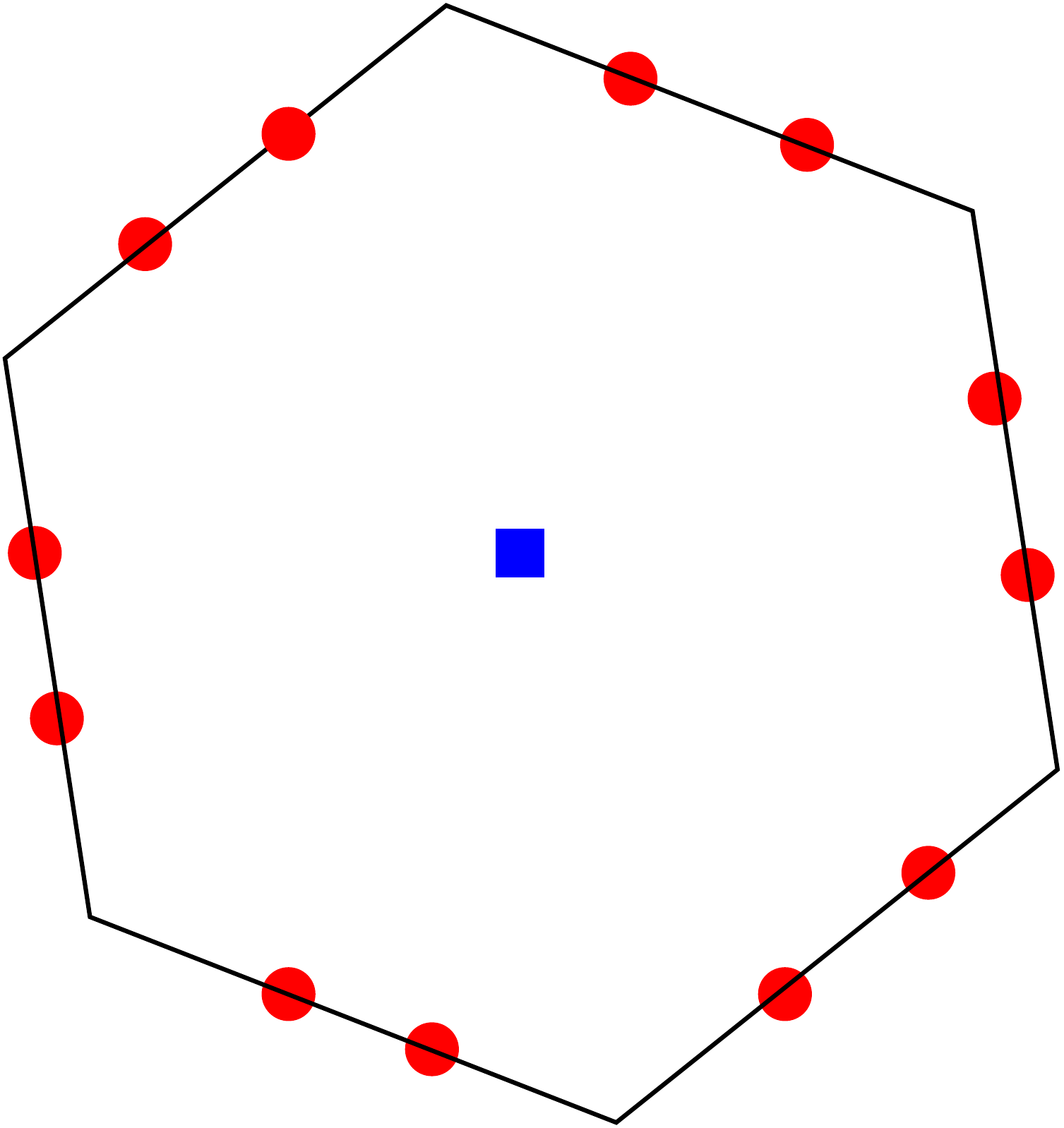} &\quad
    \includegraphics[width=0.2\textwidth]{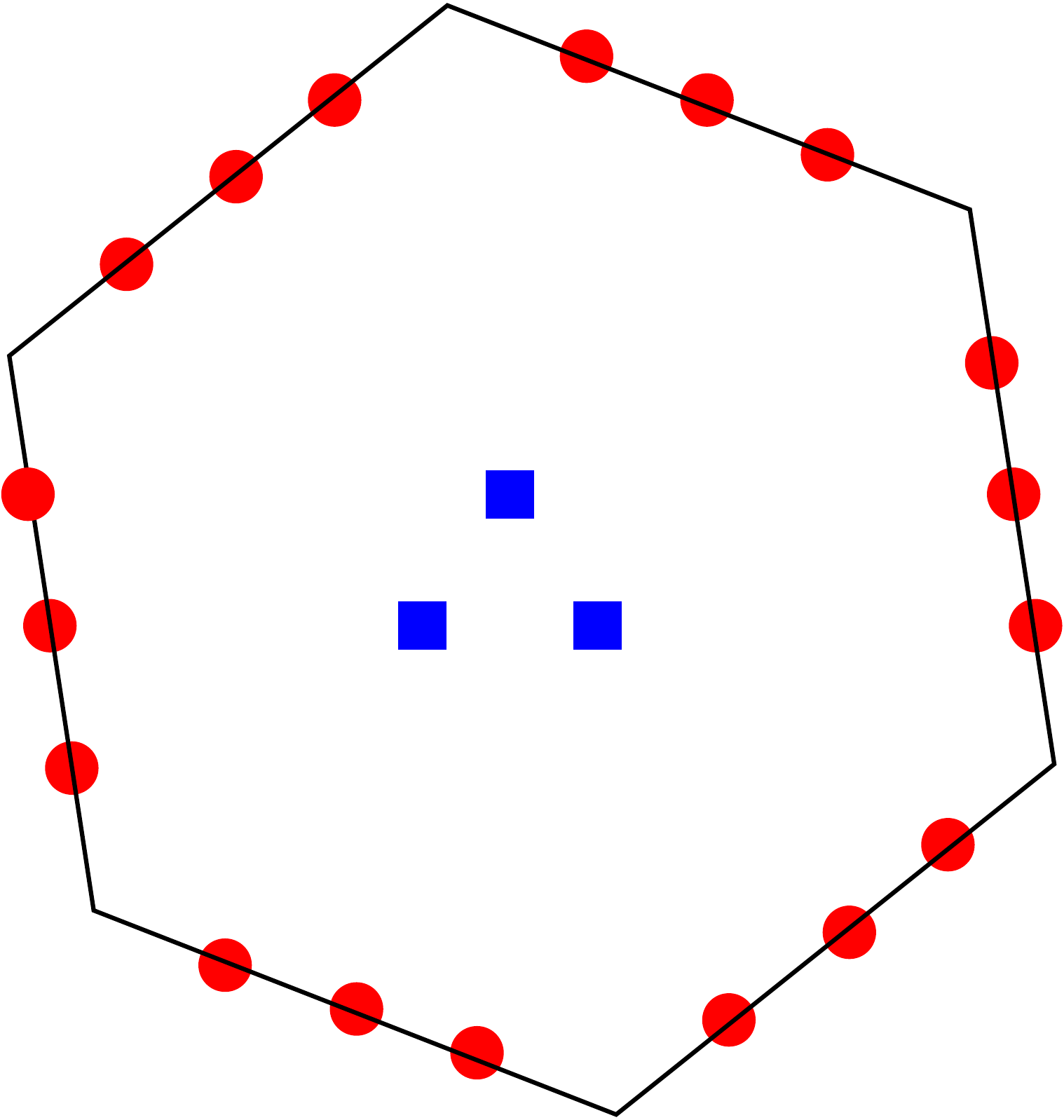} &\quad
    \includegraphics[width=0.2\textwidth]{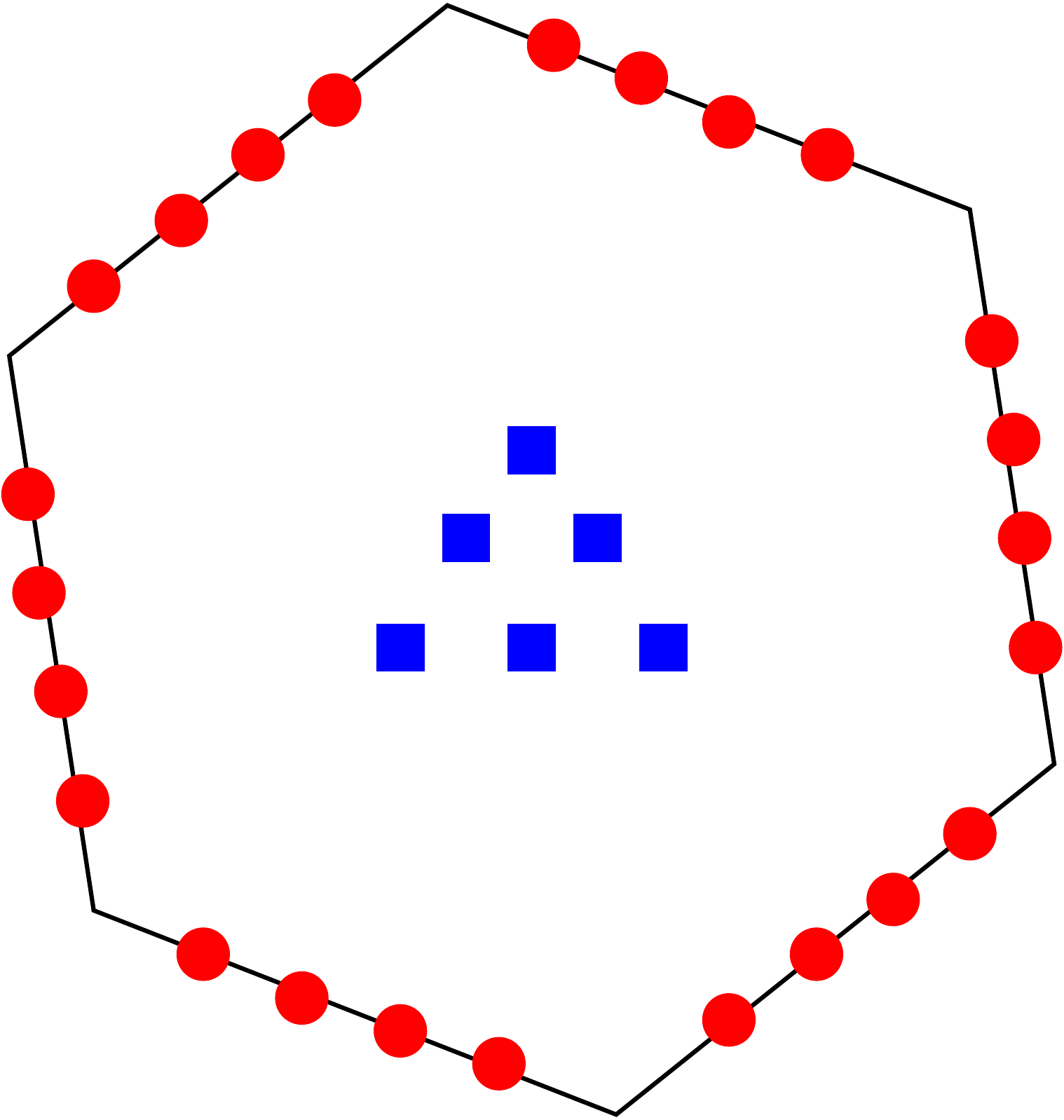} \\
    $\mathbf{k=1}$ & $\mathbf{k=2}$ & $\mathbf{k=3}$ & $\mathbf{k=4}$
  \end{tabular}
  \caption{Degrees of freedom of a hexagonal cell for $k=1,2,3,4$;
    edge moments are marked by a circle; cell moments are marked by a
    square.}
  \label{fig:dofs}
\end{figure}

\medskip
A function $v_h\in V_h^E$ is uniquely identified by the following set
of degrees of freedom:
\begin{itemize}
\item for $k\geq 1$, the moments of $v_h$ of order up to $k-1$ on each mesh interface
  $e$:
  \begin{equation}
    \frac{1}{\abs{e}}\int_{e}v_h m_{\alpha}\dxi
    \qquad\forall m_{\alpha}\in\mathcal{M}_{k-1}(e)\,;
  \end{equation}
\item for $k>1$, the moments of $v_h$ of order up to $k-2$ inside
  element $E$:
  \begin{equation}
    \frac{1}{|E|}\int_{E}v_h m_{\alpha}\dx
    \qquad\forall m_{\alpha}\in\mathcal{M}_{k-2}(E),
  \end{equation}  
\end{itemize}
The unisolvency of these degrees of freedom is proved
in~\cite{AyusodeDios-Lipnikov-Manzini:2016}.
A counting argument shows that the cardinality of this set of degrees
of freedom, which is also the dimension of $V_h^E$, is equal to
$n_E (k-1) + k(k-1)/2$, where $n_E$ is the number of edges of $E$.
The degrees of freedom for an hexagonal cell are shown in
Figure~\ref{fig:dofs}.

\begin{remark}\label{remark:projections}
  The elliptic projection $\proj[\nabla]{k}{}v_h$ is computable from
  the degrees of freedom of $v_h$.
  In fact, an integration by parts of the right-hand side
  of~\eqref{eq:defPinabla} yields:
  \begin{align*}
    \scal[E]{\nabla v_h}{\nabla p} = 
    - \scal[E]{v_h}{\Delta p} 
    + \sum_{e\in\partial E} \scal[e]{v_h}{\bn_e\cdot\nabla p},  
  \end{align*}
  The terms on the right can be expressed by using the $(k-2)$-order
  moments of $v_h$ inside $E$ and the $(k-1)$-order moments of $v_h$
  on each edge $e\in\partial E$ and are thus computable.
  A similar argument shows that also $\proj[0]{k-1}{}v_h$ and
  $\proj[0]{k-1}{}\nabla v_h$ are computable from the degrees of
  freedom of $v_h$.
\end{remark}

\subsection{Global nonconforming virtual element spaces}
\label{subsec:global:VEM:space}
For the construction of the global virtual element spaces we introduce
the nonconforming functional space
\begin{align*}
  \sobnc{k}{\Th}\defeq\bigg\{ v\in\sobh{1}{\Th}\,:\,
  \int_{e}\jump{v}\,q\,\dxi = 0\quad\forall
  q\in\Poly{k-1}{e}\; \forall e\in\Sh
  \bigg\},
\end{align*}
where $\jump{\EMPTY}$ denotes the \emph{jump operator} $\jump{\EMPTY}$
across a mesh interface, which is defined as follows.
If $e$ is an internal edge, we fix a unique unit normal vector $\n{e}$ and we set
$\jump{v}:=v^{+} - v^{-}$, where $v^{\pm}$ are the
traces of $v$ on $e$ from within the two elements $E^{\pm}$ sharing
the edge, being $E^+$ the element for which $\n{e}$ is pointing outward.
If $e$ is a boundary edge, $\n{e}$ is orthogonal to $e$ and
pointing out of  the computational domain $\Omega$ and $\jump{v}:=v^{+}$.

\medskip 
Finally, the \emph{global nonconforming virtual element space of order
  $k$} is defined by
\begin{align}
  V_h := \Big\{
  v_h\in\sobnc{k}{\Th}\,:\,{v_h}_{|E}\in V_h^E\quad\forall E\in\Th
  \Big\}.
\end{align}
Each function $v_h$ of $V_h$ is uniquely characterized by:
\begin{itemize}
\item for $k\geq 1$, the moments of order up to $k-1$ on each internal
  mesh edge $e\in\Sh^{int}$:
  \begin{equation}
    \frac{1}{|e|}\int_{e}v_h m_{\alpha}\dxi
    \qquad\forall m_{\alpha}\in\mathcal{M}_{k-1}(e);
  \end{equation}
\item for $k>1$, the moments of order up to $k-2$ inside each element
  $E\in\Th$:
  \begin{equation}
    \frac{1}{|E|}\int_{E}v_h m_{\alpha}\dx
    \qquad\forall m_{\alpha}\in\mathcal{M}_{k-2}(E).
  \end{equation}
\end{itemize}
The unisolvency of these degrees of freedom in $V_h$ is a direct
consequence of the unisolvency of the local degrees of freedom
introduced in section~\ref{subsec:local:VEM:space} and the definition
of the nonconforming space $\sobnc{k}{\Th}$,
cf.~\cite{AyusodeDios-Lipnikov-Manzini:2016}.

\subsection{SUPG-VEM formulation}
\label{sec:supg}
The discretization of the variational formulation \eqref{eq:exvarform}
may lead to instabilities when the convective term
$\scal{\beta\cdot\nabla w}{v}$ is dominant with respect to the
diffusive term $\scal{\K\nabla w}{\nabla v}$.
Here we consider also a moderate reaction term that we assume not to
be source of instabilities.
In this section we recast the classical \emph{Streamline Upwind Petrov
  Galerkin} (SUPG) approach \cite{Franca-Frey-Hughes:1992} in the
framework of the nonconforming VEM, showing that the optimal order of
convergence can be preserved.
To this end, we assume that
$\K\in\left[\sob{1}{\infty}{\Omega}\right]^{d\times d}$.
Then, we introduce the functional space
\begin{equation}
  \label{eq:defV}
  V \defeq\left\{v\in\sobho{\Omega}{}\colon \Delta v\in\lebl{E}\quad\forall
    E\in\Th\right\},
\end{equation}
the bilinear form
$\Bsupg{}{}\colon V\times \sobh[0]{1}{\Omega}\rightarrow \mathbb{R}$
given by
\begin{equation}
  \label{eq:defBsupg}
  \Bsupg{w}{v}  \defeq \a{w}{v} + \b{w}{v} + \c{w}{v} + \d{w}{v},
\end{equation}
where
\begin{align}
  \label{eq:defa}
  \a{w}{v} & \defeq\sum_{E\in\Th}\scal[E]{\K\nabla w}{\nabla
             v}+\tau_E\scal[E]{\beta\cdot\nabla
             w}{\beta\cdot\nabla v},
  \\
  \label{eq:defb}
  \b{w}{v} & \defeq \frac12\sum_{E\in\Th}
             \Big[
             \scal[E]{\beta\cdot\nabla w}{v}-\scal[E]{w}{\beta\cdot \nabla v}
             \Big],
  \\[0.75em]
  \label{eq:defc}
  \c{w}{v} & \defeq \sum_{E\in\Th}\scal[E]{\gamma w}{v +
             \tau_E \beta \cdot \nabla v},
  \\
  \label{eq:defd}
  \d{w}{v} & \defeq -\sum_{E\in\Th}\tau_E
             \scal[E]{\div(\K\nabla w)}{\beta\cdot\nabla v}.
\end{align}
Furthermore, let
$\Fsupg[]{}\colon\sobh[0]{1}{\Omega}\rightarrow\mathbb{R}$ be the
linear functional given by
\begin{equation}
  \label{eq:defFsupg}
  \Fsupg[]{v} = \scal{f}{v} + \sum_{E\in\Th} 
  \tau_E\scal[E]{f}{\beta\cdot\nabla v}.
\end{equation}
The real positive factor $\tau_E$ is the \emph{local SUPG parameter}
and is discussed in section~\ref{subsec:tau_E}.
The SUPG variational formulation of problem~\eqref{eq:problem} reads
as: \emph{Find $u\in V$ such that}
\begin{equation}
  \label{eq:exvarform:supg}
  \Bsupg{u}{v} = \Fsupg[]{v}\quad \forall v\in \sobh[0]{1}{\Omega}.
\end{equation}
\begin{remark}
  Under the assumptions of Section \ref{sec:problem}, the bilinear
  term $\b{}{}$ in~\eqref{eq:defB}
  % \begin{equation*}
  %   \b{w}{v} = \sum_{E\in\Th}\scal[E]{\beta\cdot\nabla w}{v},
  % \end{equation*}
  that corresponds to the convective flux is equivalent to the
  skew-symmetric term $\b{}{}$ in~\eqref{eq:defb}.
\end{remark}
\begin{remark}
  \label{remark:defa:alt}
  By introducing the matrix
  $\K[\beta,E]=\K+\tau_E\beta\beta^\intercal$, the bilinear form
  $\a{}{}$ in~\eqref{eq:defa} can be reformulated as:
  \begin{equation}
    \label{eq:defa:alt}
    \a{w}{v} \defeq \sum_{E\in\Th} \aE{w}{v} =
    \sum_{E\in\Th} \scal[E]{\K[\beta,E]\nabla w}{\nabla v}.
  \end{equation}
  Since matrix $\K[\beta,E]$ is positive definite, we can use the
  decomposition $\K[\beta,E]=\sqrt{\K[\beta,E]}\sqrt{\K[\beta,E]}$ and
  prove that the bilinear form is continuous, i.e,
  \begin{equation}
    \aE{w}{v}
    \leq \norm[E]{\sqrt{\K[\beta,E]}\nabla w}\,
    \norm[E]{\sqrt{\K[\beta,E]}\nabla v},
    \label{eq:a:continuity}
  \end{equation}
  which holds for every pair of nonconforming functions $v$, $w$.
\end{remark}

\medskip
The SUPG-stabilized virtual element approximation of
\eqref{eq:exvarform} reads as: \emph{Find $u_h\in V_h$ such that}
\begin{equation}
  \label{eq:vemvarform_supg}
  \Bsupg[h]{u_h}{v_h} = \Fsupg[h]{v_h} \quad \forall v_h\in V_h,
\end{equation}
where the bilinear form
$\Bsupg[h]{}{}\colon V_h\times V_h\rightarrow\mathbb{R}$ and the
right-hand side $\Fsupg[h]{} \colon V_h \rightarrow \mathbb{R}$ are
the virtual element approximation of $\Bsupg[]{}{}$ and $\Fsupg[]{}$,
respectively.
The bilinear form $\Bsupg[h]{}{}$ is given by
\begin{equation}
  \label{eq:defBhsupg}
  \Bhsupg{w_h}{v_h} \defeq \a[h]{w_h}{v_h} + \b[h]{w_h}{v_h} +
  \c[h]{w_h}{v_h} + \d[h]{w_h}{v_h} 
\end{equation}
for any $w_h,v_h\in V_h$, where
\begin{align}
  \a[h]{w_h}{v_h} 
  & \defeq
    \sum_{E\in\Th}\Big(\scal[E]{\K\Pi_{k-1}^0\nabla w_h}{\Pi_{k-1}^0\nabla v_h}
    + \tau_E \scal[E]{\beta\cdot\Pi_{k-1}^0\nabla w_h}{\beta\cdot\Pi_{k-1}^0\nabla v_h}
    \nonumber\\[-0.5em]
  &\phantom{\quad\defeq\sum_{E\in\Th}\Big(}
    + \shE{\left(I-\proj{k-1}{}\right)w_h}{\left(I-\proj{k-1}{}\right)v_h}\Big),
    \label{eq:defah}
  \\[0.25em]
  \b[h]{w_h}{v_h} 
  & \defeq \sum_{E\in\Th}
    \frac12\Big(
    \scal[E]{\beta\cdot\proj{k-1}{}\nabla w_h}
    {\proj{k-1}{}v_h}-\scal[E]{\proj{k-1}{}w_h}
    {\beta\cdot\proj{k-1}{}\nabla v_h}
    \Big),
    \label{eq:defbh}
  \\[0.5em]
  \c[h]{w_h}{v_h} 
  & \defeq \sum_{E\in\Th}
    \scal[E]{\gamma\Pi^0_{k-1}w_h}{\Pi^0_{k-1}v_h + \tau_E\beta\cdot\Pi^0_{k-1}\nabla v_h},
    \label{eq:defch}
  \\[0.5em]
  \d[h]{w_h}{v_h} 
  & \defeq -\sum_{E\in\Th}\tau_E
    \scal[E]{\div\big(\K\Pi_{k-1}^0\nabla w_h\big)}{\beta\cdot\Pi_{k-1}^0\nabla v_h},
    \label{eq:defdh}
\end{align}
the local VEM stabilization term in $\a[h]{w_h}{v_h}$ is given by
\begin{equation}
  \shE{\left(I-\proj{k-1}{}\right)w_h}{\left(I-\proj{k-1}{}\right)v_h}
  \defeq
  (\K[E]+\tau_E\beta_E^2)
  \vemstabAST[E]{\left(I-\proj{k-1}{}\right)w_h}
  {\left(I-\proj{k-1}{}\right)v_h},
  \label{eq:local:SUPG:stabterm}
\end{equation}
where
$\vemstabAST[E]{\left(I-\proj{k-1}{}\right)v_h}{\left(I-\proj{k-1}{}\right)v_h}$
is such that there exist two constants $\sigma^\ast,\,\sigma_\ast>0$,
independent of $h$ and the problem parameters satisfying,
$\forall v_h\in V_h $,
\begin{equation}
  \label{eq:SASTequivalence}
  \begin{split}
    \sigma_\ast \norm[E]{\nabla \left(I-\proj{k-1}{}\right) v_h}^2 \leq
    \vemstabAST[E]{\left(I-\proj{k-1}{}\right)v_h}
    {\left(I-\proj{k-1}{}\right)v_h} \leq \sigma^\ast
    \norm[E]{\nabla \left(I-\proj{k-1}{}\right) v_h}^2 \,.
  \end{split}
\end{equation}
Moreover, the linear functional $\Fsupg[h]{v_h}$ is given by
\begin{equation}
  \label{eq:defFsupgh}
  \Fsupg[h]{v_h} = \scal{f}{\Pi_{k-1}^0 v_h} + \sum_{E\in\Th} 
  \tau_E\scal[E]{f}{\beta\cdot\Pi^0_{k-1}\nabla v_h},
\end{equation}
for any $v_h\in V_h$.
In view of Remark~\ref{remark:defa:alt}, we can define the local
bilinear form
\begin{align*}
  \ahE{w_h}{v_h} = \scal[E]{\K[\beta,E]\Pi_{k-1}^0\nabla w_h}{\Pi_{k-1}^0\nabla v_h} 
\end{align*}
such that
\begin{align*}
  \a[h]{w_h}{v_h} = \sum_{E\in\Th}\left(\ahE{w_h}{v_h}
  +\shE{\left(I-\proj{k-1}{} \right) w_h}
  {\left(I-\proj{k-1}{}\right) v_h}\right) \,.
\end{align*}
Notice that, by \eqref{eq:a:continuity},
\eqref{eq:local:SUPG:stabterm} and \eqref{eq:SASTequivalence}, the
stabilization term
$\vemstab[E]{}{}\colon V_h\times V_h\rightarrow\mathbb{R}$ satisfies
\begin{equation}
  \label{eq:vemstab_coercivity}
  \begin{split}
    \vemstab[E]{\left(I-\proj{k-1}{}\right)
      v_h}{\left(I-\proj{k-1}{}\right) v_h} &\geq \sigma_\ast \left(
      \K[E] + \tau_E\beta_E^2 \right) \norm[E]{\nabla v_h - \nabla
      \proj{k-1}{} v_h}^2
    \\
    &\geq \sigma_\ast \left( \K[E] + \tau_E\beta_E^2 \right)
    \norm[E]{\nabla v_h - \proj{k-1}{} \nabla v_h}^2
    \\
    &\geq \sigma_\ast \norm[E]{\sqrt{\K[\beta,E]}\left(\nabla v_h -
        \proj{k-1}{} \nabla v_h\right)}^2 \,,
  \end{split}
\end{equation}
being
$\norm[E]{\nabla v_h - \proj{k-1}{}\nabla v_h} \leq \norm[E]{\nabla
  v_h - \nabla \proj{k-1}{} v_h}$.
According to~\cite{Cangiani-Manzini-Sutton:2016}, a possible choice
for $\vemstabAST[E]{}{}$ is given by
\begin{equation}
  \label{eq:defSE}
  \vemstabAST[E]{\left(I-\proj{k-1}{}\right)w_h}
  {\left(I-\proj{k-1}{}\right)v_h} = \sum_{i=1}^{N_E}
  \chi_i\left(\left(I-\proj{k-1}{}\right)w_h\right)
  \chi_i \left(\left(I-\proj{k-1}{}\right)v_h\right),
\end{equation}
where $N_E$ is the number of degrees of freedom on the element $E$ and
$\chi_i$ is the operator that selects the $i$-th degree of freedom.

The effect of the SUPG stabilization in the VEM stabilization is
reflected by the term $\tau_E\beta_E^2$ that appears in the local
coefficient multiplying $\vemstab[E]{}{}$ in
definition~\eqref{eq:local:SUPG:stabterm}.

\subsection{The SUPG parameter $\tau_E$}
\label{subsec:tau_E}

According to~\cite{Franca-Frey-Hughes:1992,
  Benedetto-Berrone-Borio-Pieraccini-Scialo:2016a}, the stability
parameter $\tau_E$ when there is no reaction term is defined by
\begin{equation}
  % \label{eq:deftau}
  \tau_E =\frac{h_E}{2\beta_E}\min\left\{\Pe[E],1\right\},
  % \end{equation*}
  \qquad{\rm where}\qquad
  % \begin{gather}
  \label{eq:defPe}
  \Pe[E] \defeq m_k^E\frac{\beta_E h_E}{2 \K[E] }, 
  % \intertext{and}
\end{equation}
and
\begin{gather}
  \notag \label{eq:defmk} m_k^E \defeq
  \begin{cases}
    \frac13 & \text{if $\div\left(\K\nabla v_h\right) = 0$
      $\forall v_h\in V_h^E$,}
    \\
    2\tilde{C}_k^E & \text{otherwise,}
  \end{cases}
\end{gather}
is the \emph{mesh P\'{e}clet number} of $E$; $\tilde{C}^E_k$ is the
biggest constant number satisfying the following inverse inequality:
\begin{equation}
  \label{eq:estimnormlapl}
  \tilde{C}^E_k h_E^2\norm[E]{\div(\K\nabla v_h)}^2 \le
  \norm[E]{\K\nabla v_h}^2 \quad \forall v_h\in V_h^E.
\end{equation}
A proof of such a local inverse inequality for the virtual element
space $V_h^E$ is provided in
\cite[Lemma~10]{Cangiani-Georgoulis-Pryer-Sutton:2016} for constant
$\K$ and any function with polynomial laplacian under the current mesh
regularity assumptions.
For a nonconstant $\K$, using standard manipulations we obtain
\eqref{eq:estimnormlapl} with a constant $\tilde{C}_k^E$ that may
depend on the variations of $\K$ on the element.

In~\cite{Franca-Frey-Hughes:1992,
  Benedetto-Berrone-Borio-Pieraccini-Scialo:2016a} no reaction term
was considered.
Since here we have such a term we need to modify the definition of
$\tau_E$ by adding a constraint that guarantees the coercivity of
$\Bsupg{}{}$ and $\Bsupg[h]{}{}$.
As in the proof of Lemma \ref{lem:coercivity_Bh} we need that
$\frac{1}{2}-\frac{\gamma_E\tau_E}{2}\geq C>0$, we may assume that
there exists a constant $C_\tau \in (0,1)$ such that
\begin{equation}
  \label{eq:deftau}
  \tau_E \defeq \min\left\{
    \frac{\tilde{C}^E_k h^2_E}{\K[E]}, 
    \frac{h_E}{2\beta_E},
    \frac{C_\tau}{\gamma_E}\right\}.
\end{equation}

% Karlovitz number (vedi anche Damk\"{o}hler)
Finally, we introduce the local \emph{Karlovitz number}, i.e., the
dimensionless parameter associated with each mesh element $E$,
\begin{equation*}
  \mathrm{Ka}_E \defeq \frac{2\beta_E C_\tau}{h_E\gamma_E},
\end{equation*}
and we redefine $\tau_E$ as
\begin{equation*}
  \tau_E = \frac{h_E}{2\beta_E}\min\left\{\Pe[E],1,\mathrm{Ka}_E\right\}.
\end{equation*}
The comparison between $\Pe[E]$ and $\mathrm{Ka}_E$ determines whether
the value of $\tau_E$ is dominated by the convective term, the
diffusive term, or the reactive term.
The two curves in Figure \ref{fig:tauregimes} show the behaviour of
$\tau_E$ for two possible choices of the problem coefficients.
The curves are parametrized by the diameter $h_E$ with decreasing
values from left to right along each curve.
We see that for small values of $h_E$ (right-most part of each curve),
$\tau$ falls in the diffusive regime, possibly passing through the
convective regime, as expected.
\begin{figure}
  \centering
  \includegraphics[width=.5\textwidth]{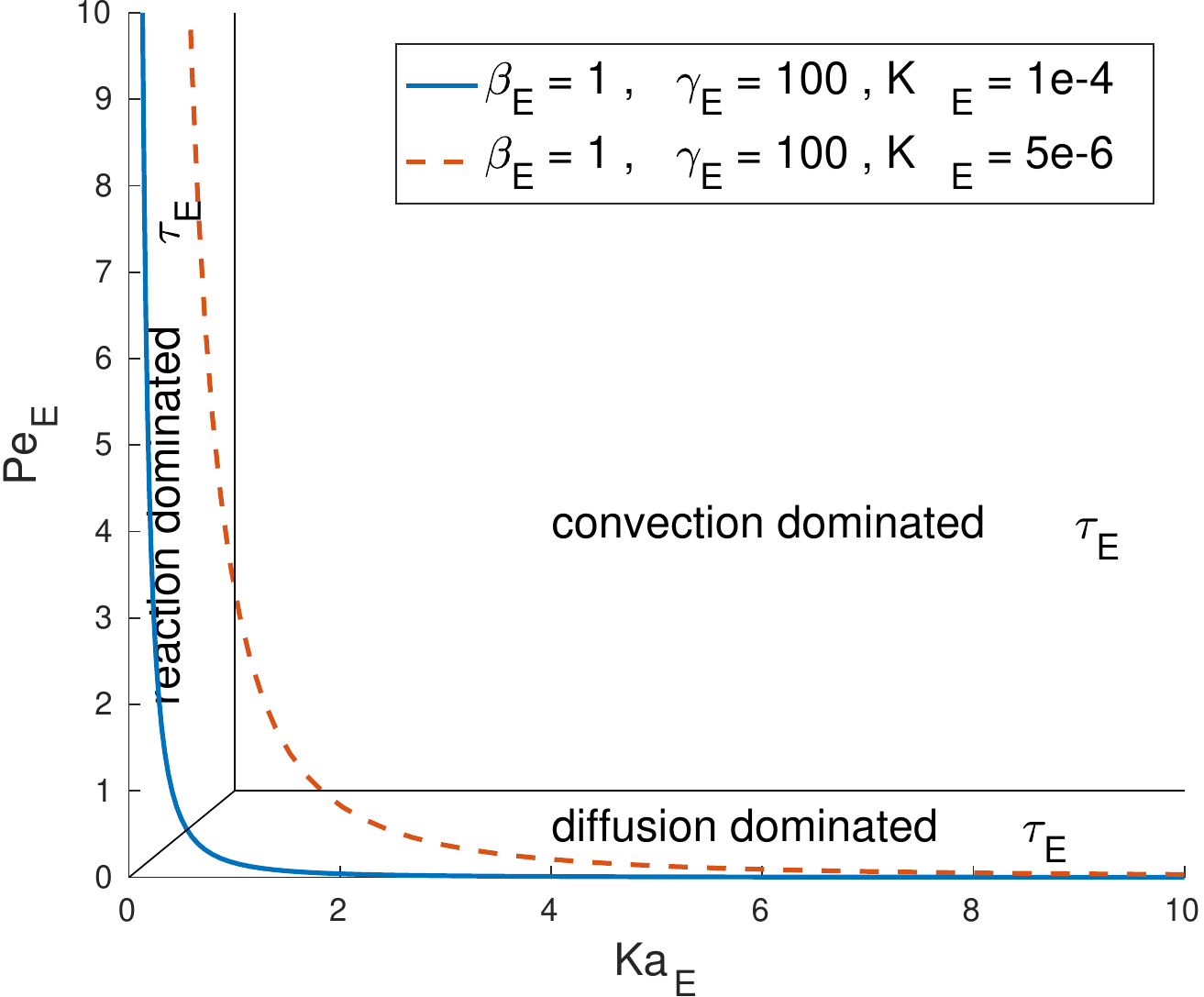}
  \caption{Different regimes of $\tau_E$ for different values of
    $\Pe[E]$ and $\mathrm{Ka}_E$.  }
  \label{fig:tauregimes}
\end{figure}

\section{Error Analysis}
\label{sec:estimate}

In the following we assume that the problem is well written in the
non-dimensional way, and consequently $\K[E]\leq 1$, $h_E\leq 1$,
$\beta_E=O(1)$, $\gamma_E \leq O(1)$.
Let $h\defeq \max_{E\in\Th}h_E$ and define the following norms:
\begin{align}
  \label{eq:defennormKbeta}
  \ennorm[\K\beta]{v} &\defeq \left( \a{v}{v} + \a[h]{v}{v} \right)^{\frac12} \,,
  \\
  \label{eq:defennormKbetagamma}
  \ennorm[\K\beta\gamma]{v} &\defeq \left(
                              \ennorm[\K\beta]{v}^2
                              + \norm{\sqrt{\gamma} \proj{k-1}{} v}^2
                              \right)^{\frac12} \,,
\end{align}
on the nonconforming space $\sobnc{k}{\Th}$. In
\eqref{eq:defennormKbeta}, for the evaluation of $\vemstab[E]{v}{v}$,
we assume to use the VEM interpolant of the function $v$, see
\cite[Theorem 11]{Cangiani-Georgoulis-Pryer-Sutton:2016}.
Clearly, $\ennorm[\K\beta]{v}\leq\ennorm[\K\beta\gamma]{v}$.

\subsection{Discretization errors}
\label{sec:discretization_errors}
The following Lemmas~\ref{lem:estim_a-ah}, \ref{lem:estim_b-bh},
\ref{lem:estim_c-ch}, and~\ref{lem:estim_d-dh} provide a continuity
bound for the discrete bilinear
forms~\eqref{eq:defah}-\eqref{eq:defdh} and an estimate of the
approximation error when compared with the corresponding continuous
ones~\eqref{eq:defa}-\eqref{eq:defd}.
Throughout the section, we use the approximation results for the local
polynomial projections of a function $v\in\sobh{s+1}{E}$,
cf. \cite[Lemma 5.1]{BeiraodaVeiga-Brezzi-Marini-Russo:2015}, given
by:
\begin{align}
  \label{eq:estimP0k-1}
  \norm[E]{v-\proj{k-1}{}v} + h_E\seminorm[1,E]{v-\proj{k-1}{}v} 
  & \leq\mathcal{C} h_E^{s+1}\seminorm[s+1,E]{v}\quad
    1\leq s+1\leq k,
  \\[0.25em]
  \label{eq:estimPnablak}
  \norm[E]{v-\proj[\nabla]{k}{}v} + h_E\seminorm[1,E]{v-\proj[\nabla]{k}{}v} 
  & \leq \mathcal{C}h_E^{s+1}\seminorm[s+1,E]{v} \quad
    1\leq s+1\leq k+1,
\end{align}
which hold for every mesh element $E$ and polynomial degree $k\geq 1$
under the mesh assumptions of Section~\ref{sec:mesh}.
For every internal edge $e=\partial E^+\cap\partial E^-$ and functions
$v\in\sobh{s+1}{\omega_e}$ with $\omega_e=E^+\cup E^-$ we will also
consider the trace inequality
\begin{align}
  \label{eq:estim:trace:Pnablak}
  \norm[e]{v-\proj[0,e]{k-1}v}+h_{e}\seminorm[1,e]{v-\proj[0,e]{k-1}v} 
  & \leq\mathcal{C}h_{e}^{s+\frac{1}{2}}\seminorm[s+1,\omega_e]{v}\quad
    1\leq s+1\leq k \,.
\end{align}

We use the error estimate for the virtual element interpolant of order
$k$ of a function $\varphi\in\sobh{s+1}{E}$, $1 \leq s+1\leq k+1$
\cite{AyusodeDios-Lipnikov-Manzini:2016}:
\begin{equation}
  \label{eq:estim_veminterp}
  \norm[E]{\varphi-\varphi_I} +
  \hE\seminorm[1,E]{\varphi-\varphi_I} 
  \leq\mathcal{C} 
  \hE^{s+1} \seminorm[s+1,E]{\varphi}.
\end{equation}
Furthermore, \eqref{eq:estim_veminterp} implies that
\begin{equation}
  \label{eq:H1:apriori:proof:00}
  \begin{split}
    \ennorm[\K\beta\gamma]{\varphi-\varphi_I}^2 & = \sum_{E\in\Th}
    \left( \norm[E]{\sqrt{\K}\nabla(\varphi-\varphi_I)}^2 +
      \tau_E\norm[E]{\beta\cdot\nabla(\varphi-\varphi_I)}^2 + \right.
    \\
    & \quad + \norm[E]{\sqrt{\K} \proj{k-1}{
        \nabla(\varphi-\varphi_I)}}^2 + \tau_E\norm[E]{\beta\cdot
      \proj{k-1}{\nabla(\varphi-\varphi_I)}}^2
    \\
    & \left. \quad + \shE{\left(I-\proj{k-1}{}\right) \left(
          \varphi-\varphi_I \right)}{\left(I-\proj{k-1}{}\right)
        \left( \varphi-\varphi_I \right)} \right.
    \\
    &\left. \quad+ \norm[E]{\sqrt{\gamma}
        \proj{k-1}{\varphi-\varphi_I}}^2\right)
    \\
    &\leq \mathcal{C} \sum_{E\in\Th}\max \left\{\K[E] \,,
      \tau_E\beta_E^2 \,, h_E^2\gamma_E \right\} h^{2s}_E
    \seminorm[s+1,E]{\varphi}^2,
  \end{split}
\end{equation}
for some positive constant $\mathcal{C}$ independent of $h$ and the
local problem coefficients $\K$, $\beta$, $\gamma$.

\begin{assumption}
  We assume that the solution $u$ to \eqref{eq:exvarform:supg} belongs
  to $\sobh{s+1}{\Th}\cap V$, with $1 < s+1 \leq k+1$ and that
  $\K\in\big[\sob{s}{\infty}{\Omega}\big]^{d\times d}$,
  $\beta\in\left[\sob{s+1}{\infty}{\Omega}\right]^{d}$,
  $\gamma\in\left[\sob{s+1}{\infty}{\Omega}\right]$.
\end{assumption}

The following technical lemma is needed in the upcoming proofs.
\begin{lemma}
  \label{lem:estim_ab-Pab}
  Let $a,b \in \sob{s}{\infty}{E}$ be given, $E\in\Th$.  Then,
  \begin{equation}
    \label{eq:estim_ab-Pab}
    \norm[\sob{s}{\infty}{E}]{ab - \proj{0}{ab}} \leq \frac32 \left(
      \norm[\sob{s}{\infty}{E}]{a} \norm[\sob{s}{\infty}{E}]{b - \proj{0}{}b} +
      \norm[\sob{s}{\infty}{E}]{b} \norm[\sob{s}{\infty}{E}]{a - \proj{0}{}a}
    \right) \,.
  \end{equation}
\end{lemma}
\begin{proof}
  We consider the following decomposition, exploiting the fact that
  $\proj{0}{}a \proj{0}{}b = \proj{0}{a\proj{0}{}b} =
  \proj{0}{b\proj{0}{}a}$:
  \begin{equation*}
    \begin{split}
      ab - \proj{0}{ab} &= \frac12 \left( ab + ab - \proj{0}{ab} -
        \proj{0}{ab} \right)
      \\
      &= \frac12 \left[ ab - \proj{0}{a}b + \proj{0}{a}b -
        \proj{0}{a}\proj{0}{b} + \proj{0}{b\proj{0}{a}} - \proj{0}{ab}
      \right.
      \\
      & \left. \quad + ab - a \proj{0}{b} + a \proj{0}{b} -
        \proj{0}{a} \proj{0}{b} + \proj{0}{a\proj{0}{b}} -
        \proj{0}{ab} \right]
      \\
      &= \frac12 \left[\left(a-\proj{0}{}a\right) b +
        \left(b-\proj{0}{}b\right) \proj{0}{}a + a
        \left(b-\proj{0}{}b\right) + \left(a-\proj{0}{}a\right)
        \proj{0}{}b \right.
      \\
      & \left. \quad \proj{0}{\left(\proj{0}{}a - a\right) b} +
        \proj{0}{\left(\proj{0}{}b - b\right)a} \right]\,.
    \end{split}
  \end{equation*}
  The proof is concluded by the triangle inequality, the
  Cauchy-Schwarz inequality and by exploiting the fact that
  $\norm[\sob{s}{\infty}{E}]{\proj{0}{}a} = \abs{\eval{ \left(
        \proj{0}{}a \right)}{E}} \leq \norm[\infty,E]{a} \leq
  \norm[\sob{s}{\infty}{E}]{a}$.
\end{proof}

We now estimate the terms inside $\Bsupg[h]{}{}$ to analyse their
continuity, and their consistency with respect to polynomials of order
$k$.
%% Lemma, ah
%%
\begin{lemma}
  \label{lem:estim_a-ah}
  For every function $w\in\sobnc{k}{\Th}$ and
  $v_h\in V_h\subset\sobnc{k}{\Th}$,
  \begin{equation}
    \label{eq:ah_continuity}
    \a[h]{w}{v_h} \leq 
    \ennorm[\K\beta\gamma]{w}\ennorm[\K\beta\gamma]{v_h} \,.
  \end{equation}
  Moreover, if $w\in\sobnc{k}{\Th}\cap\sobh{s+1}{\Th}$, then
  \begin{equation}
    \label{eq:estim_a-ah}
    \abs{\a{\proj{k-1}{}w}{v_h} - \a[h]{\proj{k-1}{}w}{v_h}} \leq
    \mathcal{C}
    \max_{E\in\Th}\left\{\mathcal{C}^{nc}_{a,E} \right\} h^s
    \norm[s+1]{w} \ennorm[\K\beta]{v_h} \,,
  \end{equation}
  where
  \begin{equation}
    \label{eq:defCNCa}
    \mathcal{C}^{nc}_{a,E} = \frac{\norm[\sob{s}{\infty}{E}]{\K[\beta,E]
        - \proj{0}{\K[\beta,E]}}}{\sqrt{\Kmin[E]}} \,.
  \end{equation}
\end{lemma}
\begin{proof}
  Regarding \eqref{eq:ah_continuity}, to estimate the continuity of
  $\a[h]{}{}$, we use the Cauchy-Schwarz and H\"{o}lder inequalities
  and the definition of the norm \eqref{eq:defennormKbeta}:
  \begin{equation*}
    \a[h]{w}{v_h} \leq \left(\a[h]{w}{w}\a[h]{v_h}{v_h}\right)^{\frac12}
    \leq  \ennorm[\K\beta]{w} \ennorm[\K\beta]{v_h}\,.
  \end{equation*}
  To prove~\eqref{eq:estim_a-ah}, we first notice that
  $\vemstab[E]{\left(I-\proj{k-1}{}\right)\proj{k-1}{}w}{v_h}=0$ and
  $\proj{k-1}{}\nabla \proj{k-1}{}w = \nabla \proj{k-1}{}w$:
  \begin{align*}
    &\abs{\a{\proj{k-1}{} w}{v_h} - \a[h]{\proj{k-1}{} w}{v_h}}
      \leq \sum_{E\in\Th}
      \abs{ \scal[E]{\K[\beta,E]\nabla \proj{k-1}{} w}{\nabla v_h} 
      \right.
    \\
    &    \left. \quad -\scal[E]{\K[\beta,E]\proj{k-1}{}\nabla\proj{k-1}{} w}
      {\proj{k-1}{}\nabla v_h}} =\sum_{E\in\Th}
      \abs{ \scal[E]{\K[\beta,E]\nabla \proj{k-1}{} w}{\nabla v_h-\proj{k-1}{}\nabla v_h} }\,.
  \end{align*}
  The local terms are bounded using the $k$-consistency
  $\ahE{\cdot}{\cdot} = \aE{\cdot}{\cdot}$ when the coefficients are
  constants and one of the arguments is a polynomial:
  \begin{equation*}
    \begin{split}
      &\scal[E]{\K[\beta,E]\nabla \proj{k-1}{} w} {\nabla
        v_h-\proj{k-1}{}\nabla v_h} =
      \scal[E]{\left(\K[\beta,E]-\proj{0}{\K[\beta,E]}\right) \nabla
        \proj{k-1}{} w}{\nabla v_h-\proj{k-1}{}\nabla v_h}
      \\
      &= \scal[E]{\left(\K[\beta,E]-\proj{0}{\K[\beta,E]}\right)
        \nabla \proj{k-1}{} w -
        \proj{k-1}{\left(\K[\beta,E]-\proj{0}{\K[\beta,E]}\right)
          \nabla \proj{k-1}{} w}}{\nabla v_h-\proj{k-1}{}\nabla v_h}
      \\
      &\leq \norm[E]{\left(\K[\beta,E]-\proj{0}{\K[\beta,E]}\right)
        \nabla \proj{k-1}{} w -
        \proj{k-1}{\left(\K[\beta,E]-\proj{0}{\K[\beta,E]}\right)
          \nabla \proj{k-1}{} w}} \norm[E]{\nabla v_h-\proj{k-1}{}\nabla
        v_h}
      \\
      & \leq \mathcal{C}h^s_E
      \seminorm[s,E]{\left(\K[\beta,E]-\proj{0}{\K[\beta,E]}\right)
        \nabla \proj{k-1}{} w} \left(\norm[E]{\nabla
          v_h}+\norm[E]{\proj{k-1}{}\nabla v_h}\right)
      \\
      & \leq \mathcal{C} h^{s}_E
      \frac{\norm[\sob{s}{\infty}{E}]{\K[\beta,E]-\proj{0}{\K[\beta,E]}}}
      {\sqrt{\Kmin[E]}} \norm[s+1,E]{w} \ennorm[\K\beta,E]{v_h} \,.
    \end{split}
  \end{equation*}
\end{proof}

\begin{remark}
  We can bound $\mathcal{C}^{nc}_{a,E}$ as follows, using
  \eqref{eq:estim_ab-Pab} and \eqref{eq:deftau}:
  \begin{equation*}
    \begin{split}
      \mathcal{C}^{nc}_{a,E} &=
      \frac{\norm[\sob{s}{\infty}{E}]{\K[\beta,E]-\proj{0}{\K[\beta,E]}}}
      {\sqrt{\Kmin[E]}}
      \\
      & \leq \frac{1}{\sqrt{\Kmin[E]}}
      \left(\norm[\sob{s}{\infty}{E}]{\K - \proj{0}{}\K} + \tau_E
        \norm[\sob{s}{\infty}{E}]{ \beta\beta^\intercal -
          \proj{0}{\beta\beta^\intercal}} \right)
      \\
      & \leq \frac{\mathcal{C}}{\sqrt{\Kmin[E]}}
      \left(\norm[\sob{s}{\infty}{E}]{\K - \proj{0}{}\K} +
        \frac{h_E}{\beta_E} \norm[\sob{s}{\infty}{E}]{\beta}
        \norm[\sob{s}{\infty}{E}]{ \beta - \proj{0}{\beta} } \right)
      \,.
    \end{split}
  \end{equation*}
\end{remark}

%%
%% Lemma, bh
%%
\begin{lemma}
  \label{lem:estim_b-bh}
  For every function $w\in\sobnc{k}{\Th}$ and
  $v_h\in V_h\subset\sobnc{k}{\Th}$ it holds that
  \begin{equation}
    \label{eq:bh_continuity}
    \begin{split}
      \abs{\b[h]{w}{v_h}} &\leq \mathcal{C} \left[ \max_{E\in\Th}
        \left( \tau_E^{-\frac12} , h^{-1}_E
          \mathcal{K}^{nc}_{b,E}\right) \norm{w} + \max_{E\in\Th}
        \left( \mathcal{C}^{nc,1}_{b,E} , \mathcal{K}^{nc}_{b,E} \right)
        \norm{\nabla w} \right] \ennorm[\K\beta]{v_h}\,,
    \end{split}
  \end{equation}
  where
  \begin{align}
    \label{eq:defCNCb}
    \mathcal{C}^{nc,r}_{b,E} &=  \frac{h_E
                             \norm[\sob{r}{\infty}{E}]{\beta-\proj{0}{}\beta}}
                             {\sqrt{\Kmin[E]}}  \,,\ r\geq 1 \,,
    \\
    \label{eq:defKNCb}
    \mathcal{K}^{nc}_{b,E} &= \frac{h_E
                             \norm[\sob{1}{\infty}{\partial E}]{\beta\cdot\n{}}}
                             {\sqrt{\Kmin[E]}} \,.
  \end{align}
  Moreover, for every function
  $w\in\sobnc{k}{\Th}\cap\sobh{s+1}{\Th}$, it holds that
  \begin{equation}
    \label{eq:estim_b-bh}
    \abs{\b{\proj{k-1}{}w}{v_h} - \b[h]{\proj{k-1}{}w}{v_h}} \leq\mathcal{C}
    \max_{E\in\Th} \mathcal{C}^{nc,s+1}_{b,E} \, h^{s}\norm[s+1]{w}
    \ennorm[\K\beta]{v_h} \,.
  \end{equation}
\end{lemma}
\begin{proof}
  %% 
  %% LEMMA 1, first part
  %%
  To obtain \eqref{eq:bh_continuity}, we first introduce the following
  decomposition:
  \begin{equation*}
    \b[h]{w}{v_h} = \sum_{E\in\Th} \TERM{T}{E,1} + \TERM{T}{E,2}
    + \TERM{T}{E,3} + \TERM{T}{E,4}\,,
  \end{equation*}
  where, $\forall E\in\Th$,
  \begin{align}
    \label{eq:bh_continuity-TERM1}
    \TERM{T}{E,1} &= \scal[E]{\beta\cdot\left(\proj{k-1}{}\nabla w
                    - \nabla w\right)}{\proj{k-1}{} v_h} \,,
    \\
    \label{eq:bh_continuity-TERM2}
    \TERM{T}{E,2} &= \scal[E]{\beta\cdot\nabla w}
                    {\proj{k-1}{} v_h - v_h} \,,
    \\
    \label{eq:bh_continuity-TERM3}
    \TERM{T}{E,3} &= \scal[E]{\beta\cdot\nabla w}{v_h} \,,
    \\
    \label{eq:bh_continuity-TERM4}
    \TERM{T}{E,4} &= -\scal[E]{\proj{k-1}{} w}{
                    \beta\cdot\proj{k-1}{}\nabla v_h} \,.
  \end{align}
  We estimate $\TERM{T}{E,1}$ in \eqref{eq:bh_continuity-TERM1} as
  follows:
  \begin{equation*}
    \begin{split}
      \abs{\TERM{T}{E,1}} &= \abs{\scal[E]{\proj{k-1}{}\nabla w -
          \nabla w}{(\beta-\proj{0}{}\beta)\proj{k-1}{}v_h}}
      \\
      &= \abs{\scal[E]{\nabla
          w}{(\beta-\proj{0}{}\beta)\proj{k-1}{}v_h -
          \proj{k-1}{(\beta-\proj{0}{}\beta)\proj{k-1}{}v_h} } }
      \\
      &\leq \norm[E]{\nabla w}
      \norm[E]{(\beta-\proj{0}{}\beta)\proj{k-1}{}v_h -
        \proj{k-1}{(\beta-\proj{0}{}\beta)\proj{k-1}{}v_h} }
      \\
      & \leq \mathcal{C} h_E \norm[E]{\nabla w}
      \seminorm[1,E]{(\beta-\proj{0}{}\beta)\proj{k-1}{}v_h}
      \\
      & \leq \mathcal{C} h_E \norm[E]{\nabla w}
      \norm[\sob{1}{\infty}{E}]{\beta-\proj{0}{}\beta}
      \norm[1,E]{\proj{k-1}{}v_h}
      \\
      & \leq \mathcal{C} \frac{h_E \norm[\sob{1}{\infty}{E}]{\beta -
          \proj{0}{}\beta}}{\sqrt{\Kmin[E]}} \norm[E]{\nabla w}
      \ennorm[\K\beta]{v_h}\,.
    \end{split}
  \end{equation*}
  $\TERM{T}{E,2}$ in \eqref{eq:bh_continuity-TERM2} is estimated
  similarly, using also \eqref{eq:SASTequivalence}:
  \begin{equation*}
    \begin{split}
      \abs{\TERM{T}{E,2}} &= \abs{\scal[E]{\beta\cdot\nabla w}
        {\proj{k-1}{} v_h - v_h}}
      \\
      &\leq \abs{\scal[E]{w}{\beta\cdot\nabla\left(\proj{k-1}{} v_h -
            v_h\right)}} + \abs{\int_{\partial E} w
        \left(\beta\cdot\n{}\right) \left(\proj{k-1}{}v_h - v_h
        \right)}
      \\
      &\leq \tau^{-\frac12}_E \norm[E]{w} \tau^{\frac12}_E
      \norm[E]{\beta\cdot\nabla\left(\proj{k-1}{} v_h - v_h\right)}
      \\
      &\quad + \mathcal{C} \norm[\partial E]{w} \norm[\infty,\partial
      E]{\beta\cdot\n{}} \norm[\partial E]{v_h-\proj{k-1}{} v_h}
      \\
      &\leq \tau^{-\frac12}_E \norm[E]{w} \cdot
      \sqrt{\frac{1}{\sigma_\ast} \tau_E \beta_E^2
        \vemstabAST[E]{\left(I-\proj{k-1}{}\right)
          v_h}{\left(I-\proj{k-1}{}\right) v_h}}
      \\
      &\quad + \mathcal{C} h^{-\frac12}_E\norm[E]{w}
      \norm[\infty,\partial E]{\beta\cdot\n{}} \cdot h^{\frac12}_E
      \norm[E]{\nabla v_h}
      \\
      &\leq \mathcal{C} \left(\tau_E^{-\frac12} +
        \frac{\norm[\infty,\partial
          E]{\beta\cdot\n{}}}{\sqrt{\Kmin[E]}}\right) \norm[E]{w}
      \ennorm[\K\beta,E]{v_h} \,.
    \end{split}
  \end{equation*}
  The estimation of $\TERM{T}{E,3}$ in \eqref{eq:bh_continuity-TERM3}
  requires an application of Green's formula, as follows:
  \begin{equation*}
    \begin{split}
      \abs{\sum_{E\in\Th}\TERM{T}{E,3}} &\leq
      \abs{\sum_{E\in\Th}\scal[E]{w}{\beta\cdot \nabla v_h}} +
      \abs{\sum_{E\in\Th}\int_{\partial E}\left(\beta\cdot\n{}\right)
        w v_h }
      \\
      &= \abs{\sum_{E\in\Th}\scal[E]{w}{\beta\cdot \nabla v_h}} +
      \abs{\frac12 \sum_{E\in\Th}\int_{\partial
          E}\left(\beta\cdot\n{}\right) \jmp{w v_h} } \,.
    \end{split}
  \end{equation*}
  The first term is estimated locally by the Cauchy-Schwarz
  inequality:
  \begin{equation*}
    \scal[E]{w}{\beta\cdot\nabla v_h} \leq \tau_E^{-\frac12}\norm[E]{w}
    \ennorm[\K\beta,E]{v_h} \,.
  \end{equation*}
  The boundary terms are estimated exploiting
  $\int_{\partial E}\beta\cdot\n{} = \int_E \nabla\cdot\beta = 0$, and denoting by
  $\proj[0,\partial E]{k-1}{}v$ the piecewise polynomial projection
  of $v$ on each $e\subset \partial E$:
  \begin{equation*}
    \begin{split}
      \int_{\partial E}\left(\beta\cdot\n{}\right) \jmp{w v_h} &=
      \sum_{R\in\omega_E} \int_{E\cap R}\left(\beta\cdot\n{}\right)
      \left( \eval{w}{R} \jmp{v_h} + \jmp{w} \eval{v_h}{E} \right)
      \\
      &= \sum_{R\in\omega_E} \int_{E\cap R} \left(
        \left(\beta\cdot\n{}\right)\eval{w}{R} - \proj[0,\partial
        E]{k-1}{\left(\beta\cdot\n{}\right)\eval{w}{R}} \right)
      \jmp{v_h}
      \\
      &\quad + \sum_{R\in\omega_E} \int_{E\cap R} \jmp{w} \left(
        \left(\beta\cdot\n{}\right)\eval{v_h}{E} - \proj[0,\partial
        E]{k-1}{\left(\beta\cdot\n{}\right)\eval{v_h}{E}} \right)
      \\
      &= \sum_{R\in\omega_E} \int_{E\cap R} \left(
        \left(\beta\cdot\n{}\right)\eval{w}{R} - \proj[0,\partial
        E]{k-1}{\left(\beta\cdot\n{}\right)\eval{w}{R}} \right)
      \jmp{v_h - \proj[0,\partial E]{k-1}{v_h}}
      \\
      &\quad + \sum_{R\in\omega_E} \int_{E\cap R} \jmp{w -
        \proj{k-1}{}w } \left(
        \left(\beta\cdot\n{}\right)\eval{v_h}{E} - \proj[0,\partial
        E]{k-1}{\left(\beta\cdot\n{}\right)\eval{v_h}{E}} \right)
      \\
      &= \sum_{R\in\omega_E} \norm[R\cap
      E]{\left(\beta\cdot\n{}\right)\eval{w}{R} - \proj[0,\partial
        E]{k-1}{\left(\beta\cdot\n{}\right)\eval{w}{R}}} \norm[E\cap
      R]{\jmp{v_h - \proj[0,\partial E]{k-1}{v_h}}}
      \\
      &\quad + \sum_{R\in\omega_E} \norm[E\cap R]{\jmp{w -
          \proj{k-1}{}w }} \norm[E\cap R]{
        \left(\beta\cdot\n{}\right)\eval{v_h}{E} - \proj[0,\partial
        E]{k-1}{\left(\beta\cdot\n{}\right)\eval{v_h}{E}}}
      \\
      &\leq \mathcal{C}_1 \left( \sum_{R\in\omega_E} h_E
        \seminorm[1,R\cap E]{\left(\beta\cdot\n{}\right)\eval{w}{R}}
      \right) h^{\frac12}_E \norm[\omega_E]{\nabla v_h}
      \\
      &\quad + \mathcal{C}_2 \left( \sum_{R\in\omega_E} h_E
        \seminorm[1,E\cap R]{
          \left(\beta\cdot\n{}\right)\eval{v_h}{E}} \right)
      h^{\frac12}_E \norm[\omega_E]{\nabla w }
      \\
      & \leq \mathcal{C}_1 h_E^{\frac12}
      \norm[\sob{1}{\infty}{\partial E}]{\beta\cdot\n{}}
      \norm[\omega_E]{\nabla w} \cdot
      h_E^{\frac12}\norm[\omega_E]{\nabla v_h}
      \\
      &\quad + \mathcal{C}_2 h_E^{\frac12}
      \norm[\sob{1}{\infty}{\partial E}]{\beta\cdot\n{}}
      \norm[\omega_E]{\nabla v_h} \cdot
      h_E^{\frac12}\norm[\omega_E]{\nabla w}
      \\
      & \leq \mathcal{C} \frac{h_E \norm[\sob{1}{\infty}{\partial
          E}]{\beta\cdot\n{}}} {\sqrt{\Kmin[E]}}
      \norm[\omega_E]{\nabla w} \ennorm[\K\beta,\omega_E]{v_h} \,.
    \end{split}
  \end{equation*}
  The estimate of $\TERM{T}{E,4}$ defined by
  \eqref{eq:bh_continuity-TERM4} is obtained by the Cauchy-Schwarz
  inequality, the continuity of projections and the definition of the
  norm \eqref{eq:defennormKbeta}:
  \begin{equation*}
    \begin{split}
      \abs{ \TERM{T}{E,4}} &= \abs{\scal[E]{\proj{k-1}{} w}{
          \beta\cdot\proj{k-1}{}\nabla v_h}} \leq \mathcal{C}
      \tau_E^{-\frac12} \norm[E]{w} \ennorm[\K\beta,E]{v_h} \,.
    \end{split}
  \end{equation*}
  To derive \eqref{eq:estim_b-bh}, we set
  \begin{equation*}
    \b{\proj{k-1}{}w}{v_h} - \b[h]{\proj{k-1}{}w}{v_h} = \sum_{E\in\Th}
    \big(\TERM{R}{E,1} - \TERM{R}{E,2}\big)\,,
  \end{equation*}
  where, recalling that
  $\proj{k-1}{\nabla\proj{k-1}{}w} = \nabla\proj{k-1}{}w$,
  \begin{align}
    \label{eq:estim_b-bh-TERM1}
    \TERM{R}{E,1}&=\scal[E]{\beta\cdot\nabla\proj{k-1}{} w}{v_h - \proj{k-1}{}v_h}
                   \,,
    \\
    \label{eq:estim_b-bh-TERM2}
    \TERM{R}{E,2}&=\scal[E]{\proj{k-1}{}w}{\beta\cdot \left( \nabla v_h -
                   \proj{k-1}{}\nabla v_h \right)} \,.
  \end{align}
  $\TERM{R}{E,1}$ can be estimated as follows:
  \begin{equation*}
    \begin{split}
      \TERM{R}{E,1} &= \scal[E]{\beta\cdot\nabla\proj{k-1}{} w}{v_h -
        \proj{k-1}{}v_h} = \scal[E]{\left(\beta -
          \proj{0}{}\beta\right) \cdot \nabla\proj{k-1}{} w}{v_h -
        \proj{k-1}{}v_h}
      \\
      &= \scal[E]{\left(\beta - \proj{0}{}\beta\right) \cdot
        \nabla\proj{k-1}{} w - \proj{k-1}{\left(\beta -
            \proj{0}{}\beta\right) \cdot \nabla\proj{k-1}{} w}}{v_h -
        \proj{k-1}{}v_h}
      \\
      &\leq \norm[E]{\left(\beta - \proj{0}{}\beta\right) \cdot
        \nabla\proj{k-1}{} w - \proj{k-1}{\left(\beta -
            \proj{0}{}\beta\right) \cdot \nabla\proj{k-1}{} w}}
      \norm[E]{v_h - \proj{k-1}{}v_h}
      \\
      &\leq \mathcal{C} h^{s}_E \seminorm[s,E]{\left(\beta -
          \proj{0}{}\beta\right) \cdot \nabla\proj{k-1}{} w} \cdot h_E
      \norm[E]{\nabla v_h}
      \\
      &\leq \mathcal{C} h^{s+1}_E
      \norm[\sob{s}{\infty}{E}]{\beta-\proj{0}{}\beta} \norm[s+1,E]{w}
      \norm[E]{\nabla v_h}
      \\
      &\leq \mathcal{C} \frac{h_E \norm[\sob{s}{\infty}{E}]{\beta -
          \proj{0}{}\beta}}{\sqrt{\Kmin[E]}} h^s_E \norm[s+1,E]{w}
      \ennorm[\K\beta]{v_h} \,.
    \end{split}
  \end{equation*}
  The estimate of $\TERM{R}{E,2}$ in \eqref{eq:estim_b-bh-TERM2} is
  obtained as follows:
  \begin{equation*}
    \begin{split}
      \TERM{R}{E,2}&=\scal[E]{\proj{k-1}{}w}{\beta\cdot \left( \nabla
          v_h - \proj{k-1}{}\nabla v_h \right)} = \scal[E]{\left(
          \beta - \proj{0}{}\beta \right)\proj{k-1}{}w}{ \nabla v_h -
        \proj{k-1}{}\nabla v_h}
      \\
      &= \scal[E]{\left( \beta - \proj{0}{}\beta \right)\proj{k-1}{}w
        - \proj{k-1}{\left( \beta - \proj{0}{}\beta
          \right)\proj{k-1}{}w}}{ \nabla v_h }
      \\
      &\leq \norm[E]{\left( \beta - \proj{0}{}\beta
        \right)\proj{k-1}{}w - \proj{k-1}{\left( \beta -
            \proj{0}{}\beta \right)\proj{k-1}{}w}} \norm[E]{ \nabla
        v_h }
      \\
      &\leq \mathcal{C}h^{s+1}_E\seminorm[s+1,E]{\left( \beta -
          \proj{0}{}\beta \right)\proj{k-1}{}w} \norm[E]{ \nabla v_h }
      \\
      &\leq \mathcal{C}h^{s+1}_E\norm[\sob{s+1}{\infty}{E}]{\beta -
        \proj{0}{}\beta} \norm[s+1,E]{\proj{k-1}{}w} \norm[E]{ \nabla
        v_h }
      \\
      &\leq \mathcal{C}\frac{h_E \norm[\sob{s+1}{\infty}{E}]{\beta -
          \proj{0}{}\beta}}{\sqrt{\Kmin[E]}} h^{s}_E
      \norm[s+1,E]{w} \ennorm[\K\beta,E]{v_h} \,.
    \end{split}
  \end{equation*}
\end{proof}

\begin{remark}
  The coefficient $\mathcal{K}^{nc}_{b,E}$ can be rewritten,
  considering that
%  $\proj[0,\partial E]{0}{\beta\cdot\n{}}=0$,
  $\overline{\beta\cdot\n{}}^{\,\partial E}=\int_{\partial E}\beta\cdot\n{}=0$,
  in the
  following way:
  \begin{equation*}
    \mathcal{K}^{nc}_{b,E} = \frac{h_E
      \norm[\sob{1}{\infty}{\partial E}]{\beta\cdot\n{}}}
    {\sqrt{\Kmin[E]}}
    = \frac{h_E
%      \norm[\sob{1}{\infty}{\partial E}]{\beta\cdot\n{} -
%        \proj[0,\partial E]{0}{\beta\cdot\n{}}}}
      \norm[\sob{1}{\infty}{\partial E}]
      {\beta\cdot\n{} - \overline{\beta\cdot\n{}}^{\,\partial E}}}
    {\sqrt{\Kmin[E]}}
  \end{equation*}
\end{remark}

%% Lemma, ch
%%
\begin{lemma}
  \label{lem:estim_c-ch}
  For every function $w\in\sobnc{k}{\Th}$ and $v_h\in V_h$ it holds
  that
  \begin{equation}
    \label{eq:ch_continuity}
    \abs{\c[h]{w}{v_h}} \leq (1 + \sqrt{C_\tau})\ennorm[\K\beta\gamma]{w}
    \ennorm[\K\beta\gamma]{v_h}\,.
  \end{equation}
  Moreover, for every function
  $w\in\sobnc{k}{\Th}\cap\sobh{s+1}{\Th}$, it holds that
  \begin{equation}
    \label{eq:estim_c-ch}
    \begin{split}
      & \abs{\c{\proj{k-1}{}w}{v_h} - \c[h]{\proj{k-1}{}w}{v_h} } \leq
      \mathcal{C} \max_{E\in\Th} \mathcal{C}^{nc}_{c,E} \, h^s
      \norm[s+1]{w} \ennorm[\K\beta]{v_h} \,,
    \end{split}
  \end{equation}
  \begin{equation}
    \label{eq:defCNCc}
    \mathcal{C}^{nc}_{c,E} = \max \left\{ \frac{h^2_E
        \norm[\sob{s+1}{\infty}{E}]{\gamma-\proj{0}{}\gamma}}{\sqrt{\Kmin[E]}}
      , \frac{h_E\tau_E \norm[\sob{s+1}{\infty}{E}]
        {\gamma\beta-\proj{0}{\gamma\beta}}}{\sqrt{\Kmin[E]}}
    \right\}
  \end{equation}
\end{lemma}
\begin{proof}
  Inequality \eqref{eq:ch_continuity} follows easily from the
  definition of the norm \eqref{eq:defennormKbetagamma} and the
  definition of $\tau_E$ \eqref{eq:deftau}:
  \begin{equation*}
    \begin{split}
      &\scal[E]{\gamma \proj{k-1}{}w}{\proj{k-1}{}v_h} + \tau_E
      \scal[E]{\gamma \proj{k-1}{}w}{\beta\cdot\proj{k-1}{}\nabla v_h}
      \leq \norm[E]{\sqrt{\gamma} \proj{k-1}{}w}
      \norm[E]{\sqrt{\gamma} \proj{k-1}{} v_h}
      \\
      &\quad + \sqrt{\tau_E\gamma_E} \norm[E]{\sqrt{\gamma}
        \proj{k-1}{} w}\cdot \sqrt{\tau_E} \norm[E]{\beta\cdot
        \proj{k-1}{}\nabla v_h} \leq \left(1+\sqrt{C_\tau}\right)
      \ennorm[\K\beta\gamma,E]{w} \ennorm[\K\beta\gamma,E]{v_h} \,.
    \end{split}
  \end{equation*}
  To prove \eqref{eq:estim_c-ch}, we start with:
  \begin{equation*}
    \c{\proj{k-1}{}w}{v_h}-\c[h]{\proj{k-1}{}w}{v_h} = \sum_{E\in\Th}\left(
      \TERM{R}{E,1} +\TERM{R}{E,2}\right)\,,
  \end{equation*}
  where
  \begin{align}
    \label{eq:estim_c-ch-TERM1}
    \TERM{R}{E,1}&=\scal[E]{\gamma \proj{k-1}{} w}{v_h - \proj{k-1}{}v_h} \,,
    \\
    \label{eq:estim_c-ch-TERM2}
    \TERM{R}{E,2}&=\tau_E\scal[E]{\gamma\proj{k-1}{}w}
                   {\beta\cdot\nabla v_h - \beta\cdot\proj{k-1}{}\nabla v_h} \,.
  \end{align}
  The first term, given by \eqref{eq:estim_c-ch-TERM1}, can be bounded
  as follows:
  \begin{equation*}
    \begin{split}
      \TERM{R}{E,1} &= \scal[E]{\left(\gamma - \proj{0}{}\gamma
        \right) \proj{k-1}{} w}{v_h - \proj{k-1}{}v_h}
      \\
      &= \scal[E]{\left(\gamma - \proj{0}{}\gamma \right) \proj{k-1}{}
        w - \proj{k-1}{\left(\gamma - \proj{0}{}\gamma \right)
          \proj{k-1}{} w}}{v_h - \proj{k-1}{}v_h}
      \\
      &\leq \norm[E]{\left(\gamma - \proj{0}{}\gamma \right)
        \proj{k-1}{} w - \proj{k-1}{\left(\gamma - \proj{0}{}\gamma
          \right) \proj{k-1}{} w}} \norm[E]{v_h - \proj{k-1}{}v_h}
      \\
      &\leq \mathcal{C} h^{s+1}_E \seminorm[s+1,E]{\left(\gamma -
          \proj{0}{}\gamma \right) \proj{k-1}{} w} \cdot h_E
      \norm[E]{\nabla v_h}
      \\
      &\leq \mathcal{C} h^{s+2}_E \norm[\sob{s+1}{\infty}{E}]{\gamma -
        \proj{0}{}\gamma} \norm[s+1,E]{\proj{k-1}{}w} \norm[E]{\nabla
        v_h}
      \\
      &\leq \mathcal{C} \frac{h_E^2 \norm[\sob{s+1}{\infty}{E}]{\gamma
          - \proj{0}{}\gamma}}{\sqrt{\Kmin[E]}} h^s_E \norm[s+1,E]{w}
      \ennorm[\K\beta,E]{v_h} \,.
    \end{split}
  \end{equation*}
  The term $\TERM{R}{E,2}$ in \eqref{eq:estim_b-bh-TERM2} can be
  bounded as follows:
  \begin{equation*}
    \begin{split}
      \TERM{R}{E,2}& = \tau_E \scal[E]{\gamma\proj{k-1}{}w}
      {\beta\cdot\nabla v_h - \beta\cdot\proj{k-1}{}\nabla v_h}
      \\
      &= \tau_E \scal[E]{\left(\gamma\beta -
          \proj{0}{\gamma\beta}\right) \proj{k-1}{}w}{\nabla v_h -
        \proj{k-1}{}\nabla v_h}
      \\
      &= \tau_E \scal[E]{\left(\gamma\beta -
          \proj{0}{\gamma\beta}\right) \proj{k-1}{}w -
        \proj{k-1}{\left(\gamma\beta - \proj{0}{\gamma\beta}\right)
          \proj{k-1}{}w}}{\nabla v_h - \proj{k-1}{}\nabla v_h}
      \\
      &\leq \tau_E \norm[E]{\left(\gamma\beta -
          \proj{0}{\gamma\beta}\right) \proj{k-1}{}w -
        \proj{k-1}{\left(\gamma\beta - \proj{0}{\gamma\beta}\right)
          \proj{k-1}{}w}} \norm[E]{\nabla v_h - \proj{k-1}{}\nabla
        v_h}
      \\
      &\leq \mathcal{C} h_E^{s+1} \tau_E
      \seminorm[s+1,E]{\left(\gamma\beta -
          \proj{0}{\gamma\beta}\right) \proj{k-1}{}w} \norm[E]{\nabla
        v_h}
      \\
      &\leq \mathcal{C} h_E \tau_E
      \norm[\sob{s+1}{\infty}{E}]{\gamma\beta-\proj{0}{\gamma\beta}}
      h_E^{s} \norm[s+1,E]{w} \norm[E]{\nabla v_h}
      \\
      &\leq \mathcal{C} \frac{h_E \tau_E
        \norm[\sob{s+1}{\infty}{E}]{\gamma\beta -
          \proj{0}{\gamma\beta} } }{\sqrt{\Kmin[E]}} h^s_E
      \norm[s+1,E]{w} \ennorm[\K\beta,E]{v_h} \,.
    \end{split}
  \end{equation*}

\end{proof}

\begin{remark}
  The second argument of the max in \eqref{eq:defCNCc} can be bounded
  by \eqref{eq:estim_ab-Pab} and \eqref{eq:deftau}:
  \begin{equation*}
    \begin{split}
      \frac{h_E\tau_E \norm[\sob{s+1}{\infty}{E}]
        {\gamma\beta-\proj{0}{\gamma\beta}}}{\sqrt{\Kmin[E]}} \leq
      \frac{\mathcal{C}}{\sqrt{\Kmin[E]}} \left( \frac{C_\tau h_E}{\gamma_E}
        \norm[\sob{s+1}{\infty}{E}]{\gamma}
        \norm[\sob{s+1}{\infty}{E}]{\beta - \proj{0}{}\beta} \right.
      \\
      \left. + \frac{h^2_E}{\beta_E} \norm[\sob{s+1}{\infty}{E}]{
          \beta} \norm[\sob{s+1}{\infty}{E}]{\gamma -
          \proj{0}{}\gamma} \right)
    \end{split}
  \end{equation*}
\end{remark}

\begin{lemma}
  \label{lem:estim_d-dh}
  For any $w\in \sobnc{k}{\Th}$ and $\forall v_h\in V_h$,
  \begin{equation}
    \label{eq:dh_continuity}
    \d[h]{w}{v_h} \leq \ennorm[\K\beta]{w} \ennorm[\K\beta]{v_h} \,.
  \end{equation}
  Moreover, if $w\in V\cap\sobh{s+1}{\Th}$, then
  \begin{equation}
    \label{eq:estim_d-dh}
    \begin{split}
      \abs{\d{\proj{k-1}{}w}{v_h} - \d[h]{\proj{k-1}{}w}{v_h}} &\leq
      \mathcal{C} \max_{E\in\Th} \mathcal{C}^{nc}_{d,E} \, h^s_E
      \norm[s+1]{w} \ennorm[\K\beta]{v_h} \,,
    \end{split}
  \end{equation}
  where
  \begin{equation}
    \label{eq:defCNCd}
    \mathcal{C}^{nc}_{d,E} = \max \left\{ \frac{h^{-1}_E \tau_E
        {\displaystyle\sum_{i=1}^d}
        \norm[\sob{s}{\infty}{E}]{\beta_i\K - \proj{0}{\beta_i\K}}}
      {\sqrt{\Kmin[E]}}
      , \frac{\tau_E
        \norm[\sob{s}{\infty}{E}]{(\nabla \beta)^\intercal \K -
          \proj{0}{\left( \nabla \beta \right)^\intercal \K}
        }}{\sqrt{\Kmin[E]}} \right\} \,.
  \end{equation}
\end{lemma}
\begin{proof}
  To prove \eqref{eq:dh_continuity}, we use the inverse inequality
  \eqref{eq:estimnormlapl} and the definition of the norm
  \eqref{eq:defennormKbeta}: $\forall E\in\Th$,
  \begin{equation*}
    \begin{split}
      \tau_E\scal[E]{\nabla\cdot\left(\K\proj{k-1}{}\nabla
          w\right)}{\beta\cdot\proj{k-1}{}\nabla v_h} &\leq
      \sqrt{\tau_E} \norm[E]{\nabla\cdot\left(\K\proj{k-1}{}\nabla
          w\right)}
      \cdot\sqrt{\tau_E}\norm[E]{\beta\cdot\proj{k-1}{}\nabla v_h}
      \\
      &\leq \frac{1}{\sqrt{\K[E]}}\norm[E]{\K\proj{k-1}{}\nabla
        w}\ennorm[\K\beta]{v_h} \leq \ennorm[\K\beta]{w}
      \ennorm[\K\beta]{v_h} \,.
    \end{split}
  \end{equation*}
  Regarding \eqref{eq:estim_d-dh}, we procede as follows:
  $\forall E\in\Th$,
  \begin{equation*}
    \begin{split}
      &\tau_E\scal[E]{\nabla\cdot\left(\K \proj{k-1}{} \nabla
          w\right)}{\beta\cdot\left(\nabla v_h - \proj{k-1}{}\nabla
          v_h\right)} = \TERM{R}{E,1} + \TERM{R}{E,2}\,,
    \end{split}
  \end{equation*}
  where, with the notation $\mathcal{E}_{k-1} = I-\proj{k-1}{}$,
  \begin{align}
    \label{eq:estim_d-dh-TERM1}
    \TERM{R}{E,1} &= \tau_E\sum_{i=1}^d
                    \scal[E]{\nabla\cdot
                    \left( \beta_i\K \proj{k-1}{} \nabla w \right)}
                    {\mathcal{E}_{k-1}\left(\pder{v_h}{x_i}\right)} \,,
    \\
    \label{eq:estim_d-dh-TERM2}
    \TERM{R}{E,2} &= \tau_E \scal[E]{ -\left( \nabla \beta  \right)^\intercal
                    \K \proj{k-1}{} \nabla w}
                    {\mathcal{E}_{k-1}\left(\nabla v_h \right)} \,.
  \end{align}
  The term $\TERM{R}{E,1}$ in \eqref{eq:estim_b-bh-TERM1} can be
  estimated as follows, using \eqref{eq:estim_ab-Pab},
  \begin{equation*}
    \begin{split}
      \TERM{R}{E,1} &= \tau_E\sum_{i=1}^d \scal[E]{\nabla\cdot \left(
          \left(\beta_i\K - \proj{0}{\beta_i\K}\right) \proj{k-1}{}
          \nabla w \right)}
      {\mathcal{E}_{k-1}\left(\pder{v_h}{x_i}\right)}
      \\
      & = \tau_E\sum_{i=1}^d \scal[E]{\nabla\cdot \left(
          \mathcal{E}_{k-1}\left( \left(\beta_i\K -
              \proj{0}{\beta_i\K}\right) \proj{k-1}{} \nabla w \right)
        \right)} {\mathcal{E}_{k-1}\left(\pder{v_h}{x_i}\right)}
      \\
      & \leq \tau_E\sum_{i=1}^d \norm[E]{\nabla\cdot \left(
          \mathcal{E}_{k-1}\left( \left(\beta_i\K -
              \proj{0}{\beta_i\K}\right) \proj{k-1}{} \nabla w \right)
        \right)}
      \norm[E]{\mathcal{E}_{k-1}\left(\pder{v_h}{x_i}\right)}
      \\
      & \leq \mathcal{C} h^{-1}_E\tau_E
      \sum_{i=1}^d\norm[E]{\mathcal{E}_{k-1}\left( \left(\beta_i\K -
            \proj{0}{\beta_i\K}\right) \proj{k-1}{} \nabla w \right)}
      \norm[E]{\pder{v_h}{x_i}}
      \\
      & \leq \mathcal{C} h^{-1}_E\tau_E h_E^{s}
      \sum_{i=1}^d\seminorm[s,E]{\left(\beta_i\K -
          \proj{0}{\beta_i\K}\right) \proj{k-1}{} \nabla w}
      \norm[E]{\nabla v_h}
      \\
      & \leq \mathcal{C}h^{-1}_E\tau_E h_E^{s} \left(\sum_{i=1}^d
        \norm[\sob{s}{\infty}{E}]{\beta_i\K - \proj{0}{\beta_i\K}}
      \right) \norm[s,E]{\proj{k-1}{} \nabla w} \norm[E]{\nabla v_h}
      \\
      & \leq \mathcal{C} \frac{h^{-1}_E \tau_E \sum_{i=1}^d
        \norm[\sob{s}{\infty}{E}]{\beta_i\K -
          \proj{0}{\beta_i\K}}}{\sqrt{\Kmin[E]}} h^s_E \norm[s+1,E]{w}
      \ennorm[\K\beta,E]{v_h} \,.
    \end{split}
  \end{equation*}
  The term $\TERM{R}{E,2}$ in \eqref{eq:estim_b-bh-TERM2} can be
  estimated as follows, using also \eqref{eq:estim_ab-Pab}:
  \begin{equation*}
    \begin{split}
      \abs{\TERM{R}{E,2}} &= \tau_E \abs{\scal[E]{ \left( \nabla \beta
          \right)^\intercal \K \proj{k-1}{} \nabla w}
        {\mathcal{E}_{k-1}\left(\nabla v_h \right)}}
      \\
      &= \tau_E \abs{ \scal[E]{ \mathcal{E}_{k-1}\left( \left( \left(
                \nabla \beta \right)^\intercal \K - \proj{0}{\left(
                  \nabla \beta \right)^\intercal \K} \right)
            \proj{k-1}{} \nabla w \right)}{\nabla v_h } }
      \\
      &\leq \tau_E \norm[E]{ \mathcal{E}_{k-1}\left( \left( \left(
              \nabla \beta \right)^\intercal \K - \proj{0}{\left(
                \nabla \beta \right)^\intercal \K} \right)
          \proj{k-1}{} \nabla w \right)} \norm[E]{\nabla v_h }
      \\
      &\leq \mathcal{C}\tau_E h^{s}_E \seminorm[s,E]{ \left( \left(
            \nabla \beta \right)^\intercal \K - \proj{0}{\left( \nabla
              \beta \right)^\intercal \K} \right) \proj{k-1}{} \nabla
        w } \norm[E]{\nabla v_h}
      \\
      &\leq \mathcal{C} \frac{\tau_E \norm[\sob{s}{\infty}{E}]{(\nabla
          \beta)^\intercal \K - \proj{0}{\left( \nabla \beta
            \right)^\intercal \K} }}{\sqrt{\Kmin[E]}} h^s_E
      \norm[s+1,E]{w} \ennorm[\K\beta,E]{v_h} \,.
    \end{split}
  \end{equation*}  
\end{proof}
\begin{remark}
  The first argument of the max in \eqref{eq:defCNCd} can be bounded
  as follows, using \eqref{eq:estim_ab-Pab} and \eqref{eq:deftau}:
  \begin{equation*}
    \begin{split}
      & \frac{h^{-1}_E \tau_E \sum_{i=1}^d
        \norm[\sob{s}{\infty}{E}]{\beta_i\K -
          \proj{0}{\beta_i\K}}}{\sqrt{\Kmin[E]}} \leq
      \frac{3h_E^{-1}\tau_E}{2\sqrt{\Kmin[E]}} \left(
        2\norm[\sob{s}{\infty}{E}]{\K} \norm[\sob{s}{\infty}{E}]{\beta
          - \proj{0}{}\beta} \right.
      \\
      &\qquad \left. + 2\norm[\sob{s}{\infty}{E}]{\beta}
        \norm[\sob{s}{\infty}{E}]{\K - \proj{0}{}\K} \right)
      \\
      &\quad \leq \frac{\mathcal{C}}{\sqrt{\Kmin[E]}} \left( h_E
        \norm[\sob{s}{\infty}{E}]{\frac{\K}{\K[E]}}
        \norm[\sob{s}{\infty}{E}]{\beta - \proj{0}{}\beta} +
        \norm[\sob{s}{\infty}{E}]{\frac{\beta}{\beta_E}}
        \norm[\sob{s}{\infty}{E}]{\K - \proj{0}{}\K} \right) \,.
    \end{split}
  \end{equation*}
  Similarly, the second argument in \eqref{eq:defCNCd} can be
  bounded as follows:
  \begin{equation*}
    \begin{split}
      \frac{\tau_E \norm[\sob{s}{\infty}{E}]{(\nabla \beta)^\intercal
          \K - \proj{0}{\left( \nabla \beta \right)^\intercal \K}
        }}{\sqrt{\Kmin[E]}} &\leq \frac{\mathcal{C}}{\sqrt{\Kmin[E]}}
      \left( h^2_E \norm[\sob{s}{\infty}{E}]{\frac{\K}{\K[E]}}
        \norm[\sob{s}{\infty}{E}]{\nabla \beta} \right.
      \\
      &\left. \quad + h_E
        \norm[\sob{s}{\infty}{E}]{\frac{\nabla\beta}{\beta_E}}
        \norm[\sob{s}{\infty}{E}]{\K - \proj{0}{}\K} \right) \,.
    \end{split}
  \end{equation*}
\end{remark}

Finally, the following Lemma states the continuity of $\Bsupg{}{}$,
defined by \eqref{eq:defBsupg}.
\begin{lemma}
  \label{lem:Bsupg_continuity}
  Let $w\in\sobnc{k}{\Th}\cap \sobh{2}{\Th}$ and $v_h\in V_h$. Then,
  \begin{equation}
    \label{eq:Bsupg_continuity}
    \begin{split}
      \Bsupg{w}{v_h} \leq \mathcal{C} \left[ \max_{E\in\Th} \left(
          \tau_E^{-\frac12},h_E^{-1}\mathcal{K}^{nc}_{b,E} \right)
        \norm{w} + \frac{\sqrt{\K[E]}}{\sqrt{\Kmin[E]}}
        \ennorm[\K\beta\gamma]{w} + \max_{E\in\Th}
        \mathcal{K}^{nc}_{b,E} \norm{\nabla w} \right.
      \\
      \left. + \max_{E\in\Th} \left( \sqrt{\tau_E}
          \norm[\sob{1}{\infty}{E}]{\K} \right)
        \norm[2,E]{w-\proj{k-1}{}w} \right]
      \ennorm[\K\beta\gamma,E]{v_h} \,.
    \end{split}
  \end{equation}
\end{lemma}
\begin{proof}
  The proof of the continuity of $\a{}{}$, $\b{}{}$ and $\c{}{}$
  follows the same arguments of Lemmas \eqref{lem:estim_a-ah},
  \eqref{lem:estim_b-bh} and \eqref{lem:estim_c-ch}. The proof of the
  continuity of $\d{}{}$ is slightly different, and can be done as
  follows:
  \begin{equation*}
    \begin{split}
      & \tau_E\scal[E]{\nabla\cdot\left(\K\nabla
          w\right)}{\beta\cdot\nabla v_h} =
      \tau_E\scal[E]{\nabla\cdot\left(\K\nabla
          w-\K\nabla \proj{k-1}{}w\right)}{\beta\cdot\nabla v_h}
      \\
      &\qquad + \tau_E\scal[E]{\nabla\cdot\left(\K\nabla
          \proj{k-1}{}w\right)}{\beta\cdot\nabla v_h}
      \\
      &\quad \leq \tau_E\norm[E]{\nabla\cdot\left(\K\nabla\left(
            w-\proj{k-1}{}w\right)\right)} \norm[E]{\beta\cdot\nabla
        v_h} + \tau_E\norm[E]{\nabla\cdot\left(\K\nabla
          \proj{k-1}{}w\right)} \norm[E]{\beta\cdot\nabla v_h}
      \\
      &\quad \leq \left(\sqrt{\tau_E} \norm[E]{ \left( \nabla\cdot\K
          \right) \nabla\left( w-\proj{k-1}{}w \right)} +
        \sqrt{\tau_E} \norm[E]{ \K \Delta \left( w-\proj{k-1}{}w
          \right)} \right.
      \\
      &\qquad \left. + \norm[E]{\sqrt{\K}\proj{k-1}{}\nabla w} \right)
      \sqrt{\tau_E}\norm[E]{\beta\cdot\nabla v_h}
      \\
      &\quad \leq \mathcal{C} \left( \sqrt{\tau_E}
        \norm[\sob{1}{\infty}{E}]{\K} \norm[2,E]{w-\proj{k-1}{}w} +
        \frac{\sqrt{\K[E]}}{\sqrt{\Kmin[E]}} \norm[E]{\sqrt{\K} \nabla
          w}\right) \ennorm[\K\beta]{w} \,.
    \end{split}
  \end{equation*}
\end{proof}

The above lemmas can be summarized in the following lemma, estimating
the error of approximation of the exact bilinear form by the discrete
bilinear form.
\begin{lemma}
  For any given $w\in\sobnc{k}{\Th} \cap \sobh{s+1}{\Th}$ and any
  $v_h\in V_h$,
  \begin{equation}
    \label{eq:estim_Bsupg-Bsupgh}
    \begin{split}
      \abs{\Bsupg{w}{v_h} - \Bsupg[h]{w}{v_h}} &\leq
      \mathcal{C}\max_{E\in\Th} \left\{
        \frac{\norm[\sob{1}{\infty}{E}]{\K}}{\sqrt{\K[E]}},
        \sqrt{h_E\beta_E}, h_E\sqrt{\gamma_E}, \right.
      \\
      &\left. \qquad \mathcal{C}^{nc}_{a,E} , \mathcal{C}^{nc}_{b,E} ,
        \mathcal{C}^{nc}_{c,E} , \mathcal{C}^{nc}_{d,E} ,
        \mathcal{K}^{nc}_{b,E} \right\} h^s \seminorm[s+1]{w}
      \ennorm[\K\beta]{v_h}\,,
    \end{split}
  \end{equation}
  where $\mathcal{C}^{nc}_{a,E}$ , $\mathcal{C}^{nc}_{b,E}$ ,
  $\mathcal{C}^{nc}_{c,E}$, $\mathcal{C}^{nc}_{d,E}$ and
  $\mathcal{K}^{nc}_{b,E}$ are defined by \eqref{eq:defCNCa},
  \eqref{eq:defCNCb}, \eqref{eq:defCNCc}, \eqref{eq:defCNCd} and
  \eqref{eq:defKNCb}.
\end{lemma}
\begin{proof}
  Collecting the results of Lemmas \ref{lem:estim_a-ah},
  \ref{lem:estim_b-bh}, \ref{lem:estim_c-ch}, \ref{lem:estim_d-dh} and
  \ref{lem:Bsupg_continuity} and the approximation estimates on the
  polynomial projections, we get
  \begin{equation*}
    \begin{split}
      &\abs{\Bsupg{w}{v_h} - \Bsupg[h]{w}{v_h}} \leq \abs{
        \Bsupg{w-\proj{k-1}{}w}{v_h} } + \abs{
        \Bsupg[h]{w-\proj{k-1}{} w}{v_h} }
      \\
      & \qquad + \abs{ \Bsupg{\proj{k-1}{}w}{v_h} -
        \Bsupg[h]{\proj{k-1}{}w}{v_h} }
      \\
      &\quad \leq \mathcal{C}\left( \max_{E\in\Th} \left(
          \frac{\sqrt{\K[E]}}{\sqrt{\Kmin[E]}}
        \right)\ennorm[\K\beta\gamma]{w-\proj{k-1}{}w}
         \right.
      \\
      & \left. \qquad + \max_{E\in\Th} \left( \sqrt{\tau_E}
          \norm[\sob{1}{\infty}{E}]{\K} \right)
        \norm[2,E]{w-\proj{k-1}{}w} + \max_{E\in\Th} \tau_E^{-\frac12}
        \norm{w - \proj{k-1}{}w} \right.
      \\
      &\qquad \left. + \max_{E\in\Th} \left\{ \mathcal{C}^{nc}_{a,E} ,
          \mathcal{C}^{nc,s+1}_{b,E} , \mathcal{C}^{nc}_{c,E} ,
          \mathcal{C}^{nc}_{d,E} , \mathcal{K}^{nc}_{b,E} \right\} h^s
        \seminorm[s+1]{w} \right) \ennorm[\K\beta]{v_h} 
      \\
      &\quad \leq \mathcal{C}\max_{E\in\Th} \left\{
        \frac{\norm[\sob{1}{\infty}{E}]{\K}}{\sqrt{\K[E]}},
        \sqrt{h_E\beta_E}, h_E\sqrt{\gamma_E}, \mathcal{C}^{nc}_{a,E}
        , \mathcal{C}^{nc,s+1}_{b,E} , \mathcal{C}^{nc}_{c,E} ,
        \mathcal{C}^{nc}_{d,E} , \mathcal{K}^{nc}_{b,E} \right\} h^s
      \seminorm[s+1]{w} \ennorm[\K\beta]{v_h} .
    \end{split}
  \end{equation*}
\end{proof}

Due to the non-conformity of our approach and since the functions in
the global virtual element space $V_h$ may be discontinuous, for the
exact solution $u\in\sobh{2}{\Th}\cap \sobh[0]{1}{\Omega}$ and every
$v_h\in V_h$ it holds that
\begin{equation*}
  \Bsupg{u}{v_h} = \Fsupg{v_h} + \Nh{u}{v_h},
\end{equation*}
where
\begin{equation}
  \label{eq:Nh:def}
  \Nh{u}{v_h} \defeq \sum_{E\in\Th} \scal[\partial E]{(\K\nabla{u})\cdot\n{}-
    \frac12 \left( \beta\cdot\n{}  \right)u}{v_h}\,.
\end{equation}
is called the \emph{conformity error}.
This term is a generalization of the one of the pure diffusion problem
that is introduced and estimated in \cite[Lemma
4.1]{AyusodeDios-Lipnikov-Manzini:2016}.
\begin{lemma}[\textbf{Conformity error}]
  \label{lemma:estim_Nh}
  Let $u\in\sobh{s+1}{\Th} \cap \sobh[0]{1}{\Omega}$, $1\leq s\leq k$,
  be the solution of the variational problem~\eqref{eq:exvarform}. Let
  $\beta\in\sob{s+1}{\infty}{\Omega}$ and suppose
  $\K\nabla u\in\sobh{}{\mathrm{div},\Omega}$.
  Under the mesh regularity assumptions of Section \ref{sec:mesh}, for
  every $v_h\in V_h$ it holds that
  \begin{equation}
    \label{eq:estim_Nh}
    \abs{\Nh{u}{v_h}}  \leq \mathcal{C}
    \max_{E\in\Th}\left\{\frac{\norm[\sob{s}{\infty}{E}]{\K}}
      {\sqrt{\Kmin[E]}} , \mathcal{K}^{nc}_{\mathcal{N},E}
    \right\} h^{s} \norm[s+1]{u} \ennorm[\K\beta\gamma]{v_h} \,,
  \end{equation}
  where
  \begin{equation}
    \label{eq:defCNCNh}
    \mathcal{K}^{nc}_{\mathcal{N},E} =  \frac{h_E
      \norm[\sob{s+\frac12}{\infty}{\partial E}]{\beta\cdot\n{}}
    }{\sqrt{\Kmin[E]}}  \,.
  \end{equation}
\end{lemma}
\begin{proof}
  The first term in \eqref{eq:Nh:def} is bounded following \cite[Lemma
  4.1]{AyusodeDios-Lipnikov-Manzini:2016}, using the fact that, by
  hypothesis, $\K\nabla u\cdot\n{}$ is continuous:
  \begin{equation*}
    \begin{split}
      \sum_{E\in\Th}\scal[\partial E]{(\K\nabla u)\cdot\n{}}{v_h} & =
      \sum_{e\in\Eh} \scal[e]{(\K\nabla u)\cdot\n{}}{\jmp{v_h}} =
      \sum_{e\in\Eh} \scal[e]{(\K\nabla u - \proj{k-1}{\K\nabla
          u})\cdot\n{}}{\jmp{v_h}}
      \\
      &\leq \sum_{e\in\Eh} \norm[e]{(\K\nabla u - \proj{k-1}{\K\nabla
          u})\cdot\n{}} \norm[e]{\jmp{v_h - \proj[0,e]{k-1}{}v_h}}
      \\
      &\leq \sum_{e\in\Eh} \mathcal{C}h_e^{s-\frac12}
      \seminorm[s,\omega_e]{\K\nabla u}\cdot h^{\frac12}_e
      \norm[\omega_e]{\nabla v_h}
      \\
      &\leq \frac{\norm[\sob{s}{\infty}{E}]{\K}}{\sqrt{\Kmin[E]}}
      h^s_E \norm[s+1,\omega_e]{u} \ennorm[\omega_e]{v_h}\,.
    \end{split}
  \end{equation*}
  The second term in \eqref{eq:Nh:def} is estimated using the fact
  that $v_h\in\sobnc{k}{\Th}$. Denoting by
  $\proj[0,\partial E]{k-1}{}v_h$ the piecewise polynomial projection
  of $v_h$ on each $e\subset \partial E$ and since
  $\jmp{\proj[0,e]{k-1}{}v_h}_{e} =
  \proj[0,e]{k-1}{\jmp{v_h}_{e}} =0$ $\forall e \subset \partial E$
  because $\int_{e}\jmp{v_h} q = 0$ $\forall q\in\Poly{k-1}{e}$, and
  since $(\beta\cdot \n{})u$ is continuous across the edges being
  $\beta$ a divergence-free vector and $u\in\sobh[0]{1}{\Omega}$, we
  get
  \begin{equation*}
    \begin{split}
      \sum_{E\in\Th}\scal[\partial E]{(\beta\cdot \n{})u}{v_h} &=
      \sum_{E\in\Th} \scal[\partial E]{ \left(\beta \cdot
          \n{}\right)u}{v_h} = \frac12 \sum_{E\in\Th} \scal[\partial
      E]{\left( \beta \cdot \n{}\right)u}{\jmp{v_h} }
      \\
      &= \frac12 \sum_{E\in\Th} \scal[\partial E]{\left( \beta \cdot
          \n{}\right)u}{\jmp{v_h - \proj[0,\partial E]{k-1}{} v_h} }
      \\
      & = \frac12 \sum_{E\in\Th} \scal[\partial E]{\left( \beta \cdot
          \n{}\right)u - \proj[0,\partial E]{k-1}{\left( \beta \cdot
            \n{}\right)u}}{\jmp{v_h - \proj[0,\partial E]{k-1}{} v_h}
      }
      \\
      & \leq \frac12 \sum_{E\in\Th} \norm[\partial E]{\left( \beta
          \cdot \n{}\right)u - \proj[0,\partial E]{k-1}{\left( \beta
            \cdot \n{}\right)u}} \norm[\partial E]{\jmp{v_h -
          \proj[0,\partial E]{k-1}{} v_h} }
      \\
      & \leq \mathcal{C} \sum_{E\in\Th} h^{s+\frac12}_E
      \seminorm[s+\frac12,\partial E]{\left(\beta \cdot \n{}\right) u}
      \cdot h_E^{\frac12} \norm[\omega_E]{\nabla v_h}
      \\
      & \leq \mathcal{C} \sum_{E\in\Th} h^{s+1}_E
      \frac{\norm[\sob{s+\frac12}{\infty}{\partial E}]{\beta \cdot
          \n{}}}{\sqrt{\Kmin[E]}} \norm[s+1,\omega_E]{u}
      \ennorm[\K\beta,\omega_E]{v_h} \,.
    \end{split}
  \end{equation*}
\end{proof}

\subsection{Well-posedness of the discrete problem}
\label{sec:coercivity}
The following theorem proves the well-posedness of the discrete
formulation.

\begin{theorem}[Coercivity of $\Bhsupg{}{}$]
  \label{lem:coercivity_Bh}
  For any $v_h\in V_h$,
  \begin{equation}
    \label{eq:coercivity_Bh}
    \Bsupg[h]{v_h}{v_h} \geq \min\left\{\frac14 ,\frac{\sigma_\ast}{2}\right\}
    \frac{1-C_\tau}{2} \ennorm[\K\beta\gamma]{v_h} \,,
  \end{equation}
  where $C_\tau$ is the constant introduced in~\eqref{eq:deftau}.
\end{theorem}
\begin{proof}
  Let $v_h\in V_h$. By definition \eqref{eq:defbh}, it holds that
  \begin{equation*}
    \b{v_h}{v_h} = \frac12\sum_{E\in\Th} 
    \left[
      \scal[E]{\beta\cdot\proj{k-1}{}\nabla v_h}{\proj{k-1}{}v_h}  -
      \scal[E]{\proj{k-1}{}v_h}{\beta\cdot\proj{k-1}{}\nabla v_h}
    \right]= 0 \,.
  \end{equation*}
  Moreover, using Cauchy-Schwarz and Young inequalities we find that
  \begin{equation*}
    \begin{split}
      \tau_E\abs{\scal[E]{\gamma \proj{k-1}{}
          v_h}{\beta\cdot\proj{k-1}{} \nabla v_h}} &\leq
      \sqrt{\gamma_E}\tau_E\norm[E]{\sqrt{\gamma} \proj{k-1}{}v_h}
      \norm[E]{\beta\cdot\proj{k-1}{}\nabla v_h}
      \\
      &\leq \frac{1}{2} \norm{\sqrt{\gamma} \proj{k-1}{}v_h}^2 +
      \frac{\gamma_E\tau_E^2}{2} \norm[E]{\beta\cdot\proj{k-1}{}\nabla
        v_h}^2\,,
    \end{split}
  \end{equation*}
  which implies that
  \begin{equation}
    \tau_E\scal[E]{\gamma\proj{k-1}{}
      v_h}{\beta\cdot\proj{k-1}{}\nabla v_h} \geq - \frac{1}{2}
    \norm{\sqrt{\gamma} \proj{k-1}{}v_h}^2 - \frac{\gamma_E\tau_E^2}{2}
    \norm[E]{\beta\cdot\proj{k-1}{}\nabla v_h}^2\,.
    \label{eq:Bsupg:coercivity:proof:00}
  \end{equation}
  Inverse inequality \eqref{eq:estimnormlapl} imply that
  \begin{equation}
    \begin{split}
      \tau_E\norm[E]{\div\left(\K \proj{k-1}{}\nabla v_h\right)}^2
      &\leq \frac{\tilde{C}_k^E h^2_E}{\K[E]} \norm[E]{\div\left(\K
          \proj{k-1}{} \nabla v_h\right)}^2 \leq
      \frac{1}{\K[E]}\norm[E]{\K\proj{k-1}{}\nabla v_h}^2
      \\
      &\leq \norm[E]{\sqrt{\K}\proj{k-1}{}\nabla v_h}^2,
      \label{eq:Bsupg:coercivity:proof:10}
    \end{split}
  \end{equation}
  since
  $\norm[E]{\K\proj{k-1}{}\nabla
    v_h}\leq\sqrt{\K[E]}\norm[E]{\sqrt{\K}\proj{k-1}{}\nabla v_h}$.
  Using the definition of $\Bsupg[h]{}{}$, cf.~\eqref{eq:defBhsupg},
  Cauchy-Schwarz inequality, inverse inequality
  \eqref{eq:estimnormlapl},
  inequalities~\eqref{eq:Bsupg:coercivity:proof:00},
  \eqref{eq:Bsupg:coercivity:proof:10} and \eqref{eq:deftau}, we have
  \begin{equation*}
    \begin{split}
      &\Bsupg[h]{v_h}{v_h} = \sum_{E\in\Th}\left\{
        \norm[E]{\sqrt{\K}\proj{k-1}{}\nabla v_h}^2 +
        \tau_E\norm[E]{\beta\cdot\proj{k-1}{}\nabla v_h}^2 \right.
      \\
      &\qquad + \shE{\left(I-\proj{k-1}{}\right)
        v_h}{\left(I-\proj{k-1}{}\right) v_h}
      \\
      &\qquad + \norm[E]{\sqrt{\gamma} \proj{k-1}{}v_h}^2 +
      \tau_E\scal[E]{\gamma
        \proj{k-1}{}v_h}{\beta\cdot\proj{k-1}{}\nabla v_h}
      \\
      &\left. \qquad - \tau_E
        \scal[E]{\div\left(\sqrt{\K}\proj{k-1}{}\nabla
            v_h\right)}{\beta\cdot\proj{k-1}{}\nabla v_h} \right\}
      \\
      &\quad \geq \sum_{E\in\Th} \left\{
        \norm[E]{\sqrt{\K}\proj{k-1}{}\nabla v_h}^2 +
        \left(1-\frac{\gamma_E\tau_E}{2}\right) \tau_E
        \norm[E]{\beta\cdot \proj{k-1}{} \nabla v_h}^2 \right.
      \\
      &\qquad+ \shE{\left(I-\proj{k-1}{}\right)
        v_h}{\left(I-\proj{k-1}{}\right) v_h} +
      \left(1-\frac{1}{2}\right)\norm[E]{\sqrt{\gamma}
        \proj{k-1}{}v_h}^2
      \\
      & \left. \qquad - \tau_E
        \norm[E]{\div\left(\sqrt{\K}\proj{k-1}{}\nabla v_h\right)}
        \norm[E]{\beta\cdot\proj{k-1}{}\nabla v_h} \right\}
      \\
      &\quad \geq \sum_{E\in\Th} \left\{
        \norm[E]{\sqrt{\K}\proj{k-1}{}\nabla v_h}^2 +
        \left(\frac{1}{2}-\frac{C_{\tau}}{2}\right) \tau_E
        \norm[E]{\beta\cdot \proj{k-1}{} \nabla v_h}^2 \right.
      \\
      &\qquad+ \shE{\left(I-\proj{k-1}{}\right)
        v_h}{\left(I-\proj{k-1}{}\right) v_h} + \frac{1}{2}
      \norm[E]{\sqrt{\gamma} \proj{k-1}{} v_h}^2
      \\
      &\left. \qquad - \sum_{E\in\Th}\frac12\tau_E\norm[E]{\div
          \left(\K\proj{k-1}{} \nabla v_h \right)}^2 \right\}
      \\
      &\quad \geq \sum_{E\in\Th}\left\{ \frac{1}{2}
        \norm[E]{\sqrt{\K}\proj{k-1}{}\nabla v_h}^2 + \left(
          \frac{1}{2}-\frac{C_{\tau}}{2} \right)
        \tau_E\norm[E]{\beta\cdot\proj{k-1}{}\nabla v_h}^2 \right.
      \\
      &\left. \qquad + \shE{\left(I-\proj{k-1}{}\right)
          v_h}{\left(I-\proj{k-1}{}\right) v_h} + \frac{1}{2}
        \norm[E]{\sqrt{\gamma} \proj{k-1}{}v_h}^2\right\}
      \\
      &\quad \geq \frac{1-C_\tau}{2} \sum_{E\in\Th} \left(
        \norm[E]{\sqrt{\K}\proj{k-1}{}\nabla v_h}^2 +
        \tau_E\norm[E]{\beta\cdot\proj{k-1}{}\nabla v_h}^2 \right.
      \\
      &\qquad \left. + \shE{\left(I-\proj{k-1}{}\right)
          v_h}{\left(I-\proj{k-1}{}\right) v_h} +
        \norm[E]{\sqrt{\gamma} \proj{k-1}{}v_h}^2 \right) \,.
    \end{split}
  \end{equation*}
  Next, using the coercivity of the VEM stabilization in
  \eqref{eq:vemstab_coercivity} we get $\forall E\in\Th$,
  \begin{equation*}
    \begin{split}
      &\norm[E]{\sqrt{\K}\proj{k-1}{}\nabla v_h}^2 +
      \tau_E\norm[E]{\beta\cdot\proj{k-1}{}\nabla v_h}^2 +
      \shE{\left(I-\proj{k-1}{}\right)
        v_h}{\left(I-\proj{k-1}{}\right) v_h}
      \\
      &\quad \geq \frac12 \a[h]{v_h}{v_h}+ \frac12
      \left(\norm[E]{\sqrt{\K}\proj{k-1}{}\nabla v_h}^2 + \tau_E
        \norm[E]{\beta\cdot\proj{k-1}{}\nabla v_h}^2 \right.
      \\
      & \qquad \left. + \sigma_\ast \left( \K[E]+\tau_E \beta_E^2
        \right) \norm[E]{\nabla v_h - \nabla \proj{k-1}{} v_h}^2
      \right)
      \\
      & \quad \geq \frac12 \a[h]{v_h}{v_h} +
      \min\left\{\frac12,\sigma_\ast\right\}
      \left(\norm[E]{\sqrt{\K}\proj{k-1}{}\nabla v_h}^2 +
        \tau_E\norm[E]{\beta\cdot\proj{k-1}{}\nabla v_h}^2 \right.
      \\
      &\qquad \left. + \left(\K[E] + \tau_E \beta_E^2\right)
        \norm[E]{\nabla v_h - \proj{k-1}{} \nabla v_h }^2 \right)
      \\
      & \quad \geq \frac12 \a[h]{v_h}{v_h} +
      \min\left\{\frac12,\sigma_\ast\right\}
      \left(\norm[E]{\sqrt{\K}\proj{k-1}{}\nabla v_h}^2 +
        \tau_E\norm[E]{\beta\cdot\proj{k-1}{}\nabla v_h}^2 \right.
      \\
      &\qquad \left. + \norm[E]{\sqrt{\K} \left( \nabla v_h -
            \proj{k-1}{} \nabla v_h \right)}^2 + \tau_E \norm[E]{\beta
          \cdot \left( \nabla v_h - \proj{k-1}{} \nabla v_h \right)
        }^2 \right)
      \\
      &\quad \geq \min\left\{ \frac{1}{4},\frac{\sigma_\ast}{2} \right\} \left(
        \a[h]{v_h}{v_h} + \a{v_h}{v_h} \right) \,.
    \end{split}
  \end{equation*}
  In the last line we use the following inequalities:
  \begin{align*}
    \norm[E]{\sqrt{\K}\proj{k-1}{}\nabla v_h}^2
    + \norm[E]{\sqrt{\K}\left( \nabla v_h - \proj{k-1}{} \nabla v_h \right)
    }^2  &\geq \frac12\norm[E]{\sqrt{\K}\nabla v_h}^2\,,
    \\
    \tau_E\left( \norm[E]{\beta\cdot\proj{k-1}{}\nabla v_h}^2 + \norm[E]{\beta
    \cdot \left( \nabla v_h - \proj{k-1}{} \nabla v_h \right)
    }^2 \right) &\geq \frac12\norm[E]{\beta\cdot\nabla v_h}^2 \,.
  \end{align*}
\end{proof}

\subsection{A priori error estimates}
\label{sec:aprior_estim}

Here, we prove the a priori error estimates showing that the
stabilized formulation of the problem has optimal rates of
convergence.
Several constants in the error inequalities are numbered to track
their dependence on the local problem coefficients.
\begin{theorem}
  \label{teo:apriori}
  Let $u\in\sobh{s+1}{\Th}\cap\sobh[0]{1}{\Omega}$,
  $2\leq s+1\leq k+1$, be the solution of the variational
  problem~\eqref{eq:exvarform:supg} with $f\in\sobh{s-1}{\Omega}$,
  $\K\in\big[\sob{s+1}{\infty}{\Omega}\big]^{d\times d}$,
  $\beta\in\left[\sob{s+1}{\infty}{\Omega}\right]^{d}$ and
  $\gamma\in\sob{s+1}{\infty}{\Omega}$.
  Let $u_h\in V_h$ be the solution of the
  VEM~\eqref{eq:vemvarform_supg} under the mesh assumption of
  Section~\ref{sec:mesh}.
  Then, for $h$ sufficiently small, it holds
  \begin{equation}
    \label{eq:apriori_estim_ennorm}
    \begin{split}
      \ennorm[\K\beta\gamma]{u-u_h} &\leq \mathcal{C} h^s \left\{
        \max_{E\in\Th} \left(
          \frac{\norm[\sob{s}{\infty}{E}]{\K}}{\sqrt{\Kmin[E]}} ,
          \sqrt{h_E\beta_E} , h_E\sqrt{\gamma_E} ,
          \mathcal{C}^{nc}_{a,E} , \mathcal{C}^{nc,s+1}_{b,E} ,
          \mathcal{C}^{nc}_{c,E} , \mathcal{C}^{nc}_{d,E} ,
          \mathcal{K}^{nc}_{\mathcal{N},E} ,\right. \right.
      \\
      &\qquad \left. \left.
          \vphantom{\frac{\norm[\sob{s}{\infty}{E}]{\K}}{\sqrt{\Kmin[E]}}}
          \mathcal{K}^{nc}_{b,E} \right) \norm[s+1]{u} +
        \max_{E\in\Th} \mathcal{C}^{nc}_{f,E} \right\}
      \,,
    \end{split}
  \end{equation}
  where $\mathcal{C}^{nc}_{a,E}$, $\mathcal{C}^{nc,s+1}_{b,E}$,
  $\mathcal{C}^{nc}_{c,E}$, $\mathcal{C}^{nc}_{d,E}$,
  $\mathcal{K}^{nc}_{\mathcal{N},E}$ and $\mathcal{K}^{nc}_{b,E}$ are
  defined by \eqref{eq:defCNCa}, \eqref{eq:defCNCb},
  \eqref{eq:defCNCc}, \eqref{eq:defCNCd}, \eqref{eq:defCNCNh} and
  \eqref{eq:defKNCb} respectively, and
  \begin{equation}
    \label{eq:defCNCf}
    \mathcal{C}^{nc}_{f,E} =  \max \left\{
      \frac{\seminorm[s-1,E]{f - \proj{0}{}f}}{\sqrt{\Kmin[E]}} ,
      \frac{h^{-1}_E\tau_E \seminorm[s-1,E]{f\beta - \proj{0}{f\beta}}}
      {\sqrt{\Kmin[E]}} \right\} \,.
  \end{equation}
\end{theorem}
\begin{proof}
  \noindent
  First, by using the triangle inequality we have
  \begin{equation*}
    \ennorm[\K\beta\gamma]{u-u_h} \leq \ennorm[\K\beta\gamma]{u-u_I}
    + \ennorm[\K\beta\gamma]{u_h - u_I} \,.
  \end{equation*}
  The first term is bounded using~\eqref{eq:H1:apriori:proof:00} with
  $\psi=u$.
  We are left to estimate the norm of $e_h\defeq u_h-u_I$.
  Since $e_h\in V_h$, by \eqref{eq:coercivity_Bh} we know that
  \begin{equation}
    \begin{split}
      \alpha \ennorm[\K\beta\gamma]{e_h}^2 &\leq \Bhsupg{u_h -
        u_I}{e_h} = \Fsupg[h]{e_h} - \Bhsupg{u_I}{e_h}
      \\
      & = \Fsupg[h]{e_h} - \Fsupg{e_h} - \Nh{u}{e_h} + \Bsupg{u}{e_h}
      - \Bhsupg{u_I}{e_h}
      \\
      & \leq \abs{\Fsupg[h]{e_h} - \Fsupg{e_h}} + \abs{\Nh{u}{e_h}} +
      \abs{\Bsupg[h]{u-u_I}{e_h}}
      \\
      & \quad + \abs{ \Bsupg{u}{e_h} - \Bsupg[h]{u}{e_h} } \,.
      \label{eq:apriori_firsteqn}
    \end{split}
  \end{equation}
  We estimate the first term as follows:
  \begin{equation*}
    \label{eq:apriori_Fh-F}
    \begin{split}
      &\sum_{E\in\Th}\abs{\scal[E]{f}{e_h - \proj{k-1}{}e_h} +
        \tau_E\scal[E]{f}{\beta\cdot\left(\nabla e_h -
            \proj{k-1}{}\nabla e_h\right)}}
      \\
      &\quad = \sum_{E\in\Th}\abs{\scal[E]{f - \proj{0}{}f}{e_h -
          \proj{k-1}{}e_h}} + \tau_E \abs{\scal[E]{f\beta -
          \proj{0}{f\beta}}{\nabla e_h}}
      \\
      &\quad = \sum_{E\in\Th}\abs{\scal[E]{f - \proj{0}{}f -
          \proj{k-1}{f - \proj{0}{}f}}{e_h - \proj{k-1}{}e_h}}
      \\
      &\qquad + \tau_E \abs{\scal[E]{f\beta - \proj{0}{f\beta} -
          \proj{k-1}{f\beta - \proj{0}{f\beta}}}{\nabla e_h}}
      \\
      &\quad \leq \sum_{E\in\Th}\norm[E]{f - \proj{0}{}f -
        \proj{k-1}{f - \proj{0}{}f}} \norm[E]{e_h - \proj{k-1}{}e_h}
      \\
      &\qquad + \tau_E \norm[E]{f\beta - \proj{0}{f\beta} -
        \proj{k-1}{f\beta - \proj{0}{f\beta}}} \norm[E]{\nabla e_h}
      \\
      &\quad \leq \mathcal{C} h^s \sum_{E\in\Th}
      \left(\seminorm[s-1,E]{f - \proj{0}{}f } + h^{-1}_E\tau_E
        \seminorm[s-1,E]{f\beta - \proj{0}{f\beta}} \right)
      \norm[E]{\nabla e_h}
      \\
      &\quad \leq \mathcal{C} h^s \sum_{E\in\Th} \frac{
        \seminorm[s-1,E]{f - \proj{0}{}f} + h^{-1}_E\tau_E
        \seminorm[s-1,E]{f\beta - \proj{0}{f\beta}} }{\sqrt{\Kmin[E]}}
      \ennorm[\K\beta,E]{e_h} \,.
    \end{split}
  \end{equation*}
  Using the continuity estimate \eqref{eq:ah_continuity} to bound
  $\a[h]{}{}$, \eqref{eq:bh_continuity} to bound $\b[h]{}{}$,
  \eqref{eq:ch_continuity} to bound $\c[h]{}{}$,
  \eqref{eq:dh_continuity} to bound $\d[h]{}{}$, and the estimate of
  the VEM interpolant~\eqref{eq:H1:apriori:proof:00}, we estimate the
  third term as follows:
  \begin{equation*}
    \begin{split}
      \abs{\Bhsupg{u-u_I}{e_h}} &\leq \mathcal{C} h^s \left\{
        \ennorm[\K\beta\gamma]{u-u_I} \ennorm[\K\beta\gamma]{e_h} +
        \left(\max_{E\in\Th} \tau_E^{-\frac12} \norm{u-u_I}
        \right. \right.
      \\
      &\quad \left.\left.\vphantom{\max_{E\in\Th} \tau_E^{-\frac12}} +
          \max_{E\in\Th} \left(\mathcal{C}^{nc,1}_{b,E} +
            \mathcal{K}^{nc}_{b,E} \right) \norm{\nabla (u-u_I)
          }\right) \ennorm[\K\beta]{e_h} \right\}
      \\
      &\leq \mathcal{C} h^s \max_{E\in\Th} \left\{ \sqrt{\K[E]},
        \sqrt{h_E\beta_E}, h_E\sqrt{\gamma_E}, \mathcal{C}^{nc,1}_{b,E}
        , \mathcal{K}^{nc}_{b,E} \right\} \norm[s+1]{u}
      \ennorm[\K\beta\gamma]{e_h} \,.
    \end{split}
  \end{equation*}
  The proof of \eqref{eq:apriori_estim_ennorm} is concluded by using
  the above estimates, the estimate \eqref{eq:estim_Nh} on the
  non-conformity term and \eqref{eq:estim_Bsupg-Bsupgh}.
\end{proof}
\begin{remark}
  The second argument of the max in \eqref{eq:defCNCf} can be
  estimated as follows, using \eqref{eq:estim_ab-Pab}:
  \begin{equation*}
    \begin{split}
      \frac{h^{-1}_E\tau_E \seminorm[s-1,E]{f\beta -
          \proj{0}{f\beta}}}{\sqrt{\Kmin[E]}} \leq \mathcal{C}\frac{
        \norm[\sob{s-1}{\infty}{E}]{\beta - \proj{0}{}\beta}
        \norm[s-1,E]{f} + \norm[\sob{s-1}{\infty}{E}]{\beta}
        \norm[s-1,E]{f-\proj{0}{}f}}{\beta_E \sqrt{\Kmin[E]}} \,.
    \end{split}
  \end{equation*}
\end{remark}
\begin{remark}
  When we consider constant coefficients and a costant right-hand side
  all the non-consistency terms in \eqref{eq:apriori_estim_ennorm}
  vanish, yielding the following estimate:
  \begin{equation*}
    \begin{split}
      \ennorm[\K\beta\gamma]{u-u_h} &\leq \mathcal{C} h^s
      \max_{E\in\Th} \left( \sqrt{\K[E]} \,, \sqrt{h_E\beta_E} \,,
        h_E\sqrt{\gamma_E} , \mathcal{K}^{nc}_{\mathcal{N},E},
        \mathcal{K}^{nc}_{b,E} \right) \norm[s+1]{u} \,.
    \end{split}
  \end{equation*}
  Moreover, if we consider a conforming discretization, Theorem
  \ref{teo:apriori} proves a robust estimate with respect to the
  P\'eclet number:
  \begin{equation*}
    \ennorm[\K\beta\gamma]{u-u_h} \leq \mathcal{C} h^s
    \max_{E\in\Th} \left( \sqrt{\K[E]}\,,
      \sqrt{h_E\beta_E} \,, h_E\sqrt{\gamma_E} \right)
    \norm[s+1]{u} \,,
  \end{equation*}
  as obtained for classical Finite Elements.
\end{remark}

\section{Numerical Results}
\label{sec:numerics}

\newcommand{\MeshONE} {$\mathcal{M}_1$} \newcommand{\MeshTWO}
{$\mathcal{M}_2$} \newcommand{\MeshTHREE}{$\mathcal{M}_3$}
\newcommand{\MeshFOUR} {$\mathcal{M}_4$}
\newcommand{\HAT}[1]{\widehat{#1}}

\newcommand{\ilev}{n} \newcommand{\nR} {N_{P}} \newcommand{\nF}
{N_{F}} \newcommand{\nV} {N_{V}} \newcommand{\ndof} {\#\chi}
\newcommand{\hmax} {h_{\text{max}}}

% \subsection{Accuracy test}

The numerical experiments of this section are aimed at confirming the
convergence rates predicted by the \emph{a priori} analysis developed
in the previous sections and comparing the performance of the
nonconforming VEM with that of the conforming VEM.
In a preliminary stage, the consistency of the numerical method,
i.e. the exactness of these methods for polynomial solutions, has been
tested numerically by solving the elliptic equation with boundary and
source data determined by the monomials $u(x,y)=x^{\mu}y^{\nu}$ on
different set of polygonal meshes and for all possible combinations of
nonnegative integers $\mu$ and $\nu$ such that $\mu+\nu\leq k$, with
$k=1,2,3$.
In all the cases, the error magnitude was within the arithmetic
precision, thus confirming the consistency of the VEM.

To study the accuracy of the method we solve the
convection-reaction-diffusion equation on the domain
$\Omega=]0,1[\times]0,1[$.
The variable coefficients of the equation are given by
\begin{align}
  \label{eq:accuracy_test-K}
  \K(x,y) &=
            \begin{array}{l}
              \alpha
              \left[
              \begin{array}{cc}
                1+x^2 & xy   \\
                xy    & 1+y^2
              \end{array}
                        \right],
                        \quad\alpha=\,10^{-7}
            \end{array},
  \\[1em]
  \label{eq:accuracy_test-beta}
  \beta(x,y) &= \big(\cos(2\pi x),\sin(2\pi y)\big)^T,
  \\[1.em]
  \gamma(x,y)          &= \exp(x+y).
                         \label{eq:accuracy:test}
\end{align}
Since the P\'{e}clet number here is in the range
$\big[10^6,10^7\big]$, all calculations are in the convection
dominated regime.
The forcing term and the Dirichlet boundary conditions are set such
that the exact solution is
\begin{align}
  u(x,y) = \sin(2\pi x)\sin(2\pi y)+ \,x^5+\,y^5+1.
\end{align}

The performances of the methods presented above are investigated by
evaluating the rate of convergence on four different sequences of
unstructured meshes, labeled by~\MeshONE{}, \MeshTWO{}, \MeshTHREE{},
and \MeshFOUR{} respectively.
The top panels of Fig.~\ref{fig:Meshes} show the first mesh of each
sequence and the bottom panels show the mesh of the first refinement.
\begin{figure}
  \centering
  \begin{tabular}{cccc}
    % \includegraphics[scale=0.135]{./figures/PDF-Mesh/M400-mesh2D_0.pdf} &
    % \includegraphics[scale=0.135]{./figures/PDF-Mesh/Md201-mesh2D_0.pdf}&
    % \includegraphics[scale=0.135]{./figures/PDF-Mesh/M103-mesh2D_0.pdf} &
    % \includegraphics[scale=0.135]{./figures/PDF-Mesh/M901-mesh2D_0.pdf} \\[1em]
    % \includegraphics[scale=0.135]{./figures/PDF-Mesh/M400-mesh2D_1.pdf} &
    % \includegraphics[scale=0.135]{./figures/PDF-Mesh/Md201-mesh2D_1.pdf}&
    % \includegraphics[scale=0.135]{./figures/PDF-Mesh/M103-mesh2D_1.pdf} &
    % \includegraphics[scale=0.135]{./figures/PDF-Mesh/M901-mesh2D_1.pdf}
    %% ------------------------------------------------------------------
    \includegraphics[scale=0.135]{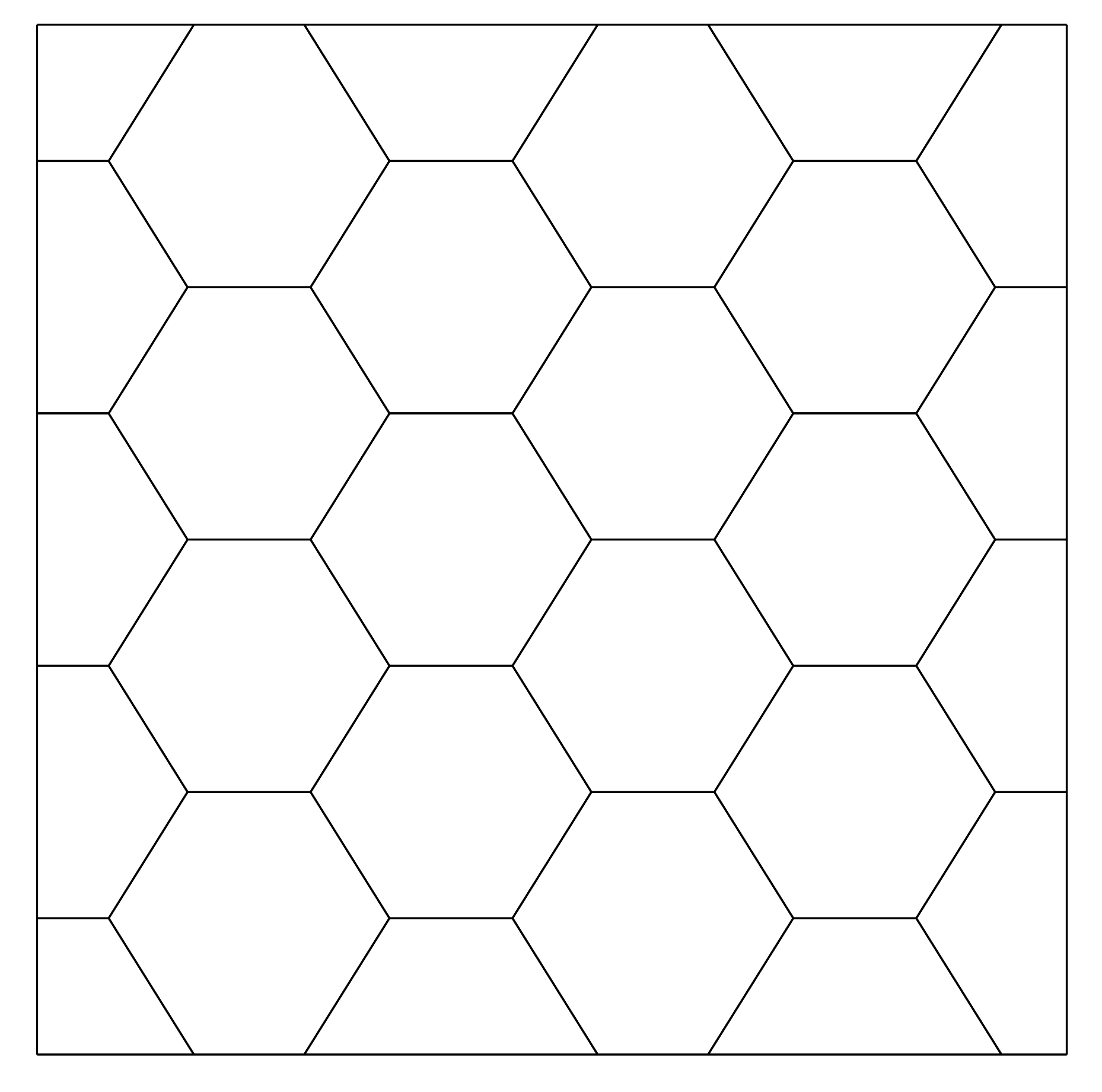} &
    \includegraphics[scale=0.135]{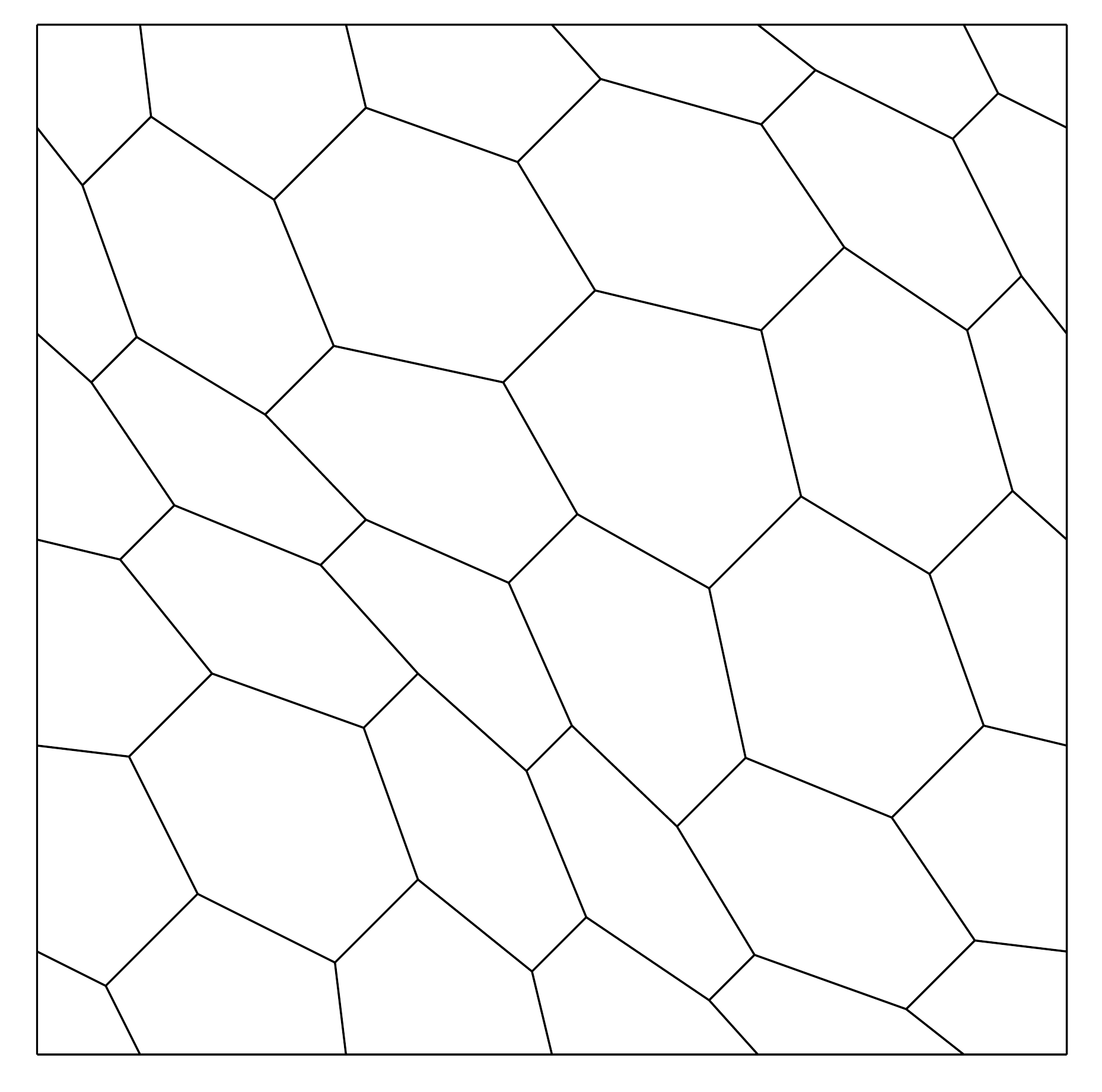}&
    \includegraphics[scale=0.135]{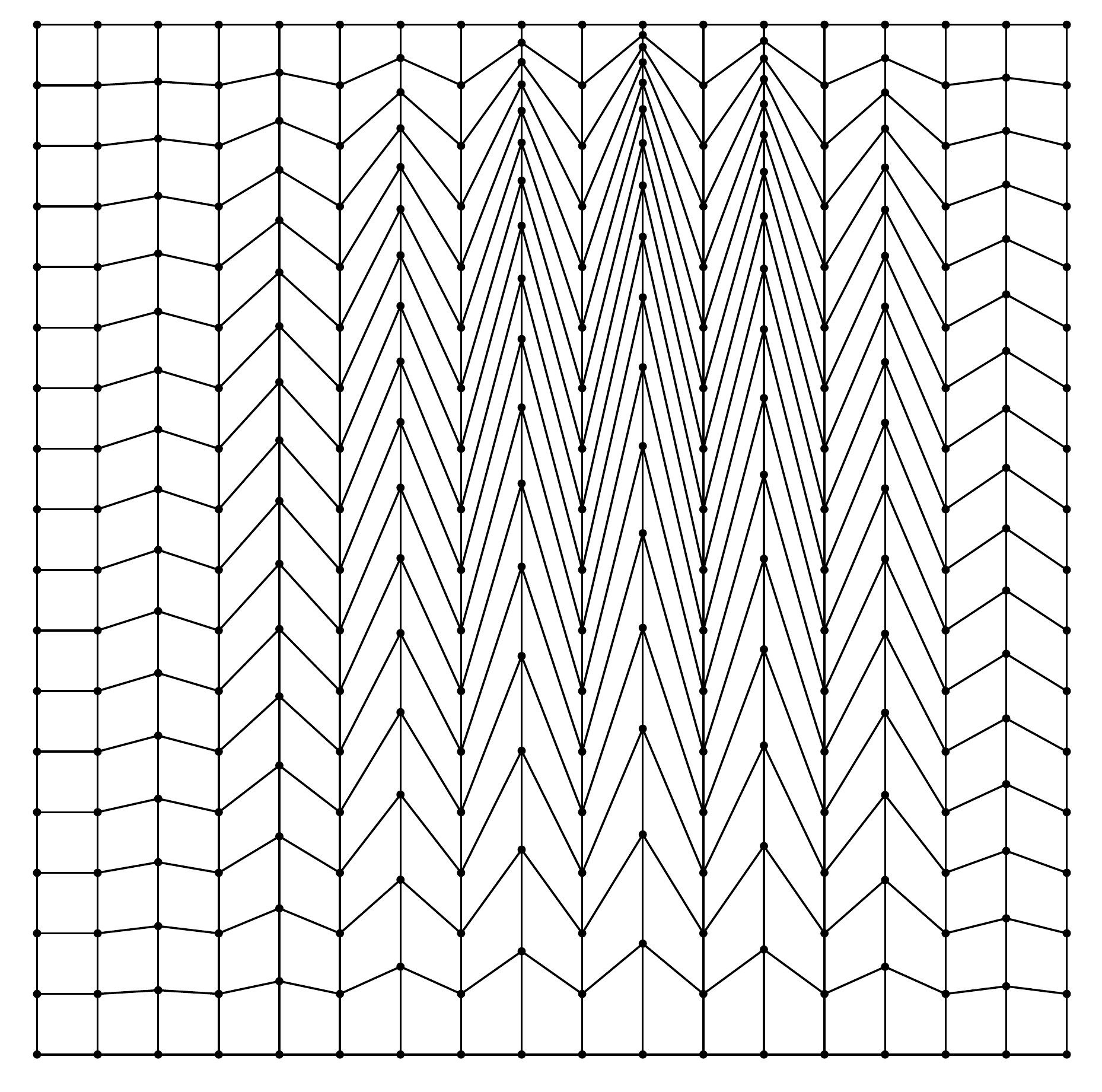} &
    \includegraphics[scale=0.135]{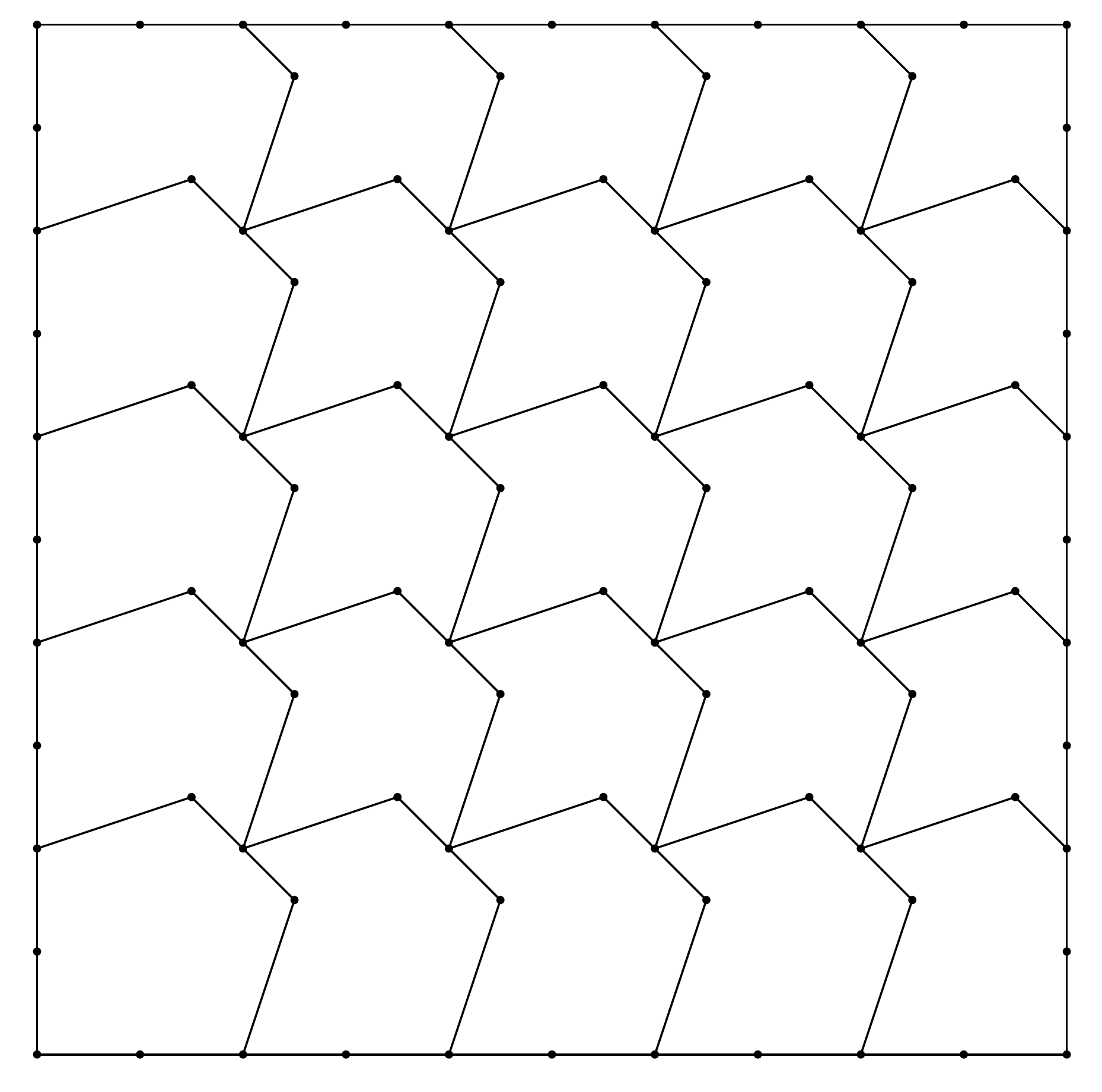} \\[1em]
    \includegraphics[scale=0.135]{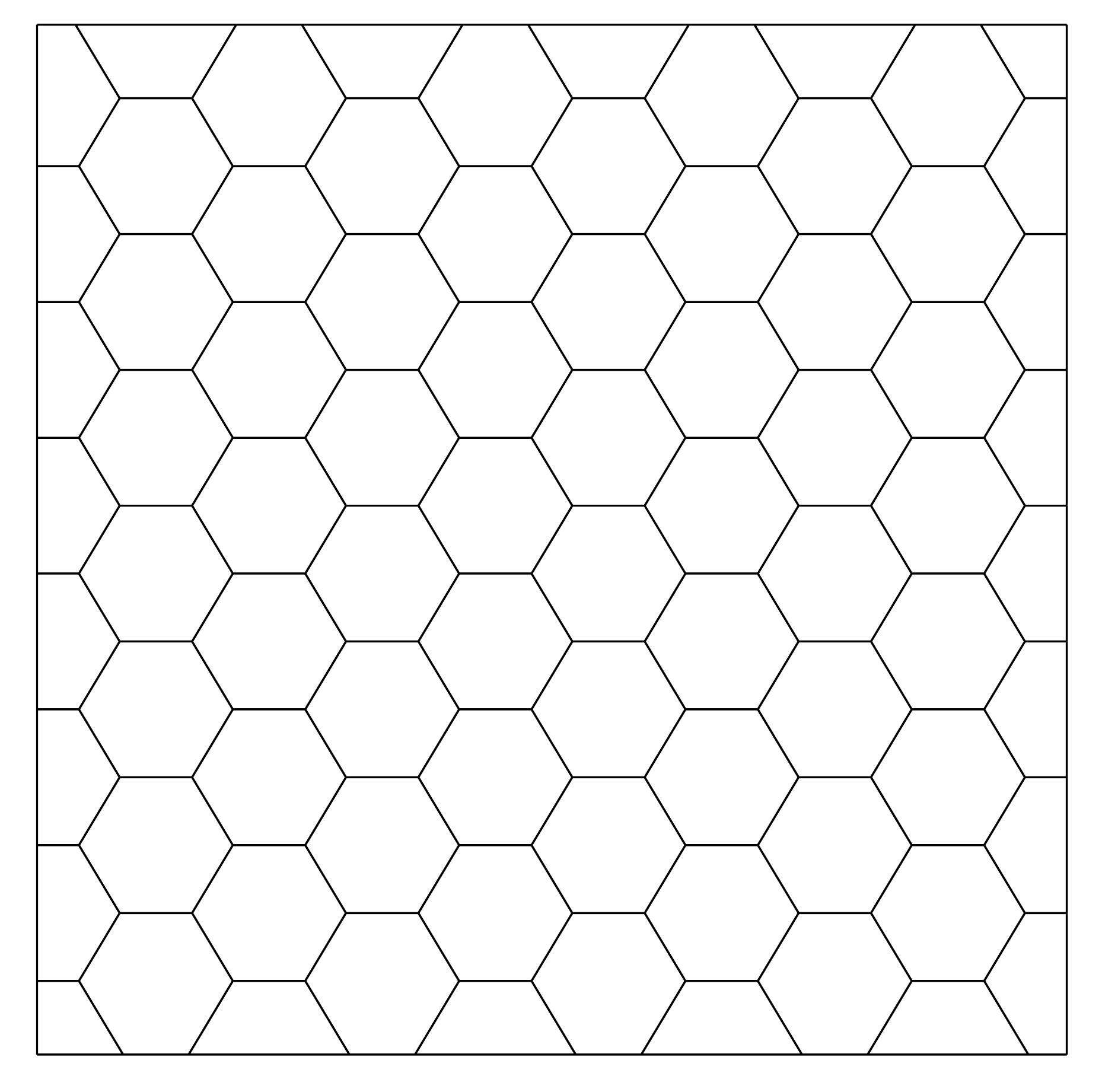} &
    \includegraphics[scale=0.135]{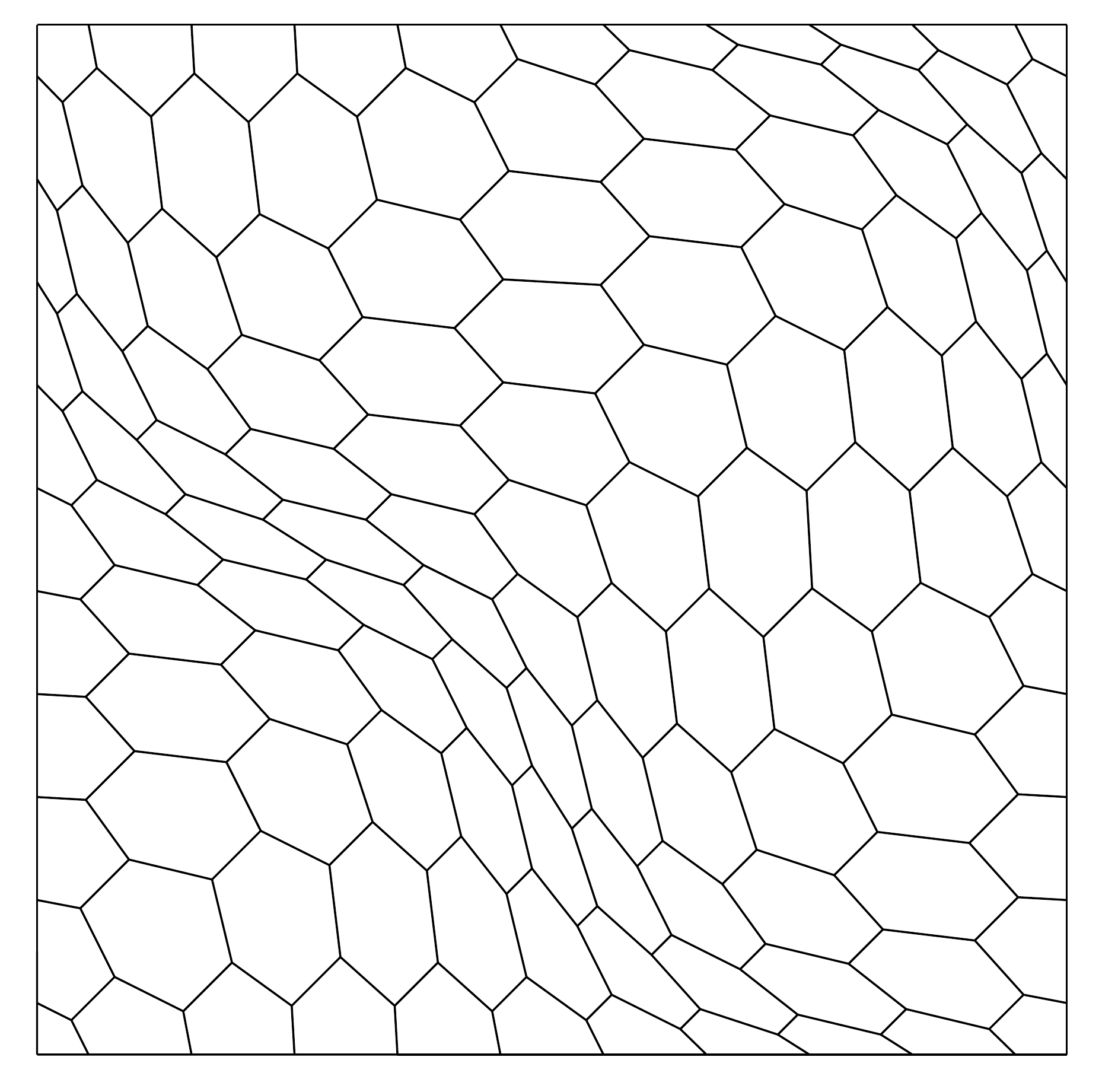}&
    \includegraphics[scale=0.135]{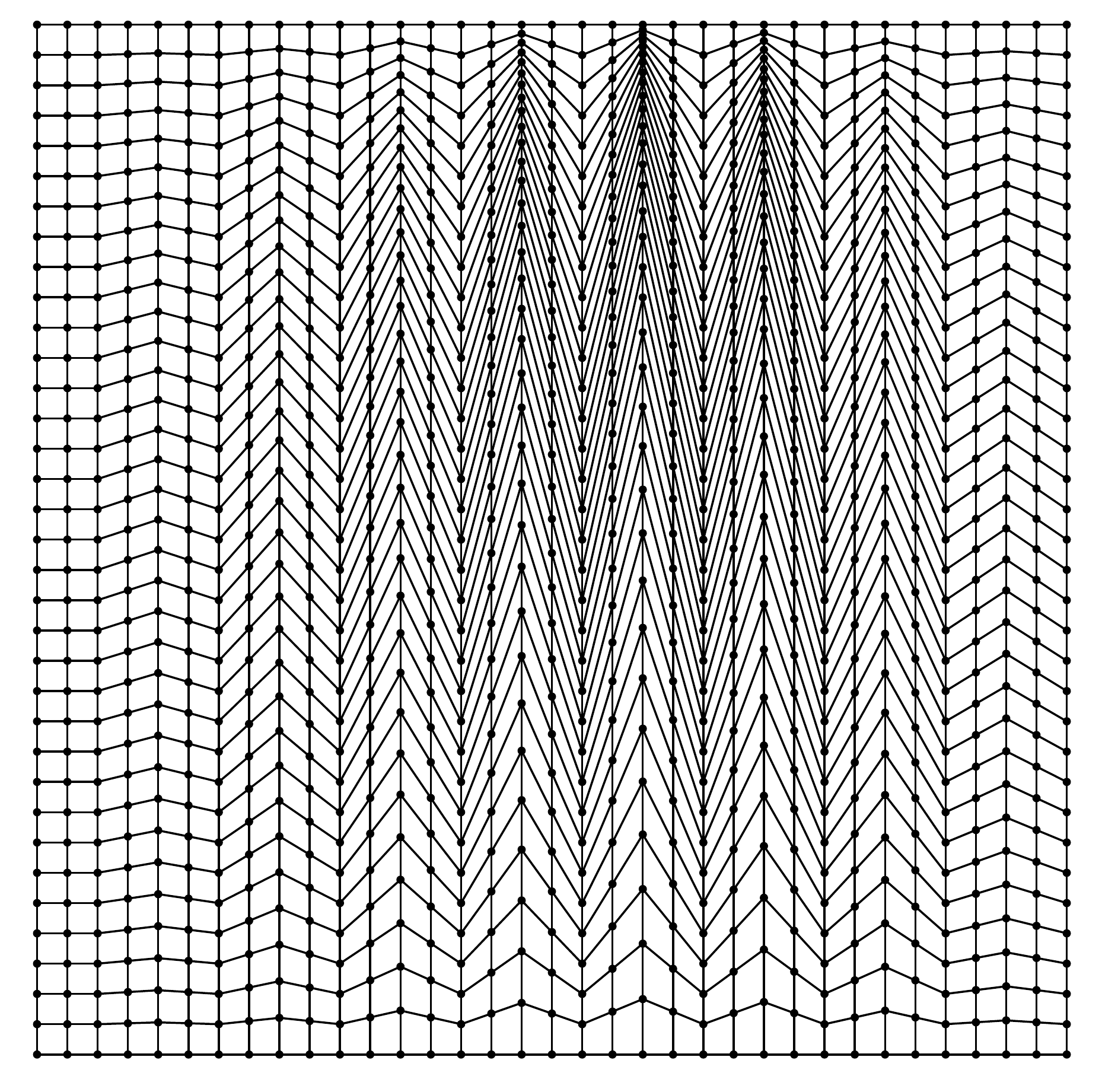} &
    \includegraphics[scale=0.135]{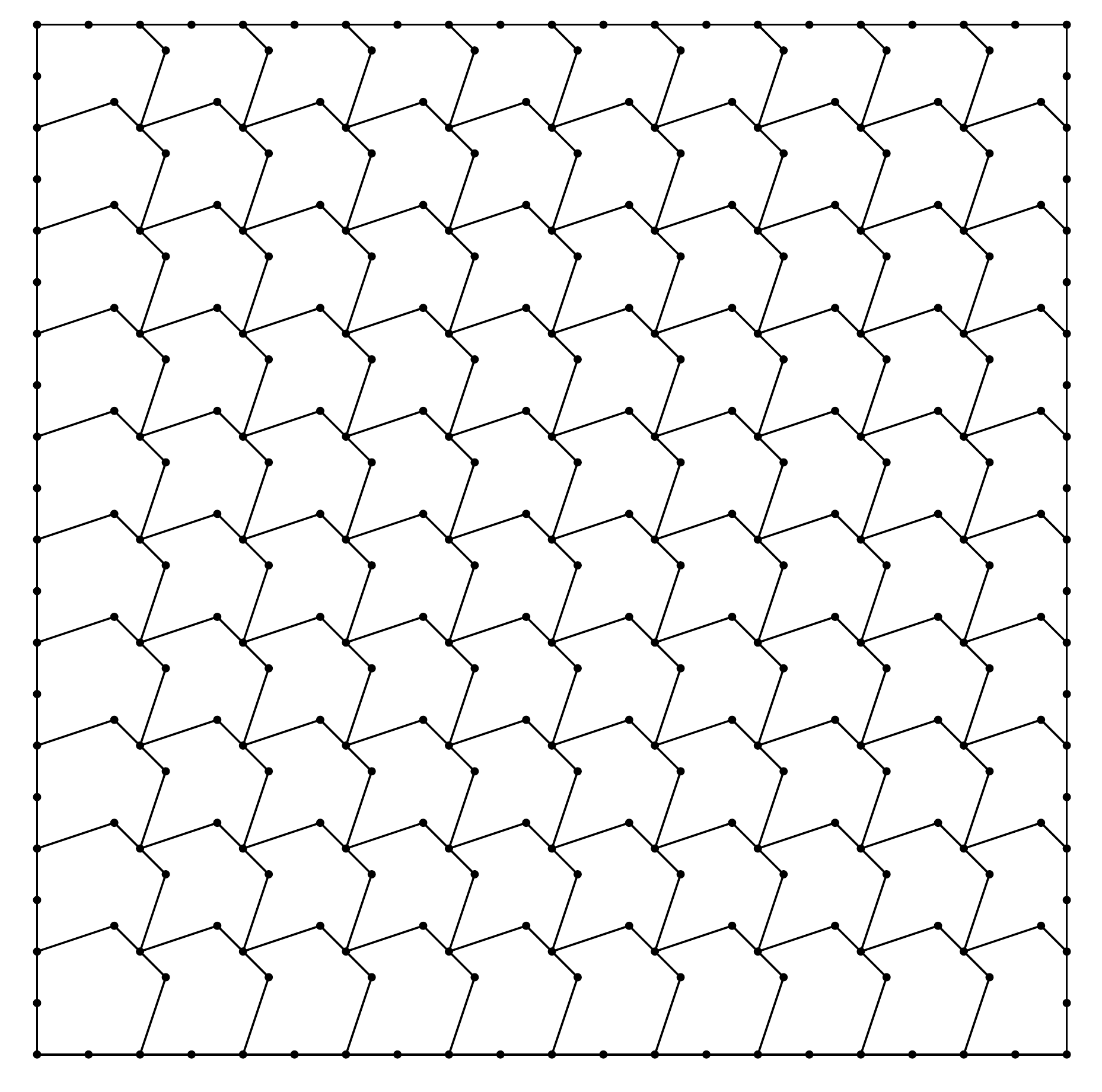}
    \\[1em]
    \text{(a)} & \text{(b)} & \text{(c)} & \text{(d)}
  \end{tabular}
  \caption{Base mesh (top row) and first refinement (bottom row) of
    the four mesh families: (a) regular hexagonal mesh; (b) remapped
    hexagonal mesh; (c) highly distorted quadrilateral mesh; (d)
    non-convex regular mesh. }
  \label{fig:Meshes}
  \vspace{-0.25cm}
\end{figure}

The meshes in \MeshONE{} are built by partitioning the domain $\Omega$
into regular hexagonal cells.
At the boundaries of $\Omega$ each mesh is completed by half hexagonal
cells.
The meshes in \MeshTWO{} are built as follows.
First, we determine a primal mesh by remapping the position
$(\HAT{x},\HAT{y})$ of the nodes of a uniform square partition of
$\Omega$ by the smooth coordinate
transformation: %%~\cite{Kuznetsov-Lipnikov-Shashkov:2004}:
\begin{align*}
  x &= \HAT{x} + (1\slash{10}) \sin(2\pi\HAT{x})\sin(2\pi\HAT{y}),\\
  y &= \HAT{y} + (1\slash{10}) \sin(2\pi\HAT{x})\sin(2\pi\HAT{y}).
\end{align*}
The corresponding mesh of \MeshTWO{} is built from the primal mesh by
splitting each quadrilateral cell into two triangles and connecting
the barycenters of adjacent triangular cells by a straight segment.
The mesh construction is completed at the boundary by connecting the
barycenters of the triangular cells close to the boundary to the
midpoints of the boundary edges and these latters to the boundary
vertices of the primal mesh.
The meshes in \MeshTHREE{} are taken from the mesh suites of the
FVCA-6
Benchmark~\cite{Eymard-Henri-Herbin-Hubert-Klofkorn-Manzini:2011:FVCA6:benchmark:proc-peer},
and are formed by highly skewed quadrilateral cells.
The meshes in \MeshFOUR{} are obtained by filling $\Omega$ with a
suitably scaled non-convex octagonal reference cell.

All the meshes are parametrised by the number of partitions in each
direction.
The starting mesh of every sequence is built from a $5\times 5$
regular grid, and the refined meshes are obtained by doubling this
resolution.

All errors, computed as in
\cite{Benedetto-Berrone-Borio-Pieraccini-Scialo:2016a,Cangiani-Manzini-Sutton:2016},
are reported in Figs.~\ref{fig:M400}, \ref{fig:Md201}, \ref{fig:M103},
and~\ref{fig:M901}.
Error values are labeled by a circle for the nonconforming VEM and by
a square for the conforming VEM, that are stabilized by the method
developed in \cite{Benedetto-Berrone-Borio-Pieraccini-Scialo:2016a}.
Each figure shows the relative errors with respect to the maximum
diameter of the discretization, in the $L^2$ norms (left panel) and in
the $H^1$ norms (right panel).
In the same figures we report the slopes $k+1$ for the $\lebl{}$-norm
and $k$ for the $\sobh{1}{}$-norm.
The numerical results confirm the theoretical rate of convergence for
the $\sobh{1}{}$-norm.
The conforming and nonconforming VEMs provide very close results on
any fixed mesh, with the conforming method slightly over performing
the nonconforming VEM in few cases.

To test the robustness of the approach with respect to very large
P\'eclet numbers, we have performed some tests with values of $\K$ and
$\beta$ in the form of \eqref{eq:accuracy_test-K} and
\eqref{eq:accuracy_test-beta}, with $\alpha$ spanning a wide range of
orders of magnitude
($\alpha\in\left\{10^{-i}\colon i=4,\ldots,11\right\}$), with
$\gamma(x,y)=0$. In Figure \ref{fig:H1error_smallK} we display the
$\sobh{1}{}$ approximation error plotted with respect to the values of
$\alpha$, on two of the meshes previously used. We can see that, as
far as the presented tests are concerned, the error is bounded
independently of the values of $\alpha$, even on non convex polygons,
thus confirming the robustness of the approach.

\subsection{Approximation of internal and boundary layers}

The second test is the classic problem
from~\cite{Franca-Frey-Hughes:1992}.
The computational domain and the boundary conditions are as shown in
Figure~\ref{fig:test-2:domain}.
The velocity forms an angle $\theta$ with the x-axis, and propagates
the non-homogeneous boundary condition $u=1$ inside $\Omega$, thus
generating an internal discontinuity, which is numerically
approximated by an internal layer, a sharp transition between the
constant solution states $u=0$ and $1$.
The homogeneous boundary condition at the top of the computational
domain produces a boundary layer.
The diffusion coefficient is constant on $\Omega$ and given by
$\K=10^{-6}$, while the velocity is
$\beta=(\cos\,\theta,\sin\,\theta)$, and $\theta=\arctan(1)$.
The P\'eclet number is about $10^6$.
We solve this problem using the remapped and the regular hexagonal
meshes (see plots $(a)-(b)$ of Figure~\ref{fig:Meshes}), with
resolution $40\times 40$ (third refinement).
Figures~\ref{fig:test-case-2b} and~\ref{fig:test-case-2c} show the
results obtained with the conforming VEM
\cite{Benedetto-Berrone-Borio-Pieraccini-Scialo:2016a} (left panels)
and the nonconfoming VEM (right panels) for the polynomial degrees
$k=1$ and $k=3$.

The results are quite similar to those presented
in~\cite{Franca-Frey-Hughes:1992,%
  Benedetto-Berrone-Borio-Pieraccini-Scialo:2016a}, and are coherent
with the expected behaviour of the method.
Undershoots and overshoots are present near the internal layer, as is
normal for this problem.
However, by increasing the accuracy order of the VEM, the numerical
solution becomes smoother.
A thorough inspection of these plots also reveals that the
nonconforming VEM tends to provide a sharper internal layer than that
of the conforming VEM at the price of a relatively bigger amplitude of
the spurious oscillations in the transition region.

%%%%%%%%%%%%%%%%%%%%%%%%%%%%%%%%%%%%%%%%%%%%%%%%%%%%%%%%%%%%%%%%%%%%%%%%%%%%%%%%%%%%
%%
%% convergence rates calculations
%%
%%%%%%%%%%%%%%%%%%%%%%%%%%%%%%%%%%%%%%%%%%%%%%%%%%%%%%%%%%%%%%%%%%%%%%%%%%%%%%%%%%%%

\begin{figure}[t]
  \hfill
  \begin{tabular}{cc}
    \begin{overpic}[width=0.40\textwidth]{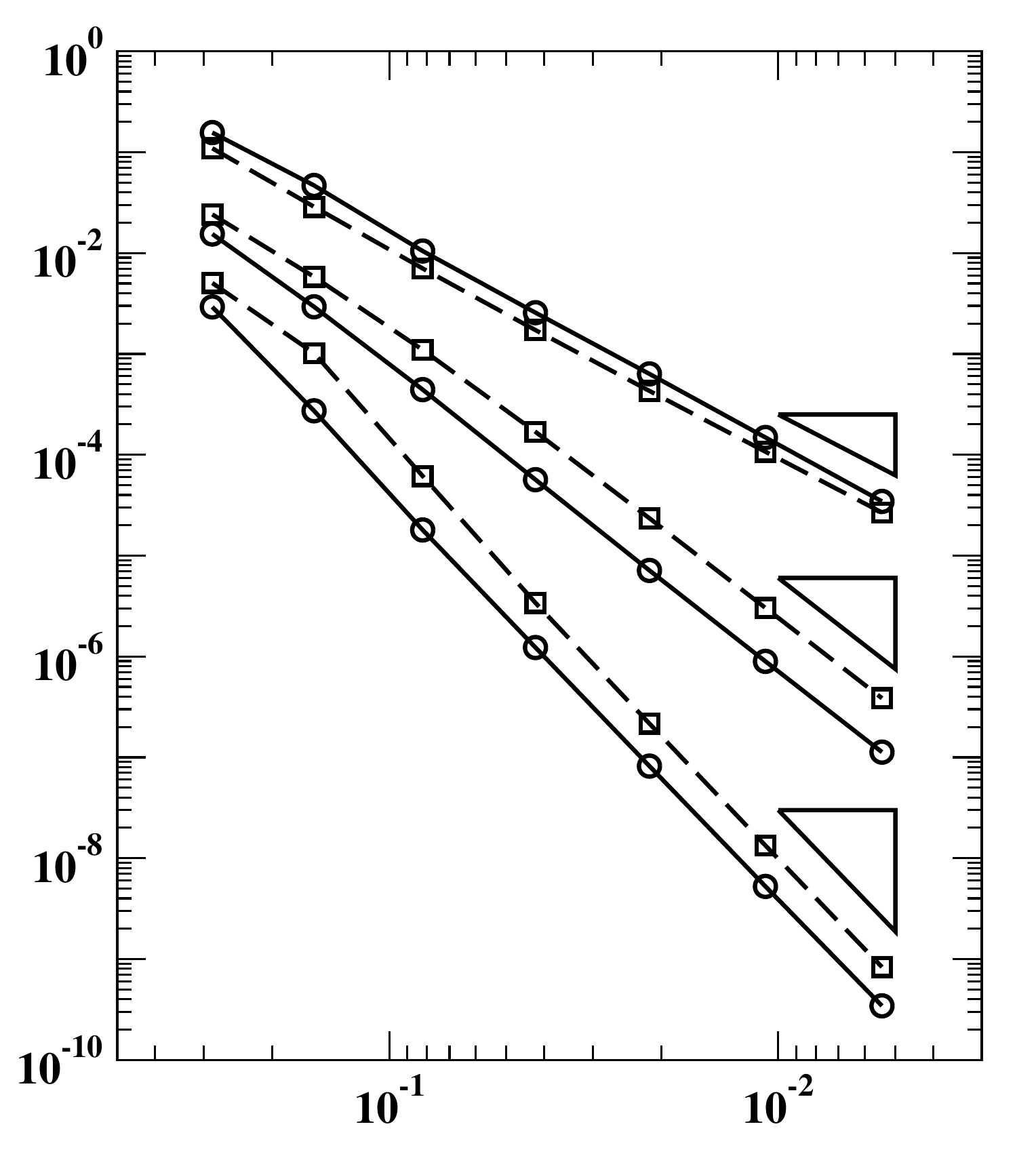}%{./figures/PDF-M400/res-M400-err_L2.pdf}
      \put(31,-5){\textbf{Mesh size $h$}}
      \put(-5,13){\begin{sideways}\textbf{ $L^2$ Approximation
            Error}\end{sideways}} \put(71,66){\textbf{2}}
      \put(79,44){\textbf{3}} \put(79,23){\textbf{4}}
    \end{overpic}
    &\quad
    \begin{overpic}[width=0.40\textwidth]{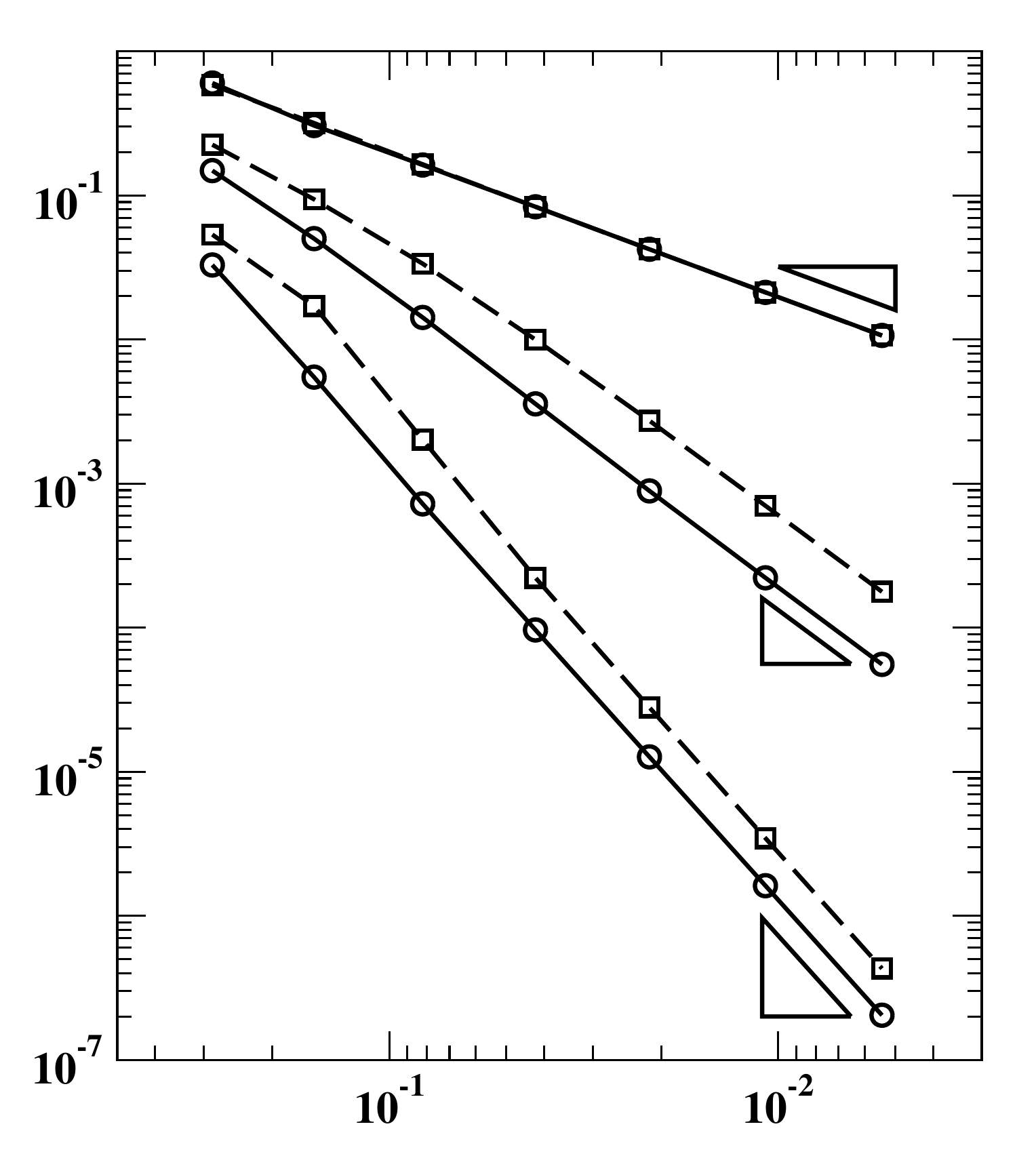}%{./figures/PDF-M400/res-M400-err_H1.pdf}
      \put(31,-5){\textbf{Mesh size $h$}}
      \put(-5,13){\begin{sideways}\textbf{ $H^1$ Approximation
            Error}\end{sideways}} \put(71,78.5){$\mathbf{1}$}
      \put(71,54 ){$\mathbf{2}$} \put(71,28 ){$\mathbf{3}$}
    \end{overpic}
  \end{tabular}
  %% Mesh M400
  \caption{Relative approximation errors obtained using the conforming VEM (dashed lines 
  labeled with squares) and the nonconforming VEM (solid lines labeled with circles)
  for $k=1,2,3$ (from top to bottom).
  Calculations are carried out using the regular hexagonal meshes of 
  Figure~\ref{fig:Meshes}(a).
  Errors are measured in the $L^2$ norm (left panels) and $H^1$ norm
  (right panels), and plotted versus $h$.}
%% (top panels) and the number of degrees of freedom (right panels). }
\label{fig:M400}
\vspace{-0.25cm}
\end{figure}
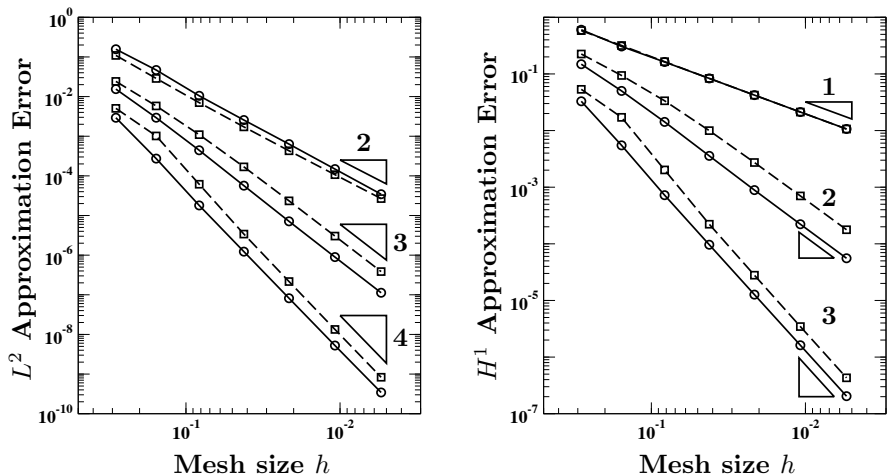

%% ================================================================================

\begin{figure}[t]
  \hfill
  \begin{tabular}{cc}
    \begin{overpic}[width=0.40\textwidth]{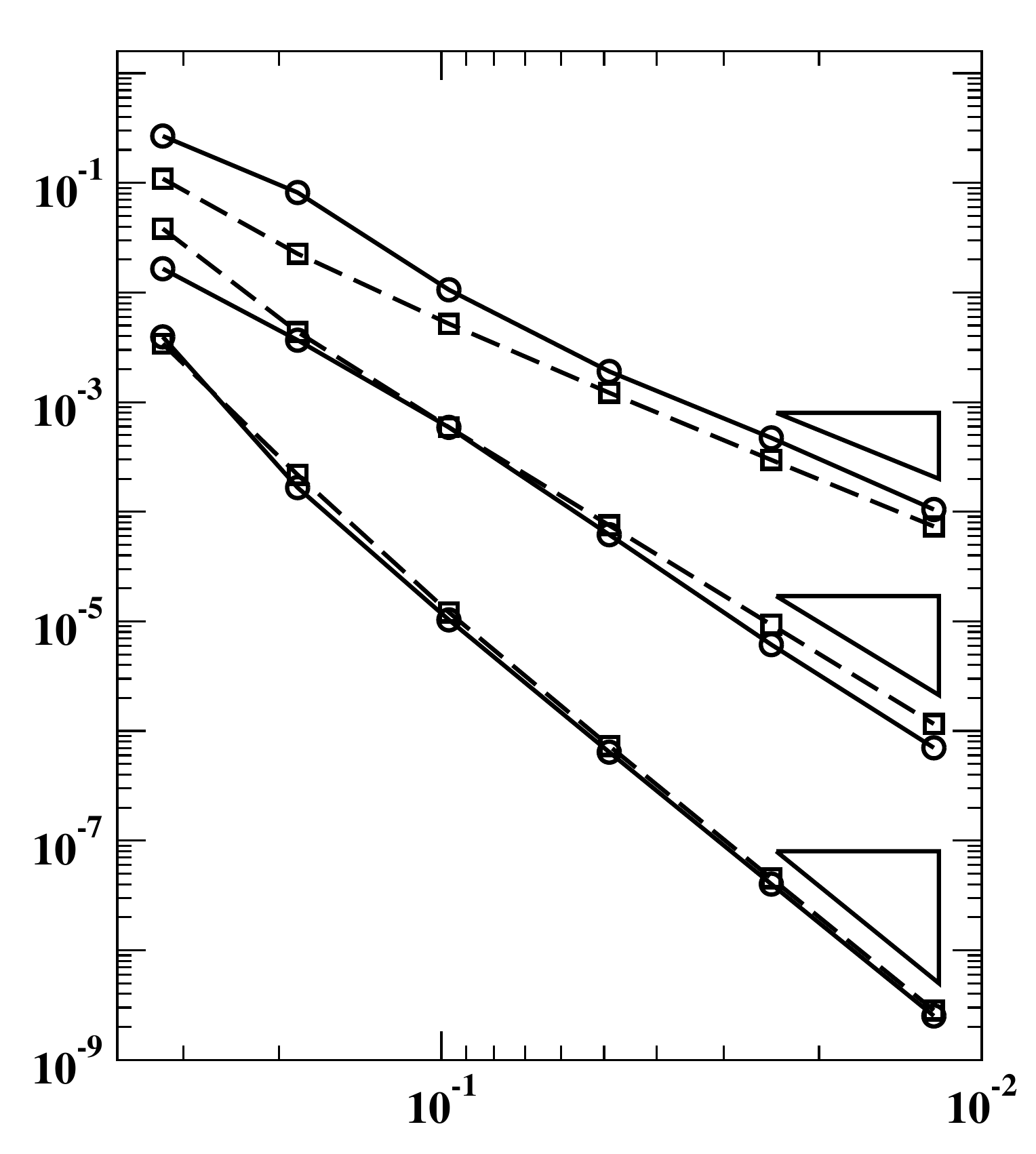}%{./figures/PDF-Md201/res-Md201-err_L2.pdf}
      \put(31,-5){\textbf{Mesh size $h$}}
      \put(-5,13){\begin{sideways}\textbf{ $L^2$ Approximation
            Error}\end{sideways}} \put(72,69){\textbf{2}}
      \put(72,52){\textbf{3}} \put(72,30){\textbf{4}}
    \end{overpic}
    &\quad
    \begin{overpic}[width=0.40\textwidth]{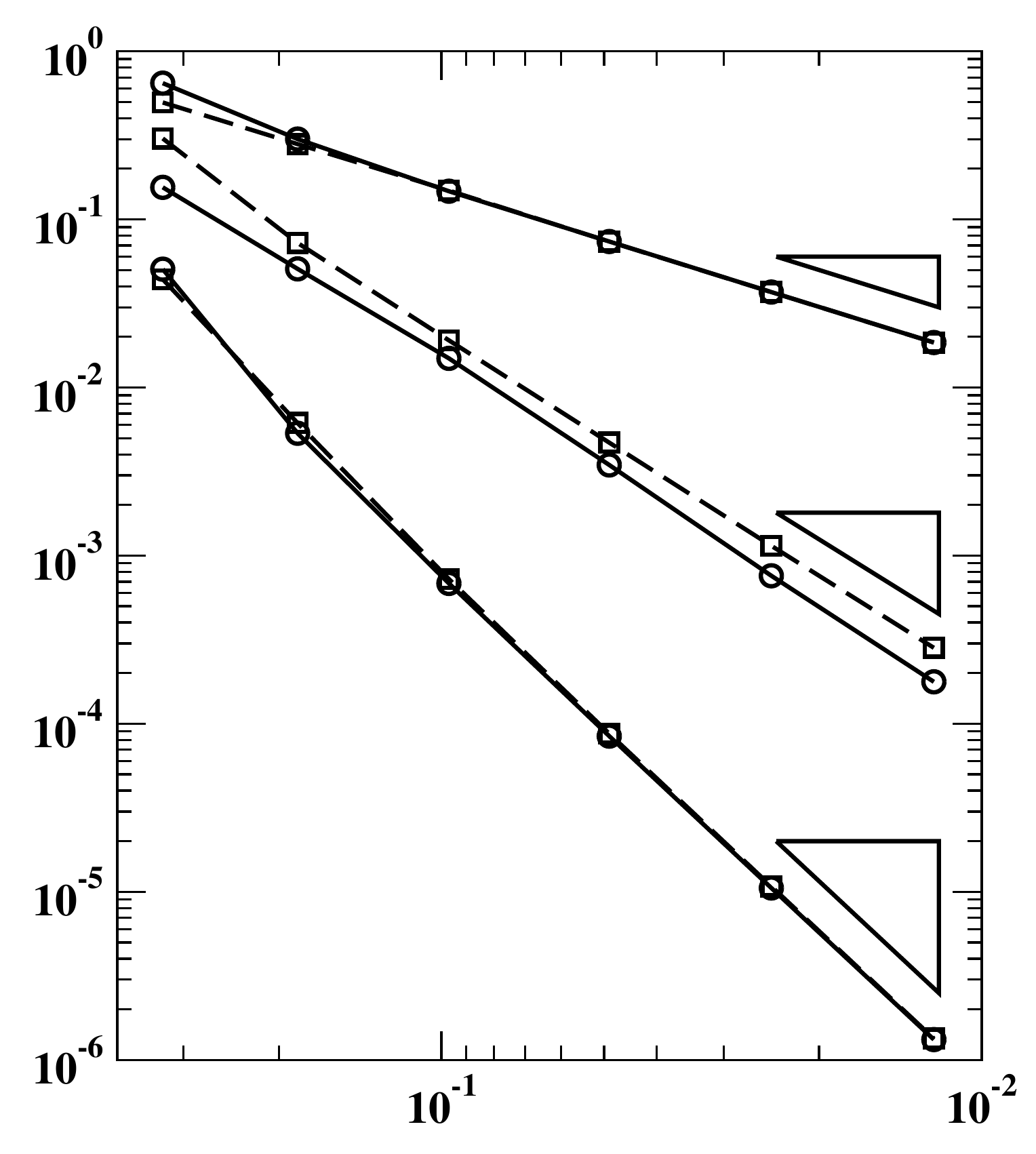}%{./figures/PDF-Md201/res-Md201-err_H1.pdf}
      \put(31,-5){\textbf{Mesh size $h$}}
      \put(-5,13){\begin{sideways}\textbf{ $H^1$ Approximation
            Error}\end{sideways}} \put(72,79){$\textbf{1}$}
      \put(72,54){$\textbf{2}$} \put(72,28){$\textbf{3}$}
    \end{overpic}
  \end{tabular}
  %% Mesh Md201
  \caption{Relative approximation errors obtained using the conforming VEM (dashed lines 
  labeled with squares) and the nonconforming VEM (solid lines labeled with circles)
  for $k=1,2,3$ (from top to bottom).
  Calculations are carried out using the remapped hexagonal meshes of 
  Figure~\ref{fig:Meshes}(b).
  Errors are measured in the $L^2$ norm (left panels) and $H^1$ norm
  (right panels), and plotted versus $h$.}
%% (top panels) and the number of degrees of freedom (right panels). }
\label{fig:Md201}
\vspace{-0.25cm}
\end{figure}

%% ================================================================================

\begin{figure}[t]
  \hfill
  \begin{tabular}{cc}
    \begin{overpic}[width=0.40\textwidth]{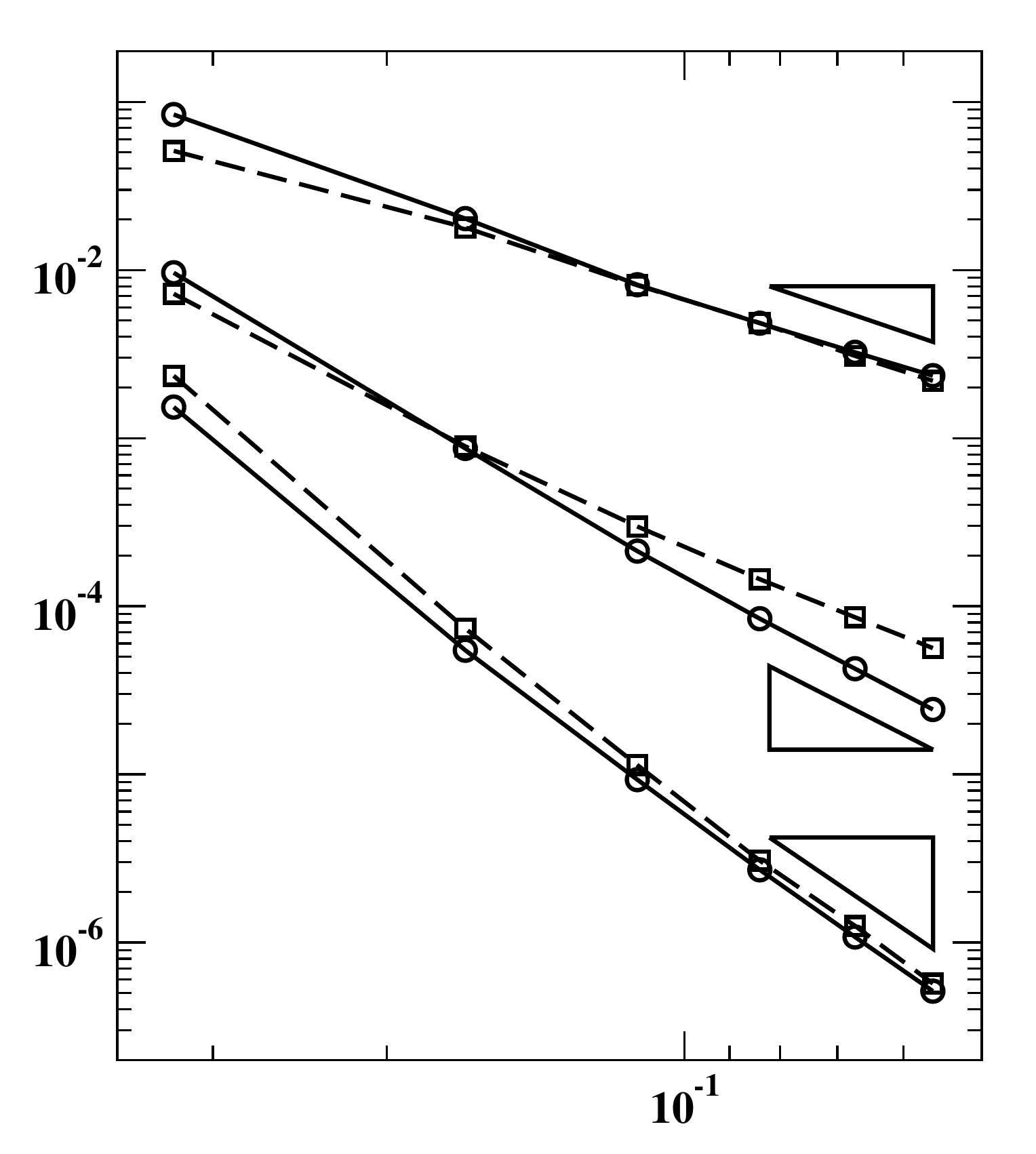}%{./figures/PDF-M103/res-M103-err_L2.pdf}
      \put(31,-5){\textbf{Mesh size $h$}}
      \put(-5,13){\begin{sideways}\textbf{ $L^2$ Approximation
            Error}\end{sideways}} \put(73,79){\textbf{2}}
      \put(73,56){\textbf{3}} \put(72,40){\textbf{4}}
    \end{overpic}
    &\quad
    \begin{overpic}[width=0.40\textwidth]{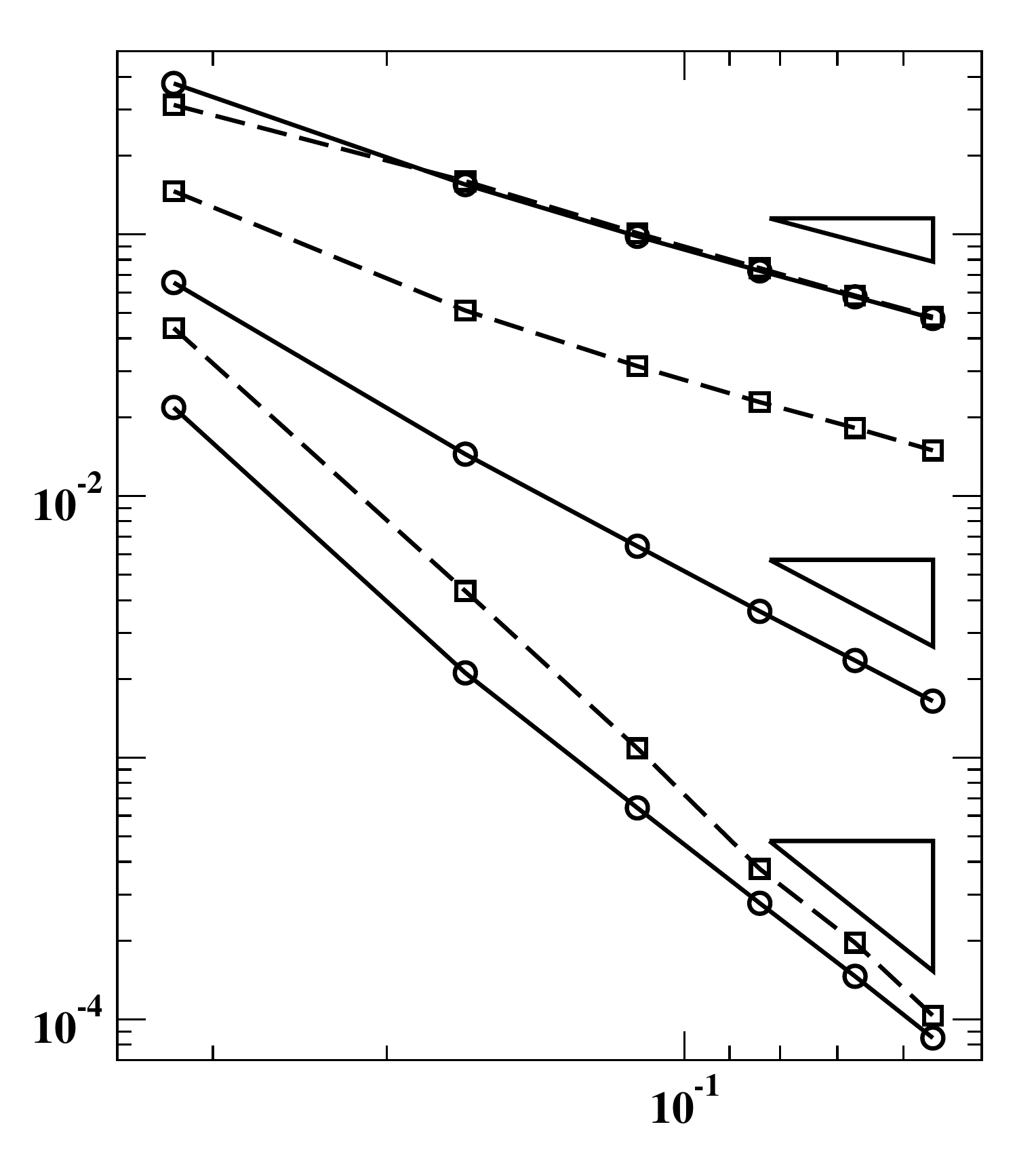}%{./figures/PDF-M103/res-M103-err_H1.pdf}
      \put(31,-5){\textbf{Mesh size $h$}}
      \put(-5,13){\begin{sideways}\textbf{ $H^1$ Approximation
            Error}\end{sideways}} \put(72,81){\textbf{1}}
      \put(72,62){\textbf{2}} \put(72,37){\textbf{3}}
    \end{overpic}
  \end{tabular}
  %% Mesh M103
  \caption{Relative approximation errors obtained using the conforming VEM (dashed lines 
  labeled with squares) and the nonconforming VEM (solid lines labeled with circles)
  for $k=1,2,3$ (from top to bottom).
  Calculations are carried out using the highly distorted quadrilateral meshes of 
  Figure~\ref{fig:Meshes}(c).
  Errors are measured in the $L^2$ norm (left panels) and $H^1$ norm
  (right panels), and plotted versus $h$.}
%% (top panels) and the number of degrees of freedom (right panels). }
\label{fig:M103}
\vspace{-0.25cm}
\end{figure}

%% ================================================================================

\begin{figure}[t]
  \hfill
  \begin{tabular}{cc}
    \begin{overpic}[width=0.40\textwidth]{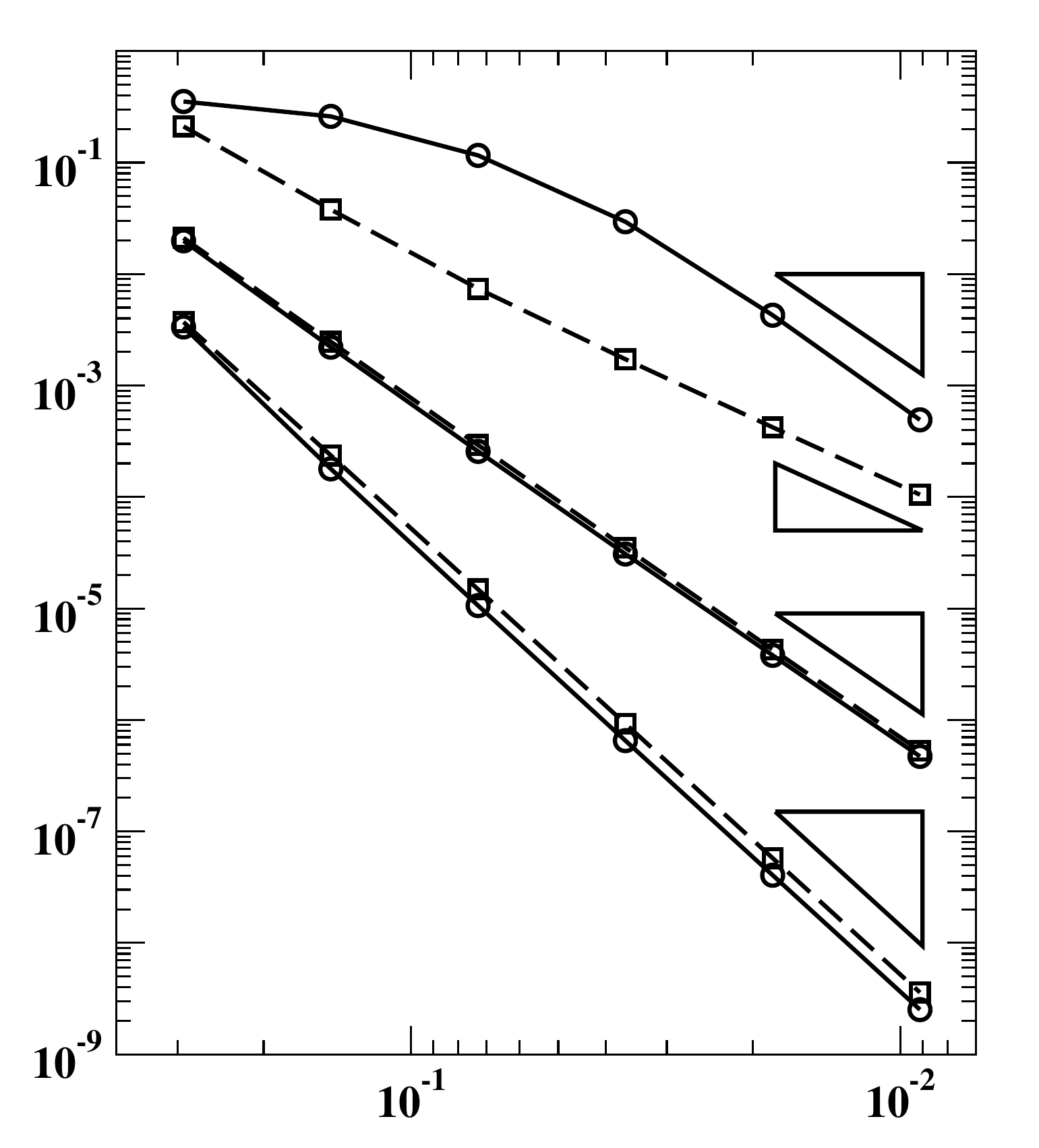}%{./figures/PDF-M901/res-M901-err_L2.pdf}
      \put(31,-5){\textbf{Mesh size $h$}}
      \put(-5,13){\begin{sideways}\textbf{ $L^2$ Approximation
            Error}\end{sideways}} \put(72,67){\textbf{2}}
      \put(72,50){\textbf{3}} \put(72,32){\textbf{4}}
    \end{overpic}
    &\quad
    \begin{overpic}[width=0.40\textwidth]{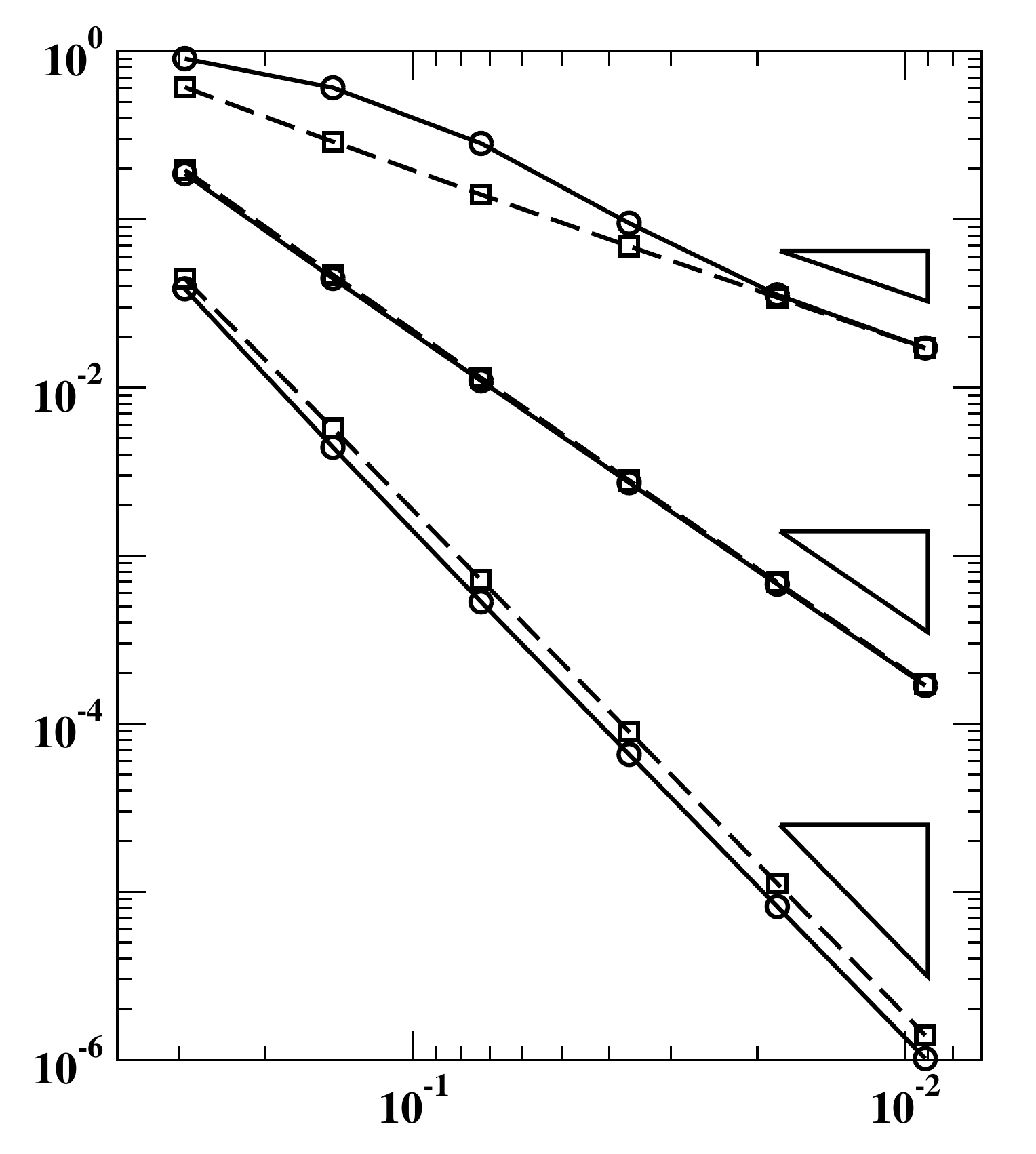}%{./figures/PDF-M901/res-M901-err_H1.pdf}
      \put(31,-5){\textbf{Mesh size $h$}}
      \put(-5,13){\begin{sideways}\textbf{ $H^1$ Approximation
            Error}\end{sideways}} \put(72,79){\textbf{1}}
      \put(72,55){\textbf{2}} \put(72,28){\textbf{3}}
    \end{overpic}
  \end{tabular}
  %% Mesh M901
  \caption{Relative approximation errors obtained using the conforming VEM (dashed lines 
  labeled with squares) and the nonconforming VEM (solid lines labeled with circles)
  for $k=1,2,3$ (from top to bottom).
  Calculations are carried out using the non convex meshes of 
  Figure~\ref{fig:Meshes}(d).
  Errors are measured in the $L^2$ norm (left panels) and $H^1$ norm
  (right panels), and plotted versus $h$.}
%% (top panels) and the number of degrees of freedom (right panels). }
\label{fig:M901}
\vspace{-0.25cm}
\end{figure}

%%%%%%%%%%%%%%
%% GRAFICI PIATTI
\begin{figure}
  \centering
  \begin{tabular}{cc}
    \hspace{0.2cm}\textbf{Conforming VEM} & \hspace{0.6cm}\textbf{Nonconforming VEM}\\[0.25em]
    \begin{overpic}[width=0.40\textwidth]{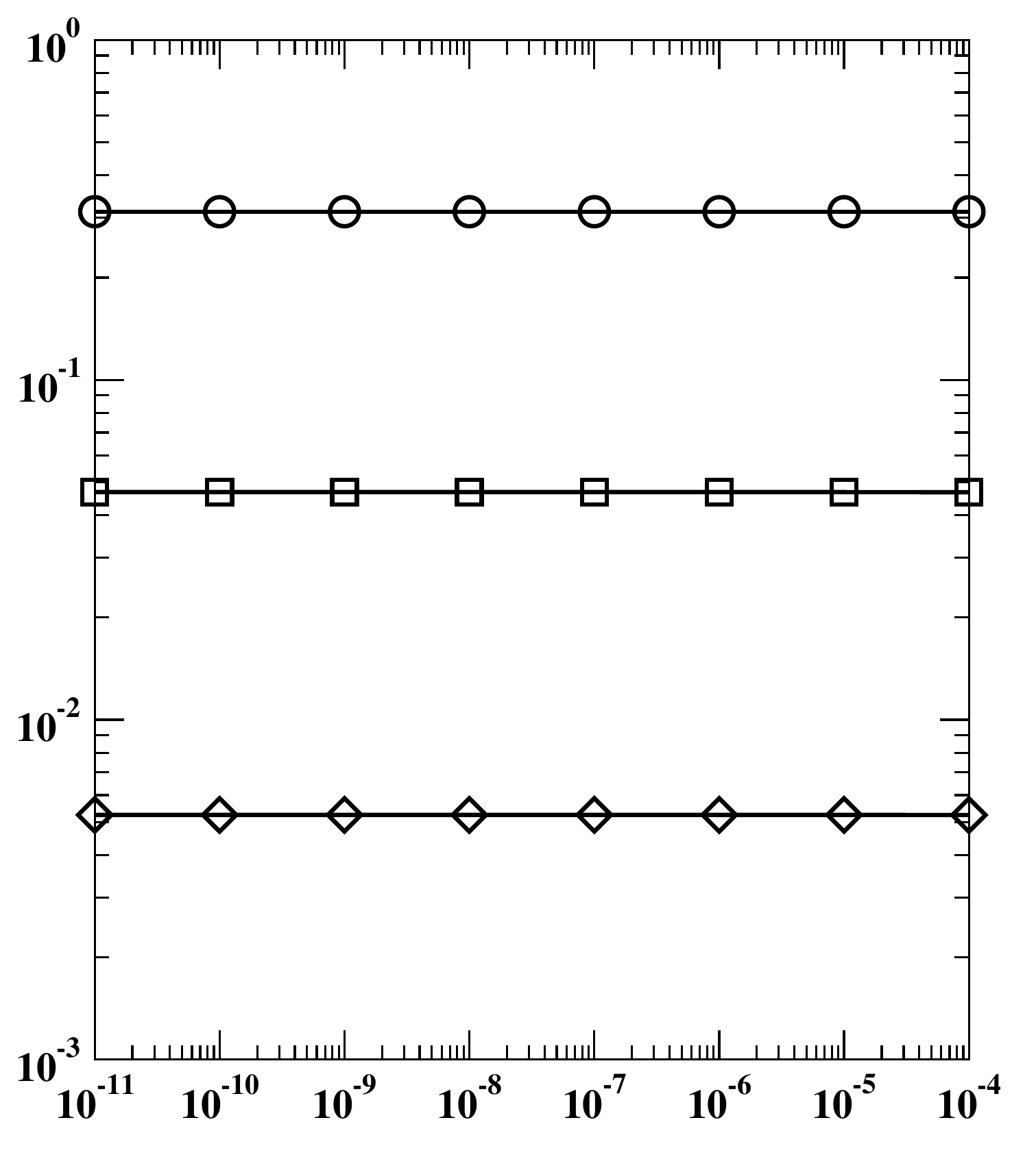}%{./figures/Appendix-vs-Alpha-VarBeta/CF-VEM-PDF/CF-VEM-H1-error-vs-alpha-RegularHexa.pdf}
      \put(15,84){$\mathbf{k=1}$} \put(15,60){$\mathbf{k=2}$}
      \put(15,32){$\mathbf{k=3}$} \put(29,-5){\textbf{Coefficient
          $\alpha$}} \put(-5,17){\begin{sideways}\textbf{
            $H^1$Approximation Error }\end{sideways}}
    \end{overpic}
    &\quad
    \begin{overpic}[width=0.40\textwidth]{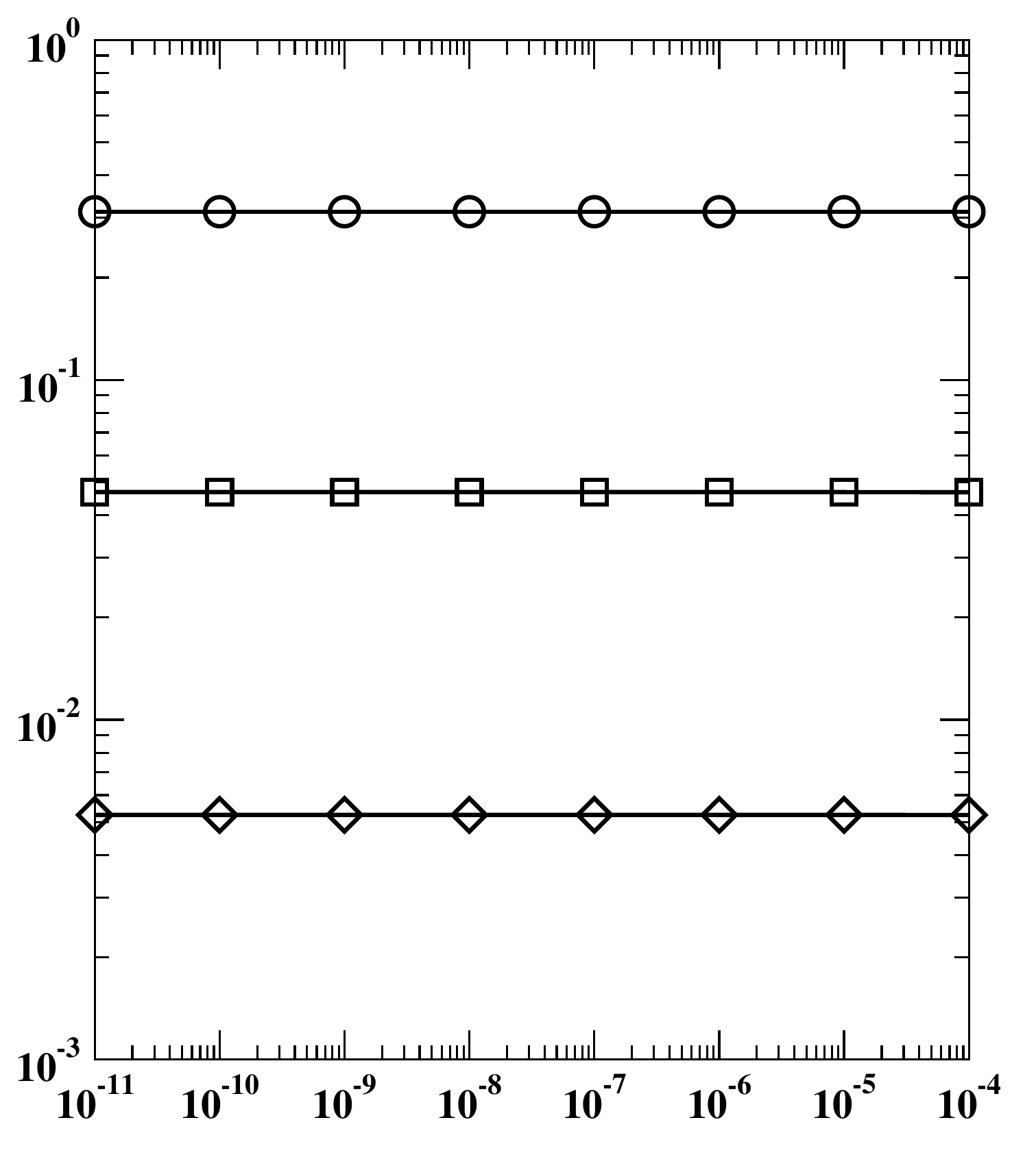}%{./figures/Appendix-vs-Alpha-VarBeta/NC-VEM-PDF/NC-VEM-H1-error-vs-alpha-RegularHexa.pdf}
    \put(15,84.0){$\mathbf{k=1}$}
    \put(15,60.5){$\mathbf{k=2}$}
    \put(15,32.5){$\mathbf{k=3}$}
    \put(29,-5){\textbf{Coefficient $\alpha$}}
    \put(-5,17){\begin{sideways}\textbf{$H^1$ Approximation Error}\end{sideways}}
    \end{overpic}
    \\[2em]
    \multicolumn{2}{c}{\begin{large}\textbf{Mesh of regular hexagons (a)}\end{large}}
    \\[1.5em]
    \hspace{0.2cm}\textbf{Conforming VEM} & \hspace{0.6cm}\textbf{Nonconforming VEM}\\[0.25em]
    \begin{overpic}[width=0.40\textwidth]{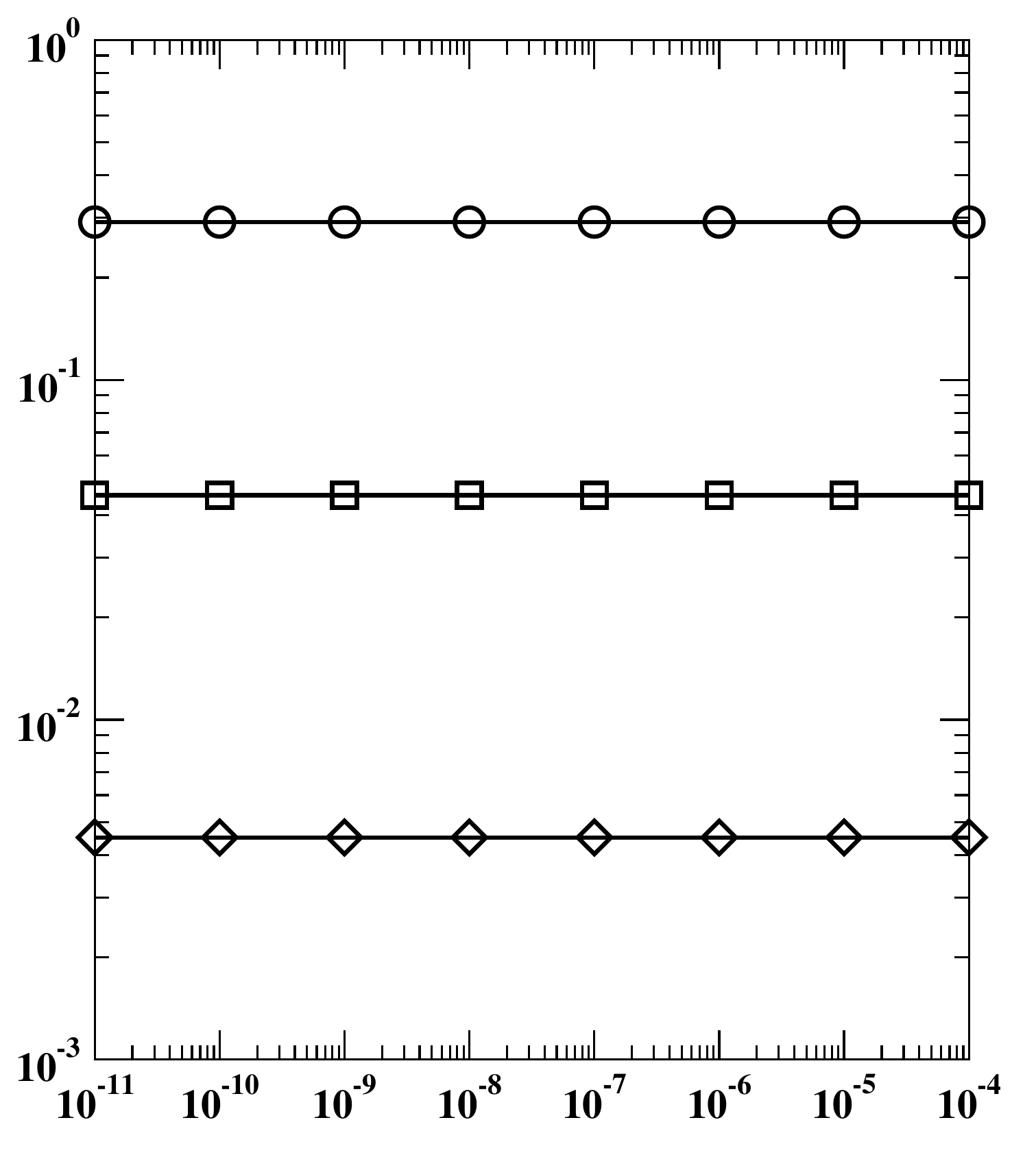}%{./figures/Appendix-vs-Alpha-VarBeta/CF-VEM-PDF/CF-VEM-H1-error-vs-alpha-OctaNonConvex.pdf}
      \put(15,83){$\mathbf{k=1}$} \put(15,60){$\mathbf{k=2}$}
      \put(15,30){$\mathbf{k=3}$} \put(29,-5){\textbf{Coefficient
          $\alpha$}} \put(-5,17){\begin{sideways}\textbf{ $H^1$
            Approximation Error }\end{sideways}}
    \end{overpic}
    &\quad
    \begin{overpic}[width=0.40\textwidth]{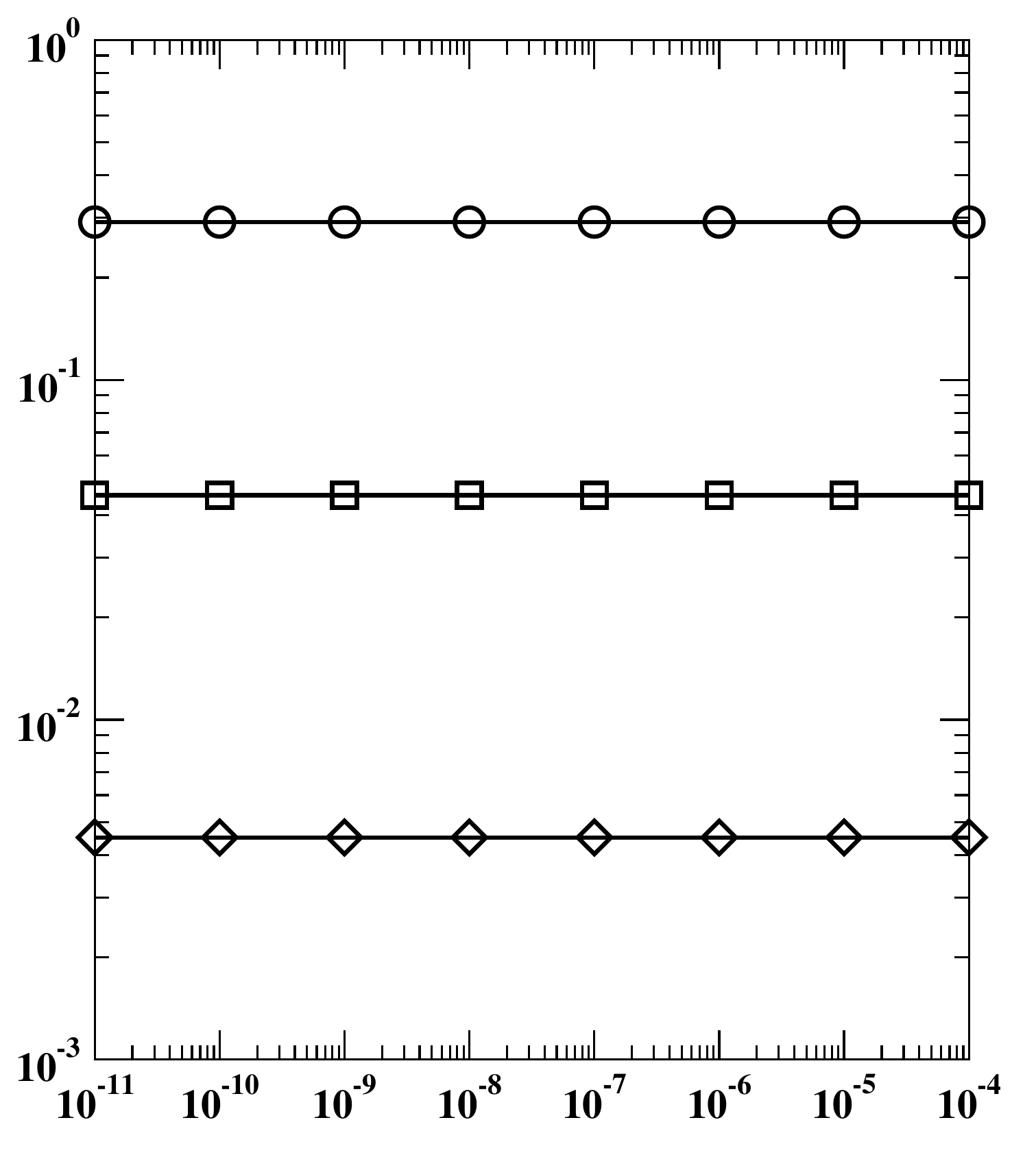}%{./figures/Appendix-vs-Alpha-VarBeta/NC-VEM-PDF/NC-VEM-H1-error-vs-alpha-OctaNonConvex.pdf}
    \put(15,83){$\mathbf{k=1}$}
    \put(15,60){$\mathbf{k=2}$}
    \put(15,30){$\mathbf{k=3}$}
    \put(29,-5){\textbf{Coefficient $\alpha$}}
    \put(-5,17){\begin{sideways}\textbf{ $H^1$ Approximation Error }\end{sideways}}
    \end{overpic}
    \\[2.em]
    \multicolumn{2}{c}{\begin{large}\textbf{Mesh of regular non convex cells (d)}\end{large}}
    \\
  \end{tabular}
  \caption{$H^1$ relative approximation error versus the viscous
    coefficient $\alpha\in[10^{-11},10^{-4}]$ using the first refined
    mesh of mesh families $(a)$ (top panels) and $(d)$ (bottom panels)
    for the test case with $\gamma(x,y)=0$. The problem is solved by
    applying the conforming VEM (left panel) and nonconforming VEM
    (right panel) of degree $k=1$ (circles), $k=2$ (squares), $k=3$
    (diamonds).}
  \label{fig:H1error_smallK}
\end{figure}

%%%%%%%%%%%%%%%%%%%%%%%%%%%%%%%%%%%%%%%%%%%%%%%%%%%%%%%%%%%%%%%%%%%%%%%%%%%%%%%%%%%%
%%
%% boundary layers calculation
%%
%%%%%%%%%%%%%%%%%%%%%%%%%%%%%%%%%%%%%%%%%%%%%%%%%%%%%%%%%%%%%%%%%%%%%%%%%%%%%%%%%%%%

\begin{figure}[t]
  \centering
  \begin{overpic}[scale=0.27]{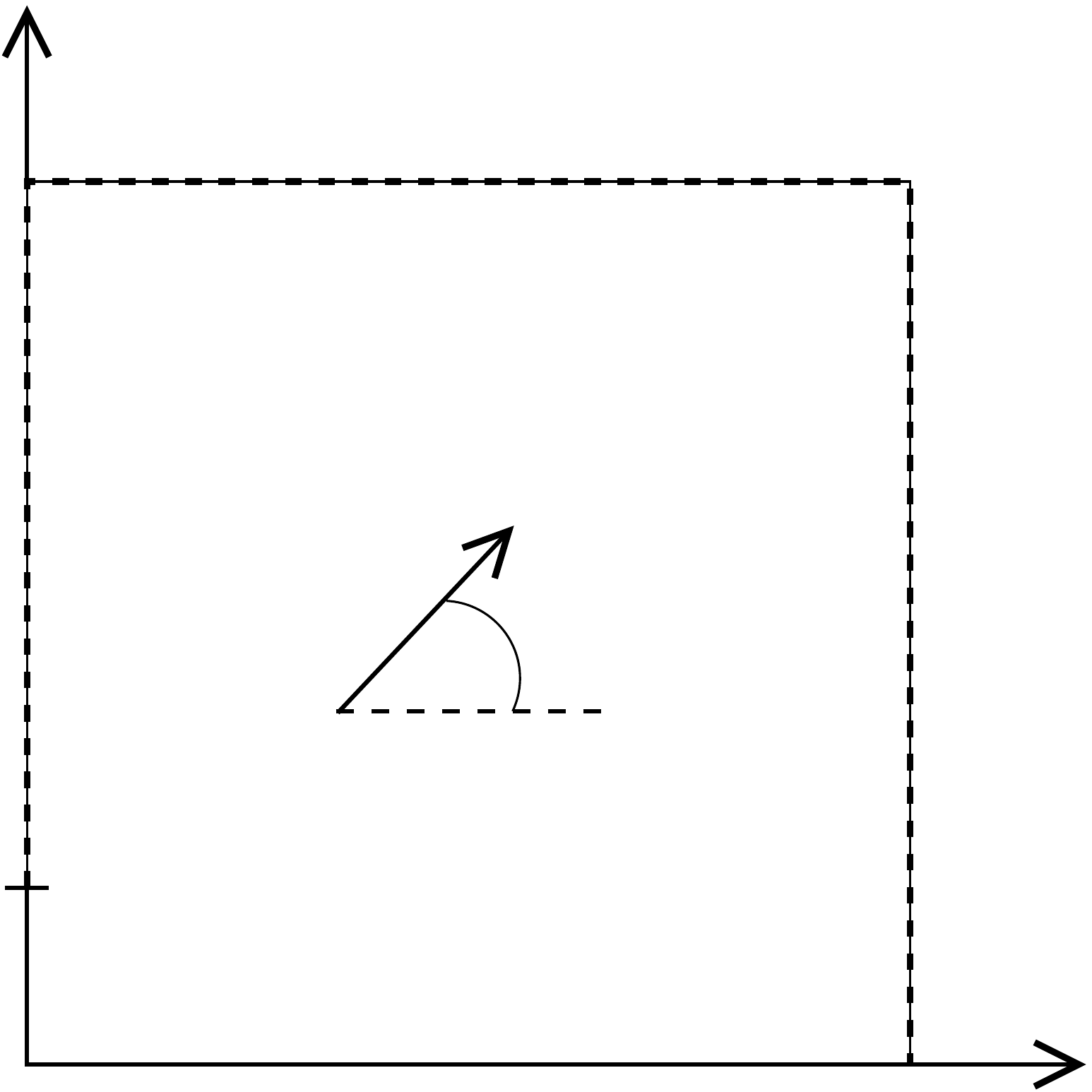}%{./figures/test2.pdf}
    \put(95,-5){$\mathbf{x}$} \put(-5,95){$\mathbf{y}$}
    %% ------------------------
    \put(76,-4){$\mathbf{x\!\!=\!\!1}$}
    \put(-14,82){$\mathbf{y\!\!=\!\!1}$}
    \put(7,17){$\mathbf{y\!\!=\!\!0.2}$}
    %% ------------------------
    \put(35,-5){$\mathbf{u=1}$} \put(35,87){$\mathbf{u=0}$}
    %% ------------------------
    \put(-20,10){$\mathbf{u=1}$} \put(-20,48){$\mathbf{u=0}$} \put(88,
    42){$\mathbf{u=0}$}
    %% ------------------------
    \put(30,54){\textbf{velocity}} \put(50,41){$\bm{\theta}$}
  \end{overpic}
  \vspace{0.25cm}
  \caption{Test~2: domain and boundary conditions.}
  \label{fig:test-2:domain}
  \vspace{-0.25cm}
\end{figure}

\begin{figure}
  \centering
  \begin{tabular}{cc}
    \includegraphics[scale=0.14]{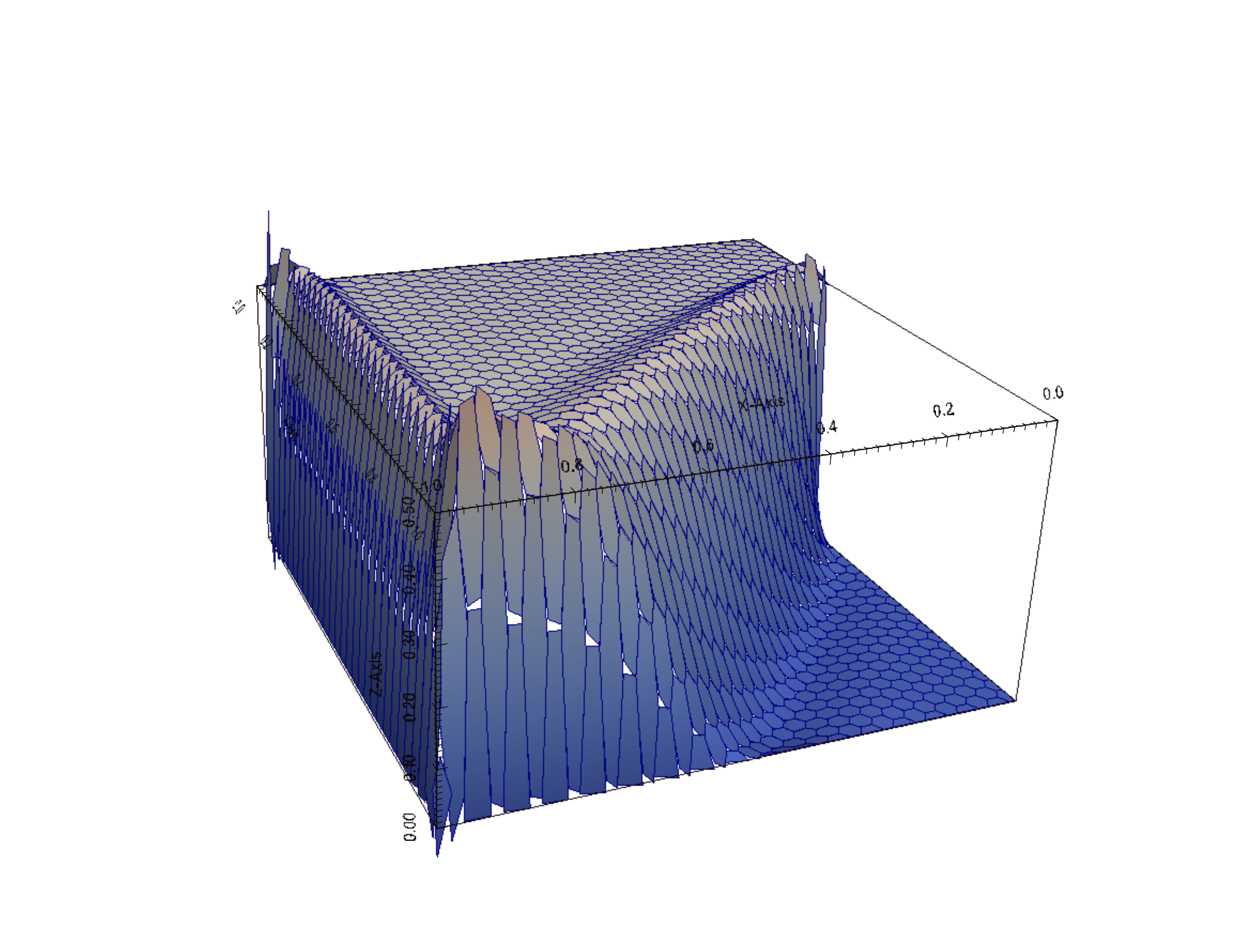}&           %{./figures/PDF-VTK/plot_06.pdf} &
    \includegraphics[scale=0.14]{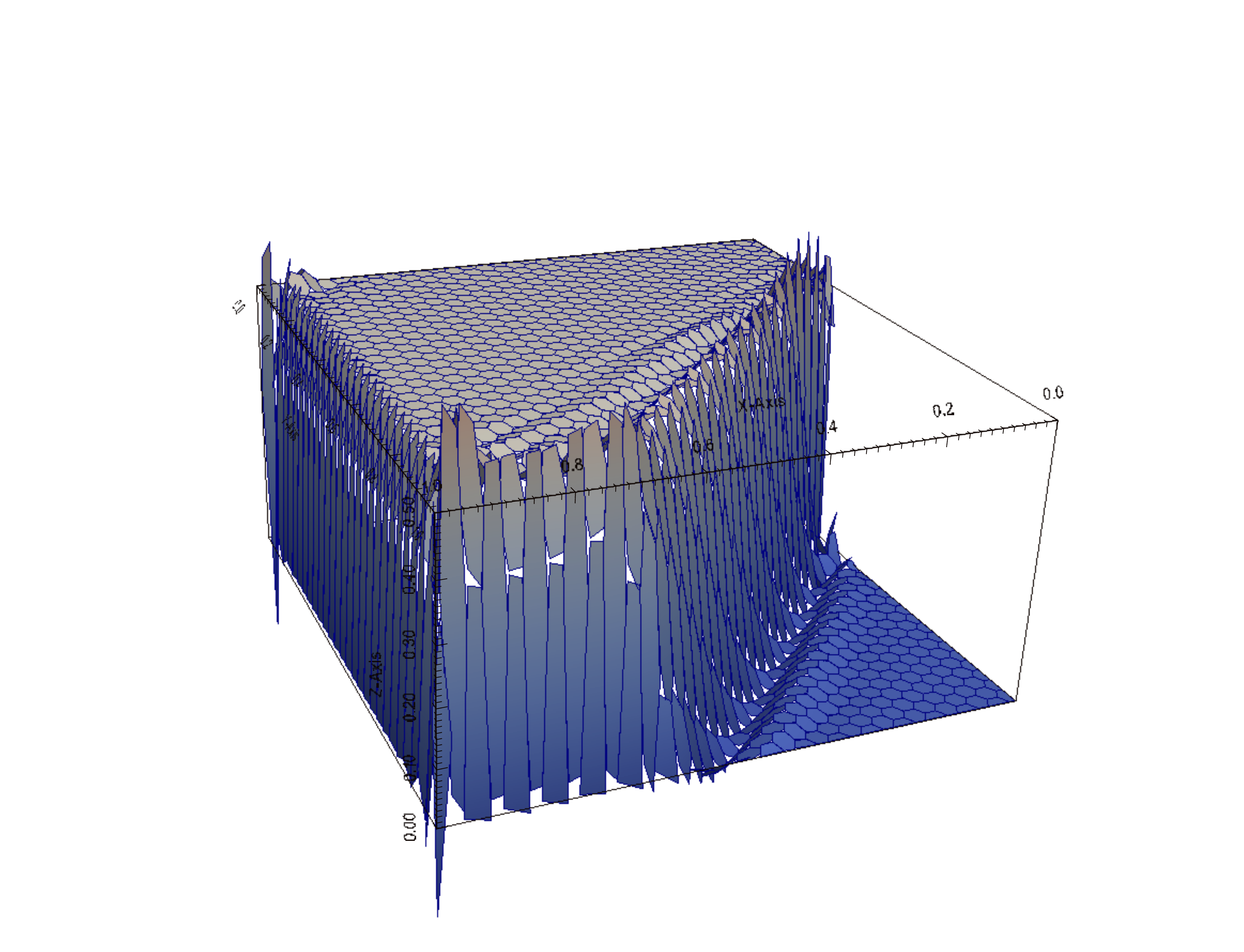}\\[0.175em] %{./figures/PDF-VTK/plot_09.pdf} \\[0.175em]
    \multicolumn{2}{c}{$\mathbf{k=1}$}\\[1em]
    %% ---------------------------------------------------------
    \includegraphics[scale=0.14]{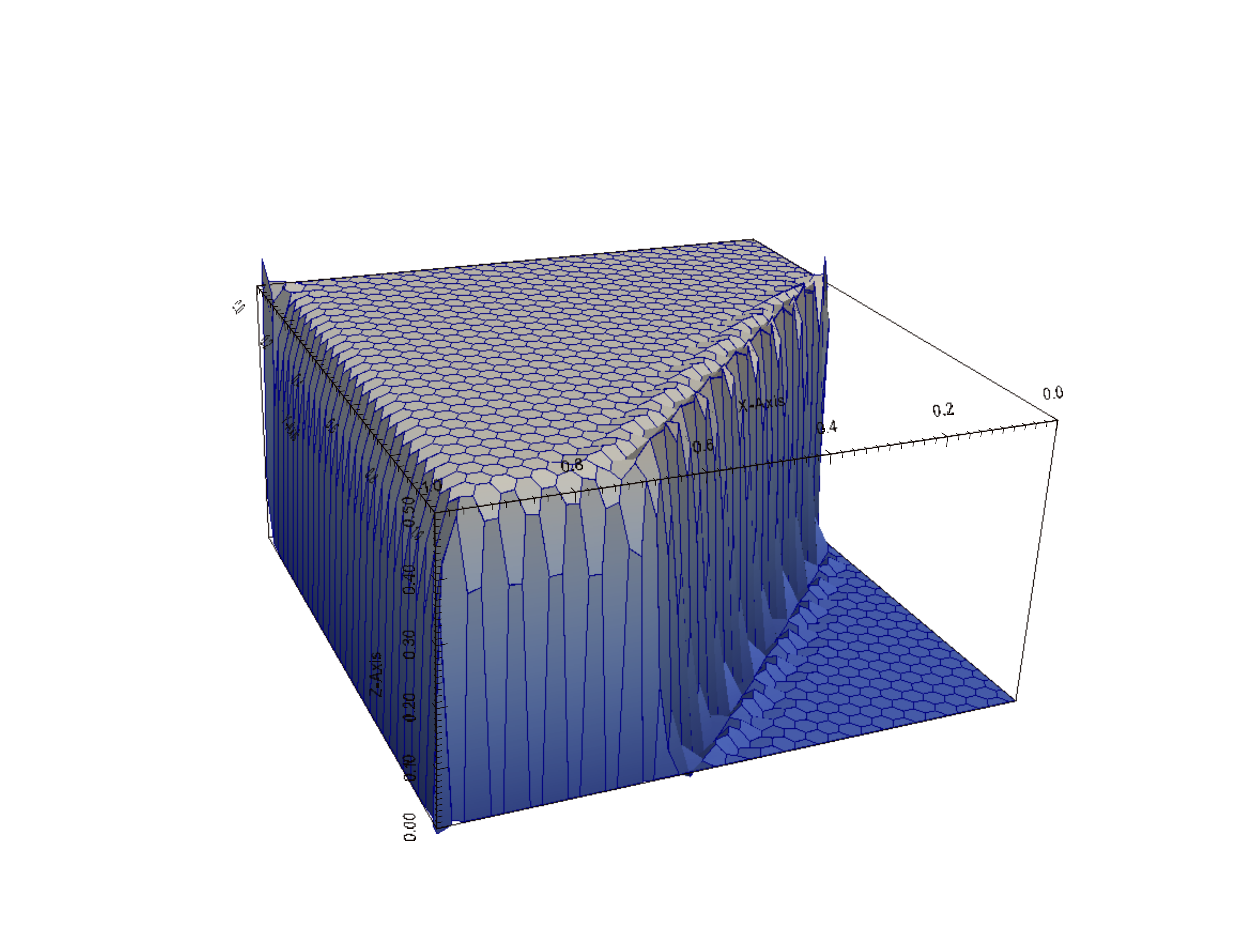}&           %{./figures/PDF-VTK/plot_08.pdf} &
    \includegraphics[scale=0.14]{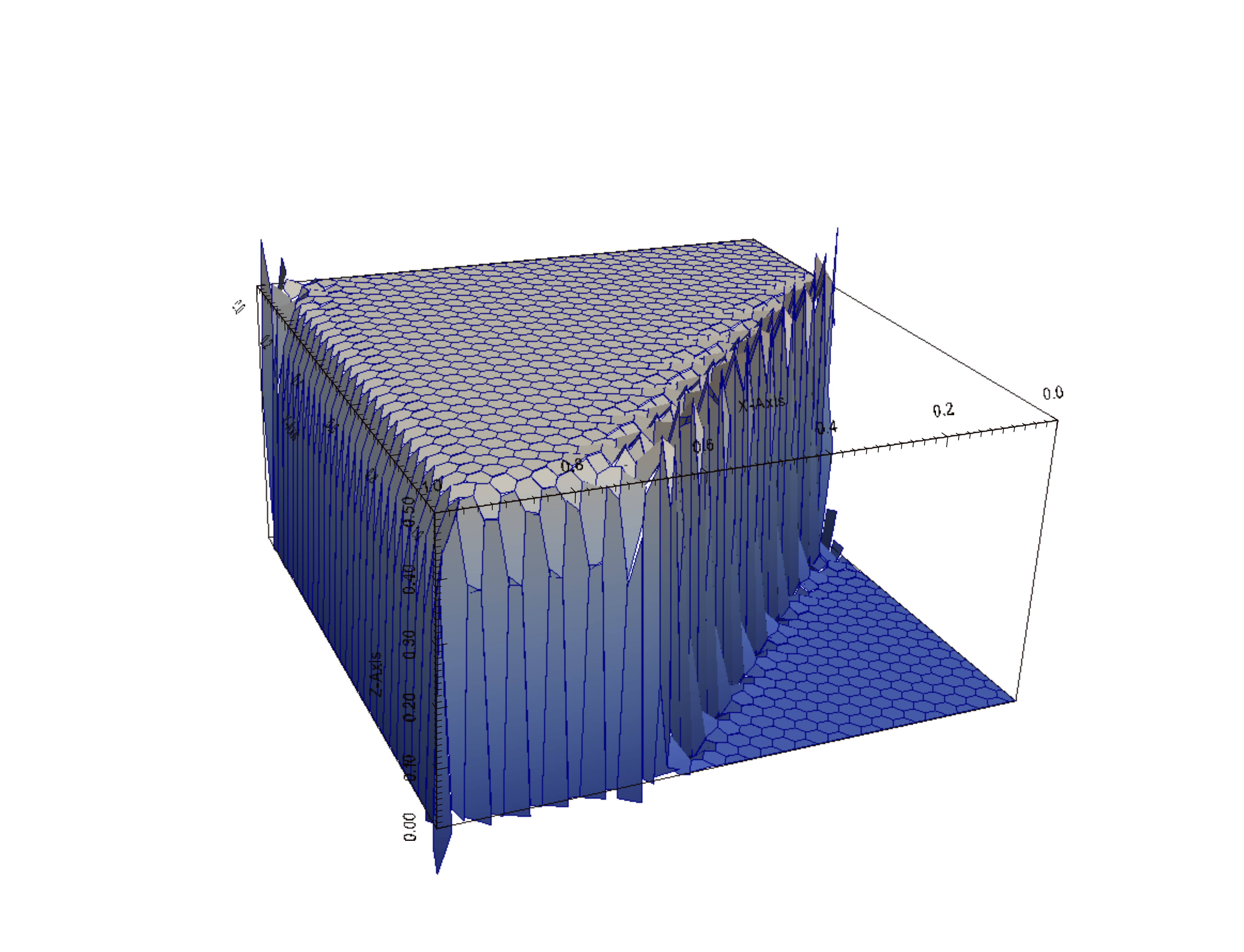}\\[0.175em] %{./figures/PDF-VTK/plot_11.pdf} \\[0.175em]
    \multicolumn{2}{c}{$\mathbf{k=3}$}%\\[1em]
  \end{tabular}
  \caption{Test Case~2: conforming VEM (left panels) and nonconforming
    VEM (right panels) for $k=1$ and $k=3$ and using a regular
    hexagonal mesh (third refinement). }
  \label{fig:test-case-2b}
  \centering
  \begin{tabular}{cc}
    \includegraphics[scale=0.14]{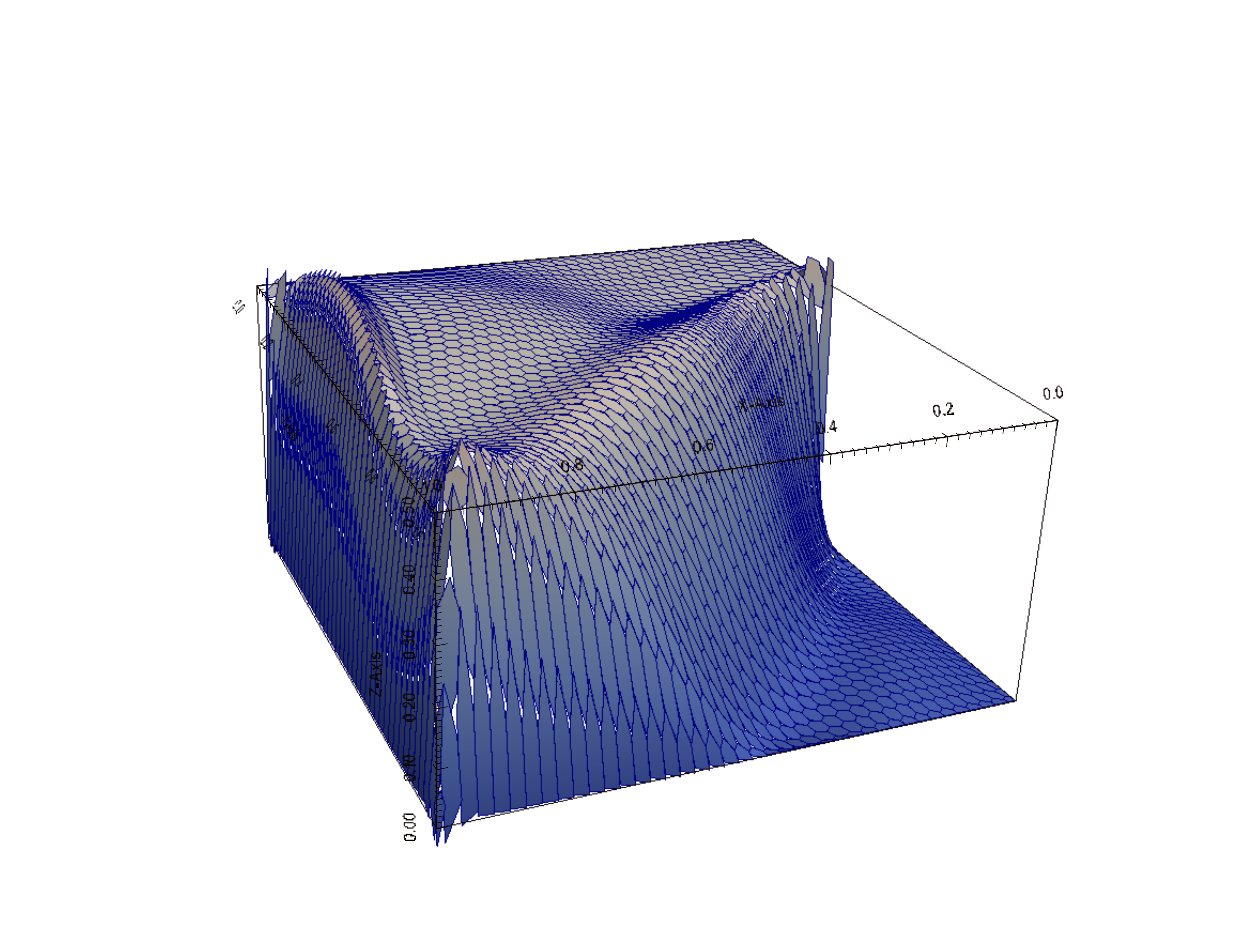}&           %{./figures/PDF-VTK/plot_12.pdf} &
    \includegraphics[scale=0.14]{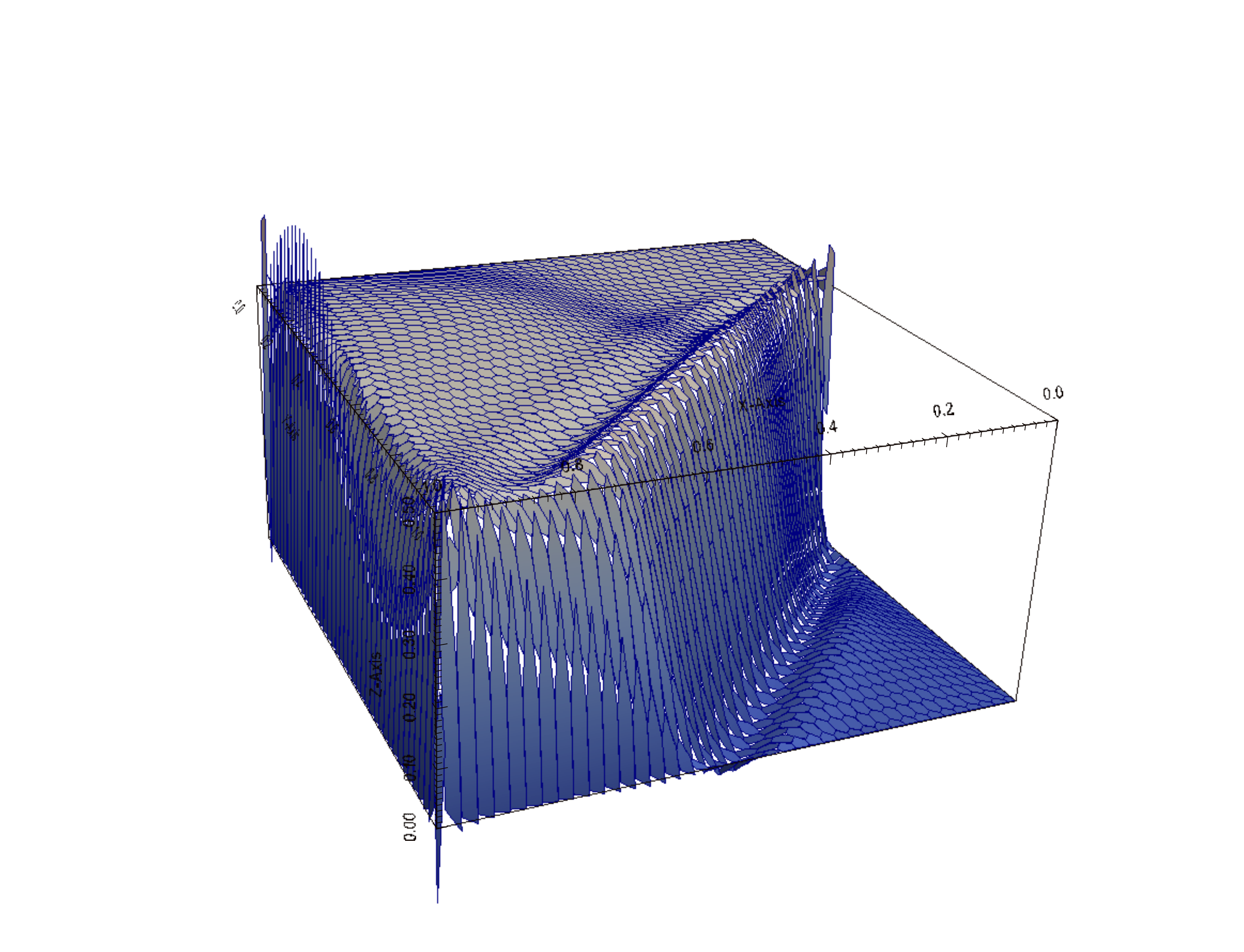}\\[0.175em] %{./figures/PDF-VTK/plot_15.pdf} \\[0.175em]
    \multicolumn{2}{c}{$\mathbf{k=1}$}\\[1em]
    %% ---------------------------------------------------------
    \includegraphics[scale=0.14]{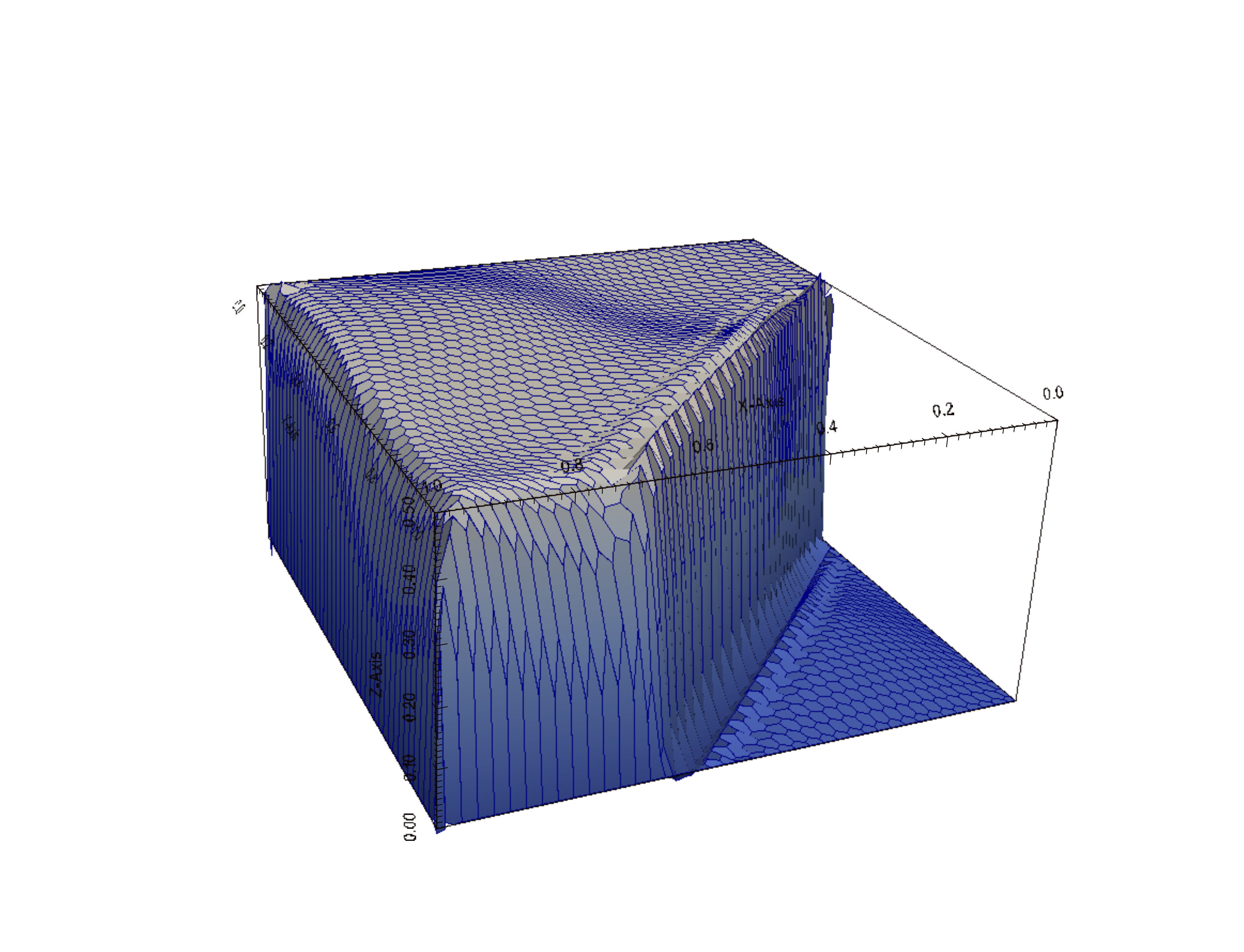}&           %{./figures/PDF-VTK/plot_14.pdf} &
    \includegraphics[scale=0.14]{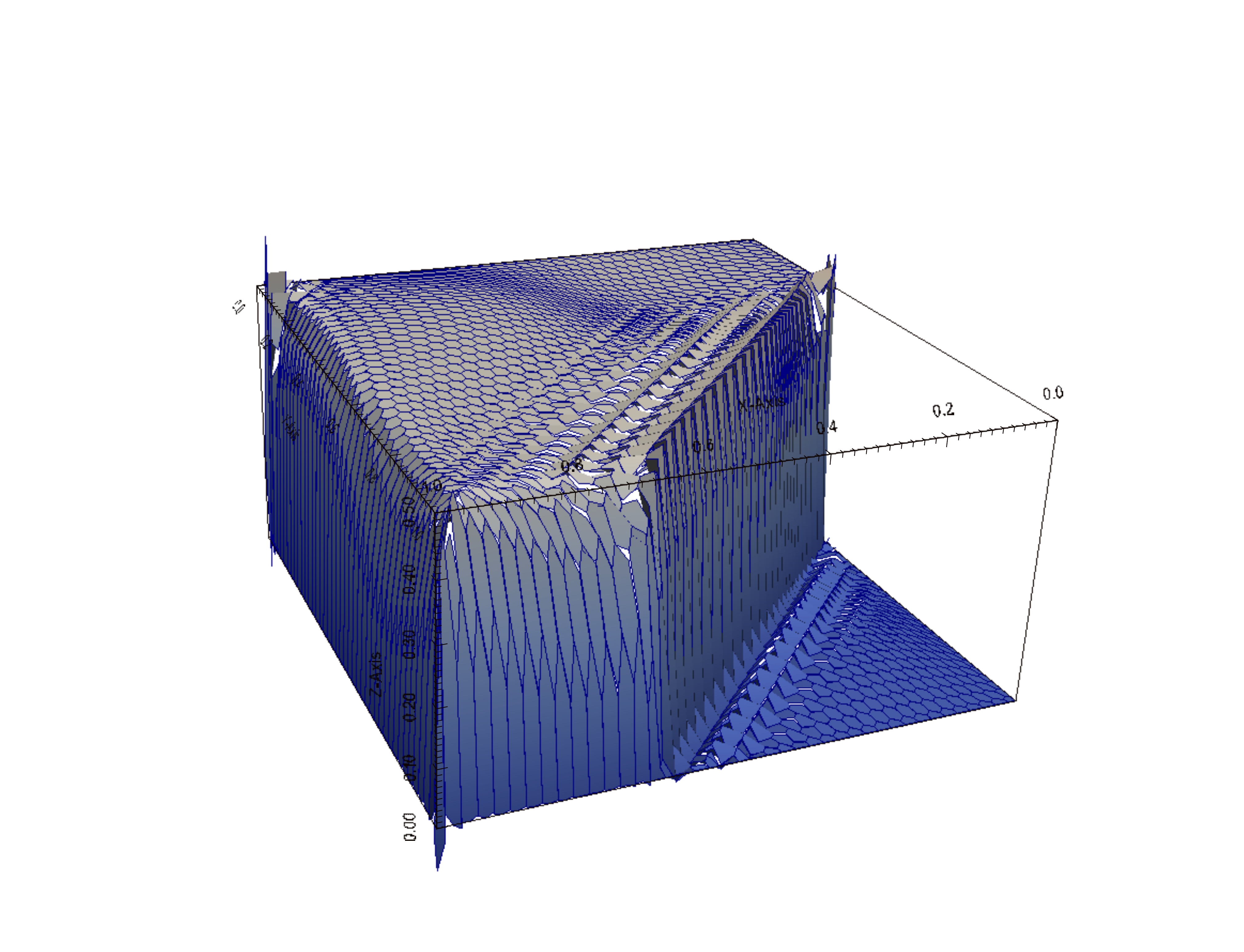}\\[0.175em] %{./figures/PDF-VTK/plot_17.pdf} \\[0.175em]
    \multicolumn{2}{c}{$\mathbf{k=3}$}%\\[1em]
  \end{tabular}
  \caption{Test Case~2: conforming VEM (left panels) and nonconforming
    VEM (right panels) for $k=1$ and $k=3$ and using a remapped
    hexagonal mesh (third refinement). }
  \label{fig:test-case-2c}
  \vspace{-0.25cm}
\end{figure}

\section{Conclusions}
\label{sec:conclusions}

In this paper, we proposed a nonconforming VEM for the
advection-diffusion-reaction problem in the convection-dominated
regime.
Due to the strong convective field with respect to the diffusion term,
we introduced the SUPG stabilization by extending to the nonconforming
VEM the stabilization technique proposed
in~\cite{Benedetto-Berrone-Borio-Pieraccini-Scialo:2016a}.
The stabilization included in the virtual element formulation is a
natural extension of the classical SUPG stabilization for the standard
FEM.
To ensure coercivity of the discrete operators, we modify the SUPG
stabilization by introducing a VEM stabilization of the SUPG
stabilization term.
Optimal convergence rates are obtained from the convergence analysis
under proper assumptions on the regularity of problem coefficients,
the meshes, and the exact solution.
The numerical results confirm the behaviour of the VEM that is
expected from the theory and the stabilizing effect of the additional
SUPG term provides stable discrete solutions even for very large
P\'eclet numbers in the order of $10^6$.

%\bibliographystyle{elsarticle-num}
%\bibliography{Bib/pub-MFD,Bib/pub-vem,Bib/pub-poly,Bib/pub-supg}

\end{document}